\documentclass[10pt, letterpaper]{amsart}

\usepackage {geometry,kantlipsum}
\usepackage{enumitem}
\setlist[itemize]{leftmargin=*}

\usepackage{bm}
\usepackage{latexsym, amsmath, amstext, amssymb, amsfonts, amscd, bm, array, multirow, amsbsy, mathrsfs}
\usepackage{amsthm}
\usepackage{t1enc}
\usepackage[mathscr]{eucal}
\usepackage{indentfirst}
\usepackage{tensor}
\usepackage{pb-diagram}
\usepackage{graphicx}
\usepackage{fancyhdr}
\usepackage{fancybox}
\usepackage{enumerate}
\usepackage{color}
\usepackage[all]{xy}
\usepackage{hyperref}
\usepackage{tikz}
\usepackage{xparse}
\hypersetup{colorlinks=false,pdfborderstyle={/S/U/W 0}}
\usetikzlibrary{matrix}

\usepackage{multirow}

\newcommand{\arxiv}[1]{{\tt
		\href{http://www.arXiv.org/abs/#1}{arXiv:#1}}}
		
\newcommand\varpm{\mathbin{\vcenter{\hbox{%
  \oalign{\hfil$\scriptstyle+$\hfil\cr
          \noalign{\kern-.3ex}
          $\scriptscriptstyle({-})$\cr}%
}}}}		
		
\newcommand\varmp{\mathbin{\vcenter{\hbox{%
  \oalign{\hfil$\scriptstyle-$\hfil\cr
          \noalign{\kern-.3ex}
          $\scriptscriptstyle({+})$\cr}%
}}}}

\usepackage{url}
\usepackage[sort&compress,comma]{natbib}
\bibpunct{[}{]}{,}{n}{}{,}
\theoremstyle{plain}
\newtheorem{thm}{Theorem}[section]

\newtheorem{prop}[thm]{Proposition}

\newtheorem{lemma}[thm]{Lemma}

\theoremstyle{definition}
\newtheorem{definition}[thm]{Definition}

\theoremstyle{remark}
\newtheorem{remark}[thm]{Remark}
\newtheorem{ep}[thm]{Example}

\newcommand{\Sym}{\mathrm{Sym}}

\newcommand{\End}{\mathrm{End}}

\newcommand{\g}{{\rm g}}

\newcommand{\eqdef}{\stackrel{{\rm def.}}{=}}

\newcommand{\Conf}{\mathrm{Conf}}
\newcommand{\Sol}{\mathrm{Sol}}

\DeclareFontFamily{U}{rsf}{}
\DeclareFontShape{U}{rsf}{m}{n}{<5> <6> rsfs5 <7> <8> <9> rsfs7 <10-> rsfs10}{}
\DeclareMathAlphabet\Scr{U}{rsf}{m}{n}

\def\s{\mathbb{s}}

\def\H{{\rm H}}

\def\s{\mathbb{s}}

\def\Sl{\mathrm{Sl}}

\def\dd{\mathrm{d}}

\newcommand{\be}{\begin{equation*}}
\newcommand{\ee}{\end{equation*}}
\newcommand{\ben}{\begin{equation}}
\newcommand{\een}{\end{equation}}
\newcommand{\beqa}{\begin{eqnarray*}}
	\newcommand{\eeqa}{\end{eqnarray*}}
\newcommand{\beqan}{\begin{eqnarray}}
\newcommand{\eeqan}{\end{eqnarray}}

\newcommand{\Tr}{\mathrm{Tr}}
\newcommand{\tr}{\mathrm{tr}}

\def\cK{\mathcal{K}}

\def\U{\mathrm{U}}

\def\cD{\mathcal{D}}

\def\cN{\mathcal{N}}

\def\SU{\mathrm{SU}}

\def\G_2{\mathrm{G_2}}

\def\cL{\mathcal{L}}

\def\s{\mathfrak{s}}

\def\Im{\mathrm{Im}}

\def\G{\mathrm{G}}

\def\ric{\text{Ric}}
\newcolumntype{P}[1]{>{\centering\arraybackslash}p{#1}}

\usepackage{enumitem}
\setlist[itemize]{leftmargin=*}
\setlist[enumerate]{leftmargin=*}
\begin{document}

\title[Contact metric three manifolds in six-dimensional supergravity]{Contact metric three manifolds and Lorentzian geometry with torsion in six-dimensional supergravity}

\author[\'Angel Murcia]{\'Angel Murcia}\address{Instituto de F\'isica Te\'orica, CSIC, Madrid, Reino de Espa\~na.}\email{angel.murcia@csic.es}

\author[C. S. Shahbazi]{C. S. Shahbazi} \address{Fachbereich Mathematik, Hamburg Universit\"at, Deutschland.}
\email{carlos.shahbazi@uni-hamburg.de}


\begin{abstract}
We introduce the notion of $\varepsilon\eta\,$-Einstein $\varepsilon\,$-contact metric three-manifold, which includes as particular cases $\eta\,$-Einstein Riemannian and Lorentzian (para) contact metric three-manifolds, but which in addition allows for the Reeb vector field to be null. We prove that the product of an $\varepsilon\eta\,$-Einstein Lorentzian $\varepsilon\,$-contact metric three-manifold with an $\varepsilon\eta\,$-Einstein Riemannian contact metric three-manifold carries a bi-parametric family of Ricci-flat Lorentzian metric-compatible connections with isotropic, totally skew-symmetric, closed and co-closed torsion, which in turn yields a bi-parametric family of solutions of six-dimensional minimal supergravity coupled to a tensor multiplet. This result allows for the systematic construction of families of Lorentzian solutions of six-dimensional supergravity from pairs of $\varepsilon\eta\,$-Einstein contact metric three-manifolds. We classify all left-invariant $\varepsilon\eta\,$-Einstein structures on simply connected Lie groups, paying special attention to the case in which the Reeb vector field is null. In particular, we show that the Sasaki and K-contact notions extend to $\varepsilon\,$-contact structures with null Reeb vector field but are however not equivalent conditions, in contrast to the situation occurring when the Reeb vector field is not light-like. Furthermore, we pose the Cauchy initial-value problem of an  $\varepsilon\,$-contact $\varepsilon\eta\,$-Einstein structure, briefly studying the associated constraint equations in a particularly simple decoupling limit. Altogether, we use these results to obtain novel  families of six-dimensional supergravity solutions, some of which can be interpreted as continuous deformations of the maximally supersymmetric solution on $\widetilde{\Sl}(2,\mathbb{R})\times S^3$.
\end{abstract}

\maketitle

\setcounter{tocdepth}{1} %doesn't display subsections in TOC 
\tableofcontents

% % % % % % % % % % % % % % % % % % % % % % % % % % % % % % % % % % % % % % 
% % % % % % % % % % % % % % % % % % % % % % % % % % % % % % % % % % % % % % 

\section{Introduction}

% % % % % % % % % % % % % % % % % % % % % % % % % % % % % % % % % % % % % % 
% % % % % % % % % % % % % % % % % % % % % % % % % % % % % % % % % % % % % % 

The ongoing mathematical study and development of supergravity \cite{Cecotti,FreedmanProeyen,Ortin} poses novel and occasionally striking mathematical problems in diverse areas of differential geometry and topology. Some of these mathematical problems are specifically concerned with the classification of supergravity solutions and the study of the associated moduli spaces of solutions. In this context, the classification of all simply connected Lorentzian manifolds admitting bosonic solutions of supergravity in the considered dimension remains as an outstanding open problem in the mathematical theory of supergravity. The purpose of this note is to propose a contribution to this problem in the specific case of minimal supergravity in six Lorentzian dimensions \cite{Nishino:1984gk,Nishino:1986dc} coupled to a tensor multiplet with constant dilaton, a theory that can be neatly rephrased as a natural geometric problem in the realm of six-dimensional Lorentzian geometry with torsion. More concretely, the bosonic configuration space $\Conf(M)$ of this six-dimensional supergravity (with constant dilaton) on a six-manifold $M$ consists on pairs $(\g,\H)$, where $\g$ is a Lorentzian metric on $M$ and $\H$ is a three-form. The bosonic theory is then geometrically defined through the following system of partial differential equations \cite{Nishino:1984gk}:
\begin{equation}
\label{eq:sugraeqintro}
\mathrm{Ric}(\nabla^{\H}) = 0\, , \qquad \dd \H = 0 \, , \qquad \dd\ast_\g \H = 0\, , \qquad \vert \mathrm{H}\vert^2_\g = 0\, ,
\end{equation}

\noindent
for pairs $(\g,\H)\in \Conf(M)$, where $\nabla^\H$ is the unique metric-compatible (with respect to $g$) connection with totally skew-symmetric torsion $\H$. Therefore, solutions of minimal supergravity consist on six-dimensional Lorentzian manifolds equipped with a Ricci-flat metric-compatible connection with isotropic, skew-symmetric, closed and co-closed torsion. Recall that in six Lorentzian dimensions the Hodge operator on three-forms squares to the identity, and the metric induced by $\g$ on the bundle of three-forms is of split signature. In particular, if $\mathrm{H}$ is closed and (anti) self-dual then it is automatically isotropic and co-closed, providing a natural candidate to solve Equations \eqref{eq:sugraeqintro}. (Anti) self-duality of $\mathrm{H}$ is in this context a clear reminiscent of the self-duality condition on the curvature of a connection on a principal $\U(1)$ bundle over a four-dimensional Riemannian manifold. This analogy can be made rigorous by using the theory of abelian gerbes \cite{Brylinski}, although we shall not pursue this point of view here. To the best of our knowledge the geometric problem posed by Equations \eqref{eq:sugraeqintro} has not been studied in the mathematical literature.

A solution $(\g,\H)$ of Equations \eqref{eq:sugraeqintro} is said to be \emph{supersymmetric} if there exists an irreducible real chiral spinor $\epsilon$ parallel with respect to the lift of $\nabla^\H$ to the spinor bundle (assuming the latter exists and is associated to a spin structure on $(M,\g)$) and sitting in the kernel of the endomorphism defined by $\mathrm{H}$ in the spinor bundle through Clifford multiplication. Supersymmetric solutions of six-dimensional supergravity have been extensively studied in the literature, see for instance  \cite{Akyol:2010iz,Akyol:2011mh,Bena:2011dd,Chamseddine:2003yy,Cano:2019gqm,deMedeiros:2018ooy,Figueroa-OFarrill:2013qzh,Gutowski:2003rg,Lam:2018jln} and references therein. However, due to the complexity of the equations involved, examples of explicit solutions are scarce even in the restricted supersymmetric case, with some celebrated exceptions, see \cite{Bena:2011dd,Lin,Liu:2004hy,Martelli:2004xq} and their citations. Rather than studying the problem of classifying solutions to minimal supergravity coupled to a tensor multiplet in its full generality, which seems to be currently out of reach, in this article we make the following simplifying assumptions, with the aim of developing a method to construct families of solutions as a first step to understand the general classification problem:

\begin{itemize}
	\item We assume that $M = N\times X$ is a direct product of three-dimensional oriented manifolds.
	
	\item We consider configurations $(\g,\H)\in \Conf(M)$ with a direct product metric $\g = \chi \oplus h $, where $\chi$ is a Lorentzian metric on $N$ and $h$ is a Riemannian metric on $X$.
\end{itemize}

\noindent
With these provisos in mind, we obtain the following result, see Theorem \ref{thm:pcontactsugrasolution} for more details.

\begin{thm}
\label{thm:intro6dsugra}
Let $(N,\chi,\alpha_N,\varepsilon_N)$ be a Lorentzian $\varepsilon\eta\,$-Einstein $\varepsilon\,$-contact metric three-manifold and let $(X,h,\alpha_X)$ be a Riemannian $\varepsilon\eta\,$-Einstein contact metric three-manifold. Then, for appropriate choices of parameters in the $\varepsilon\eta\,$-Einstein condition (specified in Theorem \ref{thm:pcontactsugrasolution}), the oriented Cartesian product manifold
\begin{equation*}
M = N\times X\, ,
\end{equation*}
\noindent
carries a family of solutions of six-dimensional minimal supergravity coupled to a tensor multiplet given by
\begin{equation*}
\g = \chi \oplus h\, , \qquad \H_{\lambda, l} = \lambda\, \nu_{\chi}  + \frac{l}{3}\,(\ast_\chi \alpha_N)\wedge \alpha_X +  \frac{l}{3}\,\alpha_N\wedge (\ast_h \alpha_X)  + \lambda\, \nu_h\, 
\end{equation*}
	
\noindent
and parametrized by $(\lambda,l)\in \mathbb{R}^2$. Equivalently, $(M = N\times X , \g = \chi \, \oplus\,  h)$ admits a bi-parametric family of metric-compatible Ricci-flat connections $\nabla^{\H_{\lambda , l}}$ with totally skew-symmetric, isotropic, closed and co-closed torsion $\H_{\lambda , l}$.
\end{thm}

\noindent
Note that, as explained in Section \ref{sec:6dsugrasolutions}, although the Levi-Civita connection $\nabla$ is a product connection, the torsionful connection $\nabla^{\H_{\lambda, l}}$ does not split in general. The notion of $\varepsilon\,$-contact structure is defined in Section \ref{sec:cmg}, and contains as a particular case the standard notions of Riemannian contact structure when $g$ is Riemannian and the standard notion of Lorentzian (para) contact structure when $(\varepsilon = 1)$ $\varepsilon = -1$. The novelty in the definition of $\varepsilon\,$-contact structure occurs in the case $\varepsilon = 0$, namely the case in which the Reeb vector field is null. To the best of our knowledge this case has not been considered in the literature, and hence we study it in more detail. Consequently, we develop the notions of K-contact and Sasakian null contact metric structure, in analogy with the corresponding definitions occurring in the cases $\varepsilon\neq 0$. We show that the Sasakian condition is equivalent to the integrability of certain nilpotent endomorphism on $\mathbb{R}\times N$, which allows to characterize such structures similarly to the case $\varepsilon \neq 0$. On the other hand, we explicitly show that, in contrast to the three-dimensional $\varepsilon\neq 0$ case, the Sasakian and K-contact conditions are not equivalent if $\varepsilon = 0$. Furthermore, we classify all left-invariant Lorentzian three-dimensional simply connected and connected Lie groups carrying left-invariant null contact structures, elucidating which of these structures are Sasakian or K-contact and finding the structure explicitly. 

The notion of $\varepsilon\eta\,$-Einstein $\varepsilon\,$-contact structure is, for $\varepsilon \neq 0$, a restricted version of the standard $\eta\,$-Einstein condition considered in the literature of contact structures, see for example \cite{Blair,Calvaruso}, which is adapted to make Theorem \ref{thm:intro6dsugra} hold. In order to apply Theorem \ref{thm:intro6dsugra} to the construction of explicit solutions of minimal supergravity coupled to a tensor multiplet with constant dilaton, we classify all left-invariant $\varepsilon\eta\,$-Einstein $\varepsilon\,$-contact structures by exploiting the classification of simply connected three-dimensional Lorentzian Lie groups, developed in \cite{CorderoParker,Rahmani}. We tabulate our results in the following form.

\begin{thm}
\label{thm:introetatime}
A three-dimensional connected and simply connected Lie group $\G$ admits a left- invariant $\varepsilon\eta\,$-Einstein contact structure $(g,\alpha)$ with time-like Reeb vector field if and only if $(\mathrm{G},g,\alpha)$ is isomorphic, through a possibly orientation-reversing isometry, to one of the items listed in the following table in terms of the orthonormal frame $\left\{ e_0 , e_1 , e_2\right\}$ appearing in Theorem \ref{thm:LorentzGroups}: 
\vspace*{-\baselineskip}
 \vspace{0.45cm}
\renewcommand{\arraystretch}{1.5}
\begin{center}
\begin{tabular}{ |P{0.5cm}|P{4.5cm}|P{1.3cm}|P{4cm}|P{1cm} | P{1.5cm}|}
\hline
		
$\mathfrak{g}$  & \emph{Structure constants } & $\alpha$ & $\eta\,$\emph{-Einstein constants} & $\G$ & \emph{Sasakian} \\
\hline
	
\multirow{4}*{$\mathfrak{g}_3$} & $\frac{1}{2}>a=1-b > 0\, , c=1$ & $\alpha=e^0$ &  $\lambda^2= \kappa=2b(1-b)$ & $\widetilde{\mathrm{Sl}}(2,\mathbb{R})$ & \emph{No} \\ \cline{2-6} &   $a=c=1\, , b=0$ & $\alpha=-e^0$ & $\lambda^2=\kappa=0$ & $\widetilde{\mathrm{E}}(1,1)$ & \emph{No}  \\ \cline{2-6} & $1 \geq a=b>0\, , c=1$ & $\alpha=e^0$ & $\lambda^2=1-\kappa=a\,$ & $\widetilde{\mathrm{Sl}}(2,\mathbb{R})$ & \emph{Yes} \\ \cline{2-6} & $ a=b=0\, , c=-1$ & $\alpha=-e^0$ & $\lambda^2=0\, , \kappa=1\,$ & $\mathrm{H}_3$ & \emph{Yes} \\ \hline $\mathfrak{g}_6$ & $b=1\, , c=d=0\, , 1 \geq a^2 > 0\,$ & $\alpha=e^0$ & $\lambda^2=1-\kappa=a^2\, ,$ & $\mathfrak{G}_6$  & \emph{Yes} \\ \hline
\end{tabular}
\end{center}

\renewcommand{\arraystretch}{1}
 
\end{thm}

\begin{thm}
\label{thm:introetapara}
A three-dimensional connected and simply connected Lie group $G$ admits a left- invariant $\varepsilon\eta\,$-Einstein para-contact structure $(g,\alpha)$ if and only if $(\mathrm{G},g,\alpha)$ is isomorphic, through a possibly orientation-reversing isometry, to one of the items listed in the following table in terms of the orthonormal frame $\left\{ e_0 , e_1 , e_2\right\}$ appearing in Theorem \ref{thm:LorentzGroups}:
 
%\vspace{-8mm}
\renewcommand{\arraystretch}{1.5}
 \begin{center}
\begin{tabular}{ |P{0.3cm}|P{4.4cm}|P{2.9cm}|P{3.2cm}|P{1cm} | P{1.2cm}|}
\hline
$\mathfrak{g}$  & \emph{Structure constants} $(s\in \mathbb{Z}_2)$  & $\alpha$ & $\eta\,$\emph{-Einstein constants} & $\mathrm{G}$ & \emph{Sasakian} \\
\hline
\multirow{6}*{$\mathfrak{g}_3$} & $a=s\, , \, b =c\, , \, sc\geq 1 $ & \multirow{1}*{$ \alpha=\pm e^1$} & $\lambda^2= s c\, , \, \kappa=  s c-1$& $\widetilde{\mathrm{Sl}}(2, \mathbb{R})$ & \emph{Yes} \\ \cline{2-6} & $b=s\, , \, a=c \, , \, \, sc\geq 1$ & \multirow{1}*{$ \alpha=\pm e^2$} & $\lambda^2=sc\, , \, \kappa=  s c-1$ & $\widetilde{\mathrm{Sl}}(2, \mathbb{R})$ & \emph{Yes} \\   \cline{2-6}& $b=0\, , a=c=s$ & $-\alpha_0^2+\alpha_1^2=1$ &  $\lambda^2=0\, , \kappa=0$ & $\widetilde{\mathrm{E}}(1,1)$ & \emph{No} \\ \cline{2-6}& $c= 0\, ,\, a=b=s$ & $\alpha_1^2+\alpha_2^2=1$ & $\lambda^2= 0\, , \kappa=0$  & $\widetilde{\mathrm{E}}(2)$ & \emph{No }\\ \cline{2-6}& $a=b=c=s$ & $-\alpha_0^2+\alpha_1^2+\alpha_2^2=1$ & $\lambda^2=1\, , \kappa=0$ & $\widetilde{\mathrm{Sl}}(2, \mathbb{R})$ & \emph{Yes} \\ \hline
\multirow{3}*{$\mathfrak{g}_6$} & $a=b=0\, ,$ $d^2\geq 1\, , c=-s$ & $\alpha_2^2=1$ & $\lambda^2=d^2\, , \kappa=d^2-1$ &\multirow{3}*{$\mathfrak{G}_6$} &  	\emph{Yes} \\ \cline{2-4} \cline{6-6}
& $b=-\mu a \neq 0 \, , \, d=-a+\mu s \, , \,$ & $a \alpha_2=(-s+\mu a ) \alpha_0$ & \multirow{2}*{$\lambda^2=1\, ,\, \kappa=0\, $} & & \multirow{2}*{\emph{Yes}} \\  & $c=-s+\mu a \, , \, \mu \in \mathbb{Z}_2\,$ & $ \alpha_2^2=1+\alpha_0^2\, $ &  & & \\ \hline
\end{tabular}
\end{center}
\renewcommand{\arraystretch}{1}
\end{thm}

\begin{thm}
\label{thm:introetanull}
A three-dimensional connected and simply connected Lie group $\mathrm{G}$ admits a left-invariant $\varepsilon\eta\,$-Einstein null contact structure $(g,\alpha)$ if and only if $(\mathrm{G},g,\alpha)$ is isomorphic, through a possibly orientation-reversing isometry, to one of the items listed in the following table in terms of the orthonormal frame $\left\{ e_0 , e_1 , e_2\right\}$ appearing in Theorem \ref{thm:LorentzGroups}:
	
%\vspace{0.25cm}
\renewcommand{\arraystretch}{1.5}
 \begin{center}
\begin{tabular}{ |P{0.5cm}|P{4.4cm}|P{2cm}|P{3.7cm}|P{1cm} |P{1.4cm}| }
\hline
$\mathfrak{g}$  & \emph{Structure constants} $(s\in \mathbb{Z}_2)$  & $\alpha$ & $\eta\,$\emph{-Einstein constants} & $\mathrm{G}$ & \emph{Sasakian}\\
\hline
\multirow{3}*{$\mathfrak{g}_3$} & $a=b=c=s$ & $ \alpha_0^2=\alpha_1^2+\alpha_2^2$ & $\lambda^2=1\, , \, \kappa=0$ &  $\widetilde{\mathrm{Sl}}(2,\mathbb{R})$  & \emph{Yes} \\ \cline{2-6} &  \multirow{2}*{$a=c=s\, , \, b=0$} &  $\alpha_2=0\, ,$   & \multirow{2}*{$\lambda^2=0\, , \, \kappa=0$} &  \multirow{2}*{$\widetilde{\mathrm{E}}(1,1)$}  & \multirow{2}*{\emph{No}}   \\ & &  $\alpha_0^2=\alpha_1^2$ & & & \\  \hline
%\multirow{2}*{$\mathfrak{g}_4$}    & $b=1 + z \, , a=1\,  , z \in \mathbb{Z}_2$ & \multirow{2}*{$\alpha_1=0\, , \alpha_0=z\alpha_2$} & $\lambda^2=1\, , \kappa=-2\sigma^2 z$ & $\widetilde{\mathrm{Sl}}(2,\mathbb{R})$ \\ \cline{2-2} \cline{4-5} & $b=1 + z\, , a=0\,  , z \in \mathbb{Z}_2$  & &$\lambda^2=0\, , \kappa=-4\sigma^2 z$ & $\widetilde{\mathrm{E}}(1,1)$\\	 
\multirow{2}*{$\mathfrak{g}_4$}    & $b=0 \, , \, a=s\, $ &$\alpha_1=0\, , \,$ & $\lambda^2=1\, , \, \alpha_0^2 \kappa=1$ & $\widetilde{\mathrm{Sl}}(2,\mathbb{R})$ &  \emph{Yes}\\ \cline{2-2} \cline{4-6} & $b=0\, , a=0\,  , \, $  & $\alpha_0=\mu  \alpha_2$  &$\lambda^2=0\, , \,\alpha_0^2 \kappa=2 $ & $\widetilde{\mathrm{E}}(1,1)$ & \emph{No} \\	 
\hline
\multirow{2}*{$\mathfrak{g}_6$}  & $a=d\neq 0\, , b=c$  &$\alpha_1=0\, ,$    & $\lambda^2=4a^2\, ,$ & \multirow{2}*{$\mathfrak{G}_6$} & \multirow{2}*{\emph{If} $a=\frac{\mu s}{2}$} \\ & $a=\mu(b+s)\, , \mu \in \mathbb{Z}_2 \,$ & $\alpha_0=-\mu \alpha_2 $ & $\kappa=0$ & & \\ \hline
\end{tabular}
\end{center}
			
			 \renewcommand{\arraystretch}{1}
	
\end{thm}

\noindent
Here $\widetilde{\mathrm{Sl}}(2,\mathbb{R})$ denotes the universal cover of $\mathrm{Sl}(2,\mathbb{R})$, $\widetilde{\mathrm{E}}(2)$ denotes the universal cover of the group of rigid motions of the
Euclidean plane, whereas $\widetilde{\mathrm{E}}(1,1)$ denotes the universal cover of the group of rigid motions of the Minkowski plane and $\mathrm{H}_3$ denotes the three-dimensional Heisenberg group. The classification of complete and simply connected $\varepsilon\eta\,$-Einstein Riemannian three-manifolds was completed in \cite{BlairKoufo,BlairKoufoII,Koufo}. The classification of complete $\varepsilon\eta\,$-Einstein Lorentzian three-manifolds with time-like Reeb vector field, presented in Theorem \ref{thm:introetatime}, proceeds along similar lines to the Riemannian case and therefore may be known to experts, although we have not been able to find it explicitly stated in the literature. By appropriately combining  $\varepsilon\eta\,$-Einstein Riemannian three manifolds with the $\varepsilon\eta\,$-Einstein $\varepsilon\,$-contact manifolds listed in the previous theorems, as prescribed in Section \ref{sec:6dsugrasolutions}, we can obtain simply-connected six-dimensional Lorentzian manifolds equipped with a continuous family of Ricci-flat metric-compatible connections with totally skew symmetric, isotropic, closed and co-closed torsion. These and all their smooth discrete quotients carry solutions of minimal supergravity coupled to a tensor multiplet with constant dilaton.

The $\varepsilon\eta\,$-Einstein condition on a globally hyperbolic Lorentzian $\varepsilon\,$-contact manifold can be interpreted as a dynamical evolution problem for a metric, a one-form and a pair of functions on a Riemann surface. As such, it admits an initial value formulation. We develop such formulation in Section \ref{sec:Einsteinpcontact}, with the goal of obtaining the associated constraint equations as a first step towards a future study of the Cauchy problem for these structures. This point of view and approach to Lorentzian $\varepsilon\eta\,$-Einstein structures seems to be largely missing in the mathematical literature, with some notable exceptions, see for instance \cite{Duggal}.  

The outline of this article is as follows. In Section \ref{sec:cmg} we introduce the notion of $\varepsilon\,$-contact structure and explore some of its generic properties, including its formulation on a globally hyperbolic Lorentzian three-manifold. Section \ref{sec:nullcontact} is devoted to $\varepsilon\,$-contact structures with light-like Reeb vector field, introducing the concepts of Sasakian and null K-contact contact structures and classifying those on a simply connected three-dimensional Lie group which are left-invariant. In Section \ref{sec:Einsteinpcontact} we introduce the notion of $\varepsilon\eta\,$-Einstein $\varepsilon\,$-contact structure and we classify all left-invariant $\varepsilon\eta\,$-Einstein $\varepsilon\,$-contact structures on simply connected three-dimensional Lie groups. In Section \ref{sec:6dsugrasolutions} we present Theorem \ref{thm:intro6dsugra} and we establish the link between $\varepsilon\eta\,$-Einstein $\varepsilon\,$-contact structures and solutions of supergravity in six dimensions. Finally, in Section \ref{sec:RicciFlatTorsion} we illustrate the type of solutions of six-dimensional supergravity that are obtained through Theorem \ref{thm:intro6dsugra} by using the $\varepsilon\eta\,$-Einstein $\varepsilon\,$-contact manifolds obtained in Theorems \ref{thm:introetatime}, \ref{thm:introetapara} and \ref{thm:introetanull}.

%%%%%%%%%%%%%%%%%%%%%%%%%%%%%%%%%%%%%%%%%%%%%%%%%%%%%%%%%%%%%%%
%%%%%%%%%%%%%%%%%%%%%%%%%%%%%%%%%%%%%%%%%%%%%%%%%%%%%%%%%%%%%%%

\subsection*{Acknowledgements}

%%%%%%%%%%%%%%%%%%%%%%%%%%%%%%%%%%%%%%%%%%%%%%%%%%%%%%%%%%%%%%%
%%%%%%%%%%%%%%%%%%%%%%%%%%%%%%%%%%%%%%%%%%%%%%%%%%%%%%%%%%%%%%%

We thank Vicente Cort\'es, Tom\'as Ort\'in, Francisco Presas and Marco Zagermann for very useful discussions. Furthermore we thank an anonymous referee of JGP for several comments which have helped improving the manuscript. The work of C.S.S. is supported by the Humboldt foundation through the Humboldt grant ESP 1186058 HFST-P and the German Science Foundation (DFG) under the Research Training Group 1670 \emph{Mathematics inspired by String Theory}. The work of AM is funded by the Spanish FPU Grant No. FPU17/04964. AM was further supported by the MCIU/AEI/FEDER UE grant PGC2018-095205-B-I00, and by the ``Centro de Excelencia Severo Ochoa'' Program grant SEV-2016-0597.

%%%%%%%%%%%%%%%%%%%%%%%%%%%%%%%%%%%%%%%%%%%%%%%%%%%%%%%%%%%%%%%
%%%%%%%%%%%%%%%%%%%%%%%%%%%%%%%%%%%%%%%%%%%%%%%%%%%%%%%%%%%%%%%

\section{$\varepsilon\,$-Contact metric three-manifolds}
\label{sec:cmg}

%%%%%%%%%%%%%%%%%%%%%%%%%%%%%%%%%%%%%%%%%%%%%%%%%%%%%%%%%%%%%%%
%%%%%%%%%%%%%%%%%%%%%%%%%%%%%%%%%%%%%%%%%%%%%%%%%%%%%%%%%%%%%%%

In this section we introduce the notion of \emph{$\varepsilon\,$-contact metric structure}, which encompasses as particular cases the standard definition of contact Riemannian metric structure, contact Lorentzian structure and para-contact metric structure in three dimensions, but which also allows for the \emph{Reeb vector field} to be null. 
 
\begin{definition}
Let $M$ be an oriented three-manifold. An {\bf $\varepsilon\,$-contact metric structure} (or $\varepsilon\,$-contact structure, in short) on $M$ consists of a triple $(g,\alpha,\varepsilon)$, with $\varepsilon \in \{-1,0,1 \}$, $g$ a Riemmanian or pseudo-Riemannian metric on $M$ and $\alpha$ a one-form $\alpha \in \Omega^1(M)$, satisfying:
\begin{equation}
\label{eq:contactcondition}
\alpha=\ast\dd \alpha\, , \qquad \vert \alpha \vert_g^2=\varepsilon,
\end{equation}

\noindent
where $\ast\colon \Omega^r(M)\to \Omega^{3-r}(M)$ ($r=0,1,2,3$) denotes the Hodge-dual with respect to $g$ and the fixed orientation on $M$, which is then said to be an $\varepsilon\,$-contact metric three-manifold (or $\varepsilon\,$-contact three-manifold, in short). When $g$ is Lorentzian, we will assume that $(M,g)$ is oriented and time-oriented. 
\end{definition}

\begin{remark}
Note that equation $\alpha=\ast\dd \alpha$ is equivalent to
\begin{equation*}
\ast\alpha = \s_g\, \dd\alpha\, ,
\end{equation*}

\noindent
where $\s_g = +1$ if $g$ is Riemannian and $\s_g = -1$ if $g$ is Lorentzian.
\end{remark}

\begin{remark}
The motivation to introduce the previous definition will be apparent in Section \ref{sec:6dsugrasolutions}, see Theorem \ref{thm:pcontactsugrasolution}.
\end{remark}

\noindent
Let $(g,\alpha,\varepsilon)$ be an $\varepsilon\,$-contact metric structure on $M$ and denote by $\nu_g$ the pseudo-Riemannian volume form associated to $g$ and the fixed orientation on $M$. If $g$ is Riemannian then necessarily $\varepsilon = 1$ (whence it can be omitted) and $(g,\alpha)$ defines in this case a standard contact metric structure on $M$ \cite{Blair}. To see this, note that the kernel of $\alpha$ defines an oriented non-integrable rank-two distribution,
\begin{equation*}
\cD \eqdef \ker(\alpha) \subset TM\, .
\end{equation*} 

\noindent
Denote by $g_{\cD}$ the restriction of $g$ to $\cD$. Being oriented, the Riemannian volume form of $(\cD,g_{\cD})$ defines a canonical almost complex structure $J_{\cD}\colon \cD \to \cD$ with respect to which $\nu_{\cD}$ is of type $(1,1)$. Furthermore, the vector field $\xi = \alpha^{\sharp}$ dual to $\alpha$ satisfies by definition $\alpha(\xi) = 1$. Equation \eqref{eq:contactcondition} implies
\begin{equation*}
\dd\alpha(v_1,v_2) = \ast\alpha(v_1,v_2) = \nu_{\cD}(v_1,v_2) = g_{\cD}(v_1,J_{\cD} v_2)\, , \qquad v_1 , v_2 \in \cD
\end{equation*}

\noindent
Therefore, the tuple $(\cD,g,\xi,\phi)$, where $\phi\vert_{\cD} = J_{\cD}$ and $\phi(\xi) = 0$, defines a standard contact metric structure on $M$. Conversely, any such Riemannian contact metric structure gives rise to a canonical Riemannian $\varepsilon\,$-contact metric structure on $M$. Similar remarks apply when $g$ is Lorentzian and $\varepsilon = -1$, in which case we recover the usual notion of Lorentzian contact metric structure, and when $g$ is Lorentzian and $\varepsilon = 1$, in which case we recover the usual notion of para-contact metric structure. Hence the following holds.

\begin{prop}
Let $M$ be an oriented three-manifold. An $\varepsilon\,$-contact metric structure $(g,\alpha,\varepsilon)$ on $M$ defines a canonical Riemannian contact metric structure if $g$ is Riemannian, a canonical Lorentzian contact metric structure if $g$ is Lorentzian and $\varepsilon = -1$, and a canonical para-contact metric structure if $g$ is Lorentzian and $\varepsilon = 1$. The converse also holds for the three previous cases.
\label{prop:equivdefis}
\end{prop}

\noindent 
\begin{remark}
By the previous argument, if  $(g,\alpha,\varepsilon)$ is an $\varepsilon\,$-contact metric structure on a Riemannian manifold $(M,g)$, we shall just denote it as a Riemannian contact metric structure in order to keep the usual nomenclature in the literature. 
\end{remark}

\noindent
Given an $\varepsilon\,$-contact structure $(g,\alpha,\varepsilon)$, we will refer to $\xi = \alpha^{\sharp}$ (the metric dual of $\alpha)$ as the Reeb vector field of $(g,\alpha,\varepsilon)$. The notion of morphism of $\varepsilon\,$-contact manifolds we consider is the expected one. 

\begin{definition}
\label{def:morphpcont}
Let $(M_a,g_a,\alpha_a,\varepsilon_a)$, $a=1,2$, be $\varepsilon\,$-contact three-manifolds. A {\bf morphism} $F$ from $(M_1,g_1,\alpha_1,\varepsilon_1)$ to $(M_2,g_2,\alpha_2,\varepsilon_2)$ with $\varepsilon_1 = \varepsilon_2$ is an orientation preserving smooth map $F\colon M_1 \to M_2$ such that:
\begin{equation*}
g_1 = F^{\ast}g_2\, , \qquad \alpha_1 = F^{\ast}\alpha_2\, .
\end{equation*} 

\noindent
If such $F$ is not orientation preserving, then we will say that it is an {\bf orientation-reversing morphism}.
\end{definition}

\begin{remark}
We denote by $\mathrm{PCont}$ the category whose objects are $\varepsilon\,$-contact three-manifolds and whose morphisms are defined as above. Relevant subcategories of $\mathrm{PCont}$ are the subcategory of contact Riemannian three-manifolds $\mathrm{PCont}_{R}$ and the category $\mathrm{PCont}_{L}(\varepsilon)$ of $\varepsilon\,$-contact Lorentzian three-manifolds with Reeb vector field of norm $\varepsilon\in \left\{-1,0,1\right\}$. 
\end{remark}

\noindent
In analogy with the standard theory of Riemannian contact structures, we introduce, associated to every $\varepsilon\,$-contact metric structure $(g,\alpha,\varepsilon)$, two endomorphisms $\phi\colon TM\to TM$ and $\mathfrak{h}\colon TM\to TM$:
\begin{equation}
\phi(v) = - \s_g \, (\iota_v\ast\alpha)^{\sharp}\, , \qquad \mathfrak{h}(v) = (\cL_{\xi}\phi)(v) \qquad \forall\,\, v\in TM\, ,
\label{eq:definitionphi}
\end{equation}

\noindent 
where $\xi = \alpha^{\sharp}$ is the Reeb vector field of $(g,\alpha,\varepsilon)$ and the symbol $\cL$ denotes Lie derivative. We will refer to $\phi\in \Gamma(TM\otimes T^{\ast}M)$ as the \emph{characteristic endomorphism} of $(g,\alpha,\varepsilon)$. Furthermore, from $\phi$ and $\mathfrak{h}$ we define:
\begin{equation*}
\tau \eqdef \mathfrak{h}\circ \phi \colon TM \to TM\, .
\end{equation*}

\noindent
These endomorphisms will play an important role later on. The following lemma summarizes some of the properties enjoyed by $\phi$.

\begin{lemma}
\label{lemma:phiproperties}
The characteristic endomorphism $\phi$ of an $\varepsilon\,$-contact metric manifold $(M,g,\alpha,\varepsilon)\in \mathrm{PCont}$ satisfies:
\begin{eqnarray*}
g(\mathrm{Id}\otimes \phi)= \dd\alpha \, , \quad \phi(\xi) = 0\, , \quad \alpha\circ \phi = 0\, , \quad \phi^2 = \s_g\,(-\varepsilon\,\mathrm{Id} + \xi \otimes \alpha)\, , \quad g\circ \phi\otimes\phi = \s_g (\varepsilon\, g  - \alpha\otimes \alpha)\, .
\end{eqnarray*}
	
\noindent
where $\xi\eqdef\alpha^\sharp$ denotes the Reeb vector field of $(g,\alpha,\varepsilon)$.
\end{lemma}
\begin{proof}
Using the definition of $\phi$ given in Equation \eqref{eq:definitionphi}, we compute:
\begin{equation*}
g(v_1, \phi(v_2))=-\mathfrak{s}_g g(v_1,(\iota_{v_2} \ast \alpha)^\sharp )=-\mathfrak{s}_g \ast \alpha (v_2,v_1)=\dd \alpha(v_1,v_2)\, , \qquad v_1 , v_2 \in \mathfrak{X}(M)\, ,
\end{equation*}

\noindent
whence the first equation of the lemma holds. The second and third equations of the lemma follow directly from the definition of $\phi$. On the other hand, the square of the endomorphism $\phi$ can be computed to be:
\begin{equation*}
\begin{split}
\phi (\phi (v))=[\iota_{(\iota_{v} \ast \alpha)^\sharp} (\ast \alpha)]^\sharp&=[(\ast \alpha)((\iota_{v} \ast \alpha)^\sharp)]^\sharp=[\ast(\alpha \wedge (\iota_{v} \ast \alpha) )]^\sharp\\&=[-\ast (\iota_{v} (\alpha \wedge \ast \alpha))+\alpha(v) (\ast \ast \alpha)]^\sharp=-\mathfrak{s}_g \varepsilon \, v+\mathfrak{s}_g \alpha (v) \xi\, , \qquad v\in \mathfrak{X}(M)\, ,
\end{split}
\end{equation*}

\noindent
which implies the fourth equation of the lemma. Finally, this last equation follows from:
\begin{equation*}
g(\phi(v_1), \phi(v_2))=\dd \alpha(\phi(v_1),v_2)=-\dd \alpha(v_2,\phi(v_1))=-g(v_2,\phi^2(v_1))=\mathfrak{s}_g \varepsilon g(v_1,v_2)-\mathfrak{s}_g \alpha(v_1)\alpha(v_2)\, .
\end{equation*}
\end{proof}

\begin{remark}
Lemma \ref{lemma:phiproperties} recovers key identities satisfied by $\varepsilon$-contact structures which in classical references are taken as part of the definition of Riemannian contact structures \cite{Blair} ($\varepsilon=\mathfrak{s_g}=1$), Lorentzian contact structures \cite{Calvaruso} ($\varepsilon=\mathfrak{s_g}=-1$) or para-contact structures \cite{CalvarusoII} ($\varepsilon=-\mathfrak{s_g}=1$).
\end{remark}

\begin{remark}
Note that the characteristic endomorphism of an $\varepsilon\,$-contact metric structure $(g,\alpha,\varepsilon)$ is always skew-symmetric with respect to $g$, that is,
\begin{equation*}
g(\phi(v_1),v_2) + g(v_1,\phi(v_2)) = 0\, , \qquad\forall \,\, v_1 , v_2\in TM\, .
\end{equation*}
\end{remark}

\noindent
Given an $\varepsilon\,$-contact structure $(g,\alpha,\varepsilon)$ we define yet another endomorphism, denoted by $\mathfrak{l}\in \End(TM)$, as follows:
\begin{equation*}
\mathfrak{l}(v) = \mathrm{R}^g(v,\xi) \xi\, , \qquad \forall\,\, v\in \mathfrak{X}(M)\, ,
\end{equation*}

\noindent
where $\mathrm{R}^g$ denotes the Riemann curvature tensor of $g$. The endomorphism $\mathfrak{l}\colon TM\to TM$ should not be confused with the endomorphism defined by the Ricci curvature, which we denote by $\mathrm{Q}^g$. This makes five the endomorphisms $(\phi,\mathfrak{h},\tau,\mathfrak{l}, \mathrm{Q}^g)$ canonically associated to every $\varepsilon\,$-contact structure $(g,\alpha,\varepsilon)$. Manifolds equipped with an $\varepsilon\,$-contact structure $(g,\alpha,\varepsilon)$ admit a special frame which is very convenient for computations.

\begin{definition}
\label{def:lcf}
Let $(g,\alpha,\varepsilon)$ be an $\varepsilon\,$-contact metric structure on $M$ with Reeb vector field $\xi\in \mathfrak{X}(M)$. An {\bf $\varepsilon\,$-contact frame} is a local frame $\left\{\xi,u,\phi(u)\right\}$, where $u\in \mathfrak{X}(M)$ is a nowhere vanishing vector field satisfying:
\begin{equation*}
g(u,u) = \s_g \,\varepsilon\, , \qquad g(u,\xi) = 1 - \varepsilon^2\, .
\end{equation*}

\noindent
When $\varepsilon = 0$ we will refer to $\left\{\xi,u,\phi(u)\right\}$ as a {\bf light-cone frame} for $(g,\alpha,0)$, following standard usage in Lorentzian geometry.
\end{definition}

\begin{remark}
\label{remark:lcb}
It is a direct calculation to verify that $\left\{\xi,u,\phi(u)\right\}$ is indeed a local frame. We have:
\begin{equation*}
g(\xi,\xi)=\varepsilon\, , \qquad g(u,\phi(u)) = 0\, , \qquad g(\xi,\phi(u)) = 0\, , \qquad g(\phi(u),\phi(u)) = 1\, ,
\end{equation*}
	
\noindent
whence $\phi(u)$ is locally nowhere vanishing and point-wise orthogonal to the real span of $\xi$ and $u$.
\end{remark}

\begin{prop}
\label{prop:eqfh}
Let $(g,\alpha,\varepsilon)$ be an $\varepsilon\,$-contact metric structure. The following equations hold: 
\begin{eqnarray}
\label{eq:eqfh}
& \nabla_{\xi}\xi = 0\, , \qquad \nabla_{\xi} \phi = 0\, , \qquad \mathfrak{h}(\xi) = 0\, , \qquad \mathfrak{l}(\xi) = 0\, , \qquad \mathrm{Tr}(\mathfrak{h}) = 0\, , \\  & \mathcal{L}_\xi\alpha=0\, , \qquad \mathrm{Tr}(\tau) = 0\, , \qquad \mathfrak{h}\circ \phi + \phi\circ \mathfrak{h} = 0\nonumber\, ,
\end{eqnarray}

\noindent
where $\nabla$ denotes the Levi-Civita connection with respect to $g$. Furthermore, both $\mathfrak{h}$ and $\tau$ are symmetric with respect to $g$.
\end{prop}

\begin{proof}
The proof of Equations \eqref{eq:eqfh} follows by direct computation on an $\varepsilon\,$-contact frame. Thus, we prove only the symmetry properties of $\mathfrak{h}$ and $\tau$. We compute:
\begin{eqnarray*}
g(\mathfrak{h}(v_1),v_2) = g((\cL_{\xi}\phi)(v_1),v_2) = g(\cL_{\xi}\phi(v_1) - \phi(\cL_{\xi}v_1),v_2) = g(  - \nabla_{\phi(v_1)}\xi + \phi(\nabla_{v_1} \xi) , v_2)\, .
\end{eqnarray*}

\noindent
This expression vanishes whenever $v_1$ or $v_2$ are equal to $\xi$. Given an $\varepsilon\,$-contact frame $\left\{ \xi,u,\phi(u)\right\}$, assume now that both $v_1$ and $v_2$ belong to the span of of $u$ and $\phi(u)$. We obtain:
\begin{equation*}
g(\mathfrak{h}(v_1),v_2) = \alpha(\nabla_{\phi(v_1)} v_2 + \nabla_{v_1}\phi(v_2)) = \alpha(\nabla_{\phi(v_2)} v_1 + \nabla_{v_2}\phi(v_1)) = g(v_1,\mathfrak{h}(v_2)) \, ,
\end{equation*}

\noindent
where we have used that $\alpha([\phi(v_1),v_2]) + \alpha([v_1,\phi(v_2)]) = 0$, since
\begin{equation*}
\begin{split}
0&=g(\phi(v_1), \phi(v_2) )=\dd \alpha(\phi(v_1),v_2)-\dd \alpha(\phi(v_2),v_1)=(\mathcal{L}_{\phi(v_1)} \alpha) (v_2)-(\mathcal{L}_{\phi(v_2)} \alpha) (v_1)\\&=\alpha([\phi(v_1),v_2]) + \alpha([v_1,\phi(v_2)])\,.
\end{split}
\end{equation*}
The symmetry of $\tau$ follows now directly:
\begin{equation*}
g(\tau(v_1),v_2) = g(\mathfrak{h}(\phi(v_1)),v_2) = g(\phi(v_1),\mathfrak{h}(v_2)) = -g(v_1,\phi(\mathfrak{h}(v_2))) = g(v_1, \tau(v_2))\, ,  
\end{equation*}

\noindent
and we conclude.
\end{proof}

\noindent
An $\varepsilon\,$-contact structure $(g,\alpha,\varepsilon)$ also satisfies the following identities, which play a key role in the classification of $\varepsilon\eta\,$-Einstein contact structures.

\begin{prop}
\label{prop:equationscontact}
Let $(g,\alpha,\varepsilon)$ be an $\varepsilon\,$-contact metric structure. The following equation holds: 
\begin{eqnarray*}
2\phi(\nabla \xi)  = \mathfrak{h} + \s_g ( \varepsilon - \xi\otimes\alpha )\, ,\\
\end{eqnarray*}
\end{prop}

\begin{remark}
\label{remark:covxi}
We can apply $\phi$ to the second equation in the previous proposition, obtaining:
\begin{equation*}
2 \phi^2(\nabla \xi) = - 2\s_g \varepsilon \, \nabla^g\xi = \s_g \varepsilon\,\phi   + \phi\circ\mathfrak{h}\, .
\end{equation*}

\noindent
For $\varepsilon\neq 0$, that is, for $\varepsilon\,$-contact structures with non-null Reeb vector field, this equation gives the well-known formula for the covariant derivative of $\xi$ \cite{Blair,CalvarusoPC}. For $\varepsilon = 0$ this equation reduces to $\tau = 0$, which we will prove independently in Lemma \ref{lemma:hphi0}. Therefore, for $\varepsilon\,$-contact structures with null Reeb vector field, the covariant derivative of the Reeb vector is not prescribed in terms of $\phi$ and $\mathfrak{h}$. This has important consequences in the classification of $\varepsilon\eta\,$-Einstein $\varepsilon\,$-contact structures with null Reeb vector field.
\end{remark}
%For the first equation directly:
%\begin{equation*}
%(\nabla_{\xi}\phi)(v) = \nabla_{\xi}(\phi(v)) - \phi(\nabla_{\xi}v) = \nabla_{\xi}(\ast\alpha)(v)^{\sharp} - (\ast\alpha(\nabla_{\xi} v))^{\sharp} = ((\ast\nabla_{\xi}\alpha)(v))^{\sharp} = 0\, ,
%\end{equation*}
\begin{proof} We use Koszul's formula for the Levi-Civita connection:
\begin{eqnarray*}
&-2 g(\phi(\nabla_{v_1}\xi),v_2) = 2 g(\nabla_{v_1}\xi,\phi(v_2)) = \xi\cdot g(v_1 , \phi(v_2)) - \phi(v_2)\cdot \alpha(v_1) - g(\cL_{\xi}v_1,\phi(v_2))\\ 
&+ g([\phi(v_2),v_1],\xi) - g(\cL_{\xi}(\phi(v_2)),v_1) = -  g(\mathfrak{h}(v_1), v_2) - \phi(v_2)\cdot \alpha(v_1) - \alpha([v_1 , \phi(v_2)]) \\
&= -  g(\mathfrak{h}(v_1), v_2) + g(\phi^2(v_1),v_2) = -  g(\mathfrak{h}(v_1),v_2) + g({\s_g}( -\varepsilon v_1 + \alpha(v_1) \xi ) , v_2) \, ,
\end{eqnarray*}

\noindent
for every $v_1 , v_2 \in\mathfrak{X}(M)$, where we have used $\phi^2 = \s_g\,(-\varepsilon\,\mathrm{Id} + \xi \otimes \alpha)$ and the equation
\begin{equation*}
\dd\alpha(v_1,\phi(v_2)) = v_1\cdot \alpha(\phi(v_2)) - \phi(v_2)\cdot \alpha(v_1) - \alpha([v_1 , \phi(v_2)]) =  - \phi(v_2)\cdot \alpha(v_1) - \alpha([v_1 , \phi(v_2)])\, .
\end{equation*}
\end{proof}

\begin{remark}
Note that the equation for $\nabla \xi$ given by Proposition \ref{prop:equationscontact} differs from the one usually found in the literature by a factor of $\frac{1}{2}$. This discrepancy is due to the different convention we are using for the exterior derivative. Given any $p$-form $\omega $, we define its exterior derivative $\dd \omega$ as the $(p+1)$-form given by:
\begin{eqnarray*}
\dd \omega(X_0,\dots, X_p)=\sum_{i} (-1)^i X_i(\omega(X_0,\dots, \hat{X}_i, \dots, X_i))\\
+ \sum_{i<j} (-1)^{i+j} \omega([X_i,X_j],X_0, \dots, \hat{X}_i, \dots, \hat{X}_j, \dots, X_p)\,.
\end{eqnarray*}
Much of the literature on contact geometry, see for example \cite{Blair}, uses the conventions of Kobayashi and Nomizu \cite{KN}, in which the formula of the exterior derivative differs by a factor of $\frac{1}{p+1}$ from the one stated above. 
\end{remark}

\noindent
Proceeding by analogy with the theory of Riemannian contact metric structures we introduce the notions of Sasakian and K-contact $\varepsilon\,$-contact metric structures.

\begin{definition}
An  $\varepsilon\,$-contact metric structure $(g,\alpha,\varepsilon)$ is said to be {\bf Sasakian} if $\mathfrak{h} = 0$. It is said to be {\bf K-contact} if the Reeb vector field is Killing, that is, if $\cL_{\xi}g=0$.
\end{definition}

\begin{remark}
\label{rem:sasepsn0}
The Sasakian and K-contact conditions are well-known to be equivalent for $\varepsilon\,$-contact structures with $\varepsilon \neq 0$ in three dimensions, see \cite{Blair,Calvaruso,CalvarusoII}. However, as we will see in Section \ref{sec:nullcontact}, this fails to be the case for null contact metric structures. Indeed, Example \ref{ep:SasakiannoK} shows that the theory of Sasakian null contact metric structures is strictly richer that the theory of null K-contact structures.
\end{remark}

\noindent
Sasakian $\varepsilon\,$-contact metric structures with $\varepsilon \neq 0$ have been extensively studied in the literature, see for instance \cite{Blair}, \cite{Calvaruso} or \cite{Zamkoyov} for more details and an exhaustive list of references. In particular, it is well-known that the Sasakian condition can sometimes be equivalently formulated as a curvature condition involving the Ricci curvature tensor  $\mathrm{Ric}^g$. Since this will be of use later, we briefly review this result using our conventions.

\begin{prop}
\label{prop:ricxien0}
Let $(g,\alpha,\varepsilon)$ be an $\varepsilon\,$-contact structure with $\varepsilon \neq 0$. Then 
\begin{equation*}
\emph{Ric}^g(\xi,\xi)=\varepsilon\,
 \mathfrak{s_g} \left( \frac{1}{2}-\frac{1}{4}\text{\emph{Tr}}(\mathfrak{h}^2) \right)\,.
\end{equation*}
where $\emph{Ric}^g$ is the Ricci curvature of $g$.
\end{prop}

\begin{proof}
The proofs of the proposition in the Riemannian and Lorentzian $\varepsilon = -1$ cases are presented in detail in \cite{Blair} and \cite{Calvaruso}, respectively. Hence we focus on the para-contact case.

Proposition \ref{prop:equationscontact} adapted to the para-contact case yields $\nabla \xi=-\frac{1}{2}\phi+\frac{1}{2}\phi \circ \mathfrak{h}$. Hence, denoting by $\mathrm{R}^g$ the Riemann curvature tensor, we have:
\begin{equation*}
\mathrm{R}^g(\xi, v)\xi=\nabla_\xi \nabla_v \xi-\nabla_v \nabla_\xi \xi-\nabla_{[\xi,v]} \xi=-\frac{1}{2}\nabla_\xi (\phi(v))+\frac{1}{2}\nabla_\xi (\phi(\mathfrak{h}(v)) )+\frac{1}{2}\phi([\xi,v])-\frac{1}{2}\phi(\mathfrak{h}([\xi,v]))\, ,
\end{equation*}

\noindent
for $v \in \ker(\alpha)$ and where we used that $\nabla_\xi \xi=0$. Applying $\phi$ at the previous equation, taking into account that $\phi^2=\text{Id}-\xi \otimes \alpha$ and that $\nabla_\xi \phi=0$, we obtain:
\begin{equation*}
\phi (\mathrm{R}^g(\xi, v)\xi)=\frac{1}{2}\nabla_\xi (-v+\mathfrak{h}(v))+\frac{1}{2}[\xi,v]-\frac{1}{2}\mathfrak{h}([\xi,v]))\, ,
\end{equation*}

\noindent
where we have used that $\alpha \circ \phi=\alpha \circ \mathfrak{h}=0$ and that $\alpha([\xi, v])=0$ (the latter equation follows from the para-contact condition $\dd \alpha=-\ast \alpha$). Applying now Proposition \ref{prop:equationscontact}, we conclude:
\begin{equation}
\label{eq:Ricciphi}
\phi (\mathrm{R}^g(\xi, v)\xi)=\frac{1}{2}(\nabla_\xi \mathfrak{h}) (v)+\frac{1}{4}\phi(v)-\frac{1}{4}\mathfrak{h}^2(\phi(v))\, .
\end{equation}

\noindent
On the other hand, $\phi^2(\mathrm{R}^g(\xi, v)\xi)=\mathrm{R}^g(\xi, v)\xi$ since $\alpha( \mathrm{R}^g(\xi, v)\xi) = 0$ vanishes identically. Consequently,
\begin{equation*}
\mathrm{R}^g(\xi, v)\xi = \frac{1}{2}\phi(\nabla_\xi \mathfrak{h})(v)+\frac{1}{4}\phi^2(v)-\frac{1}{4}\mathfrak{h}^2 (v)\, .
\end{equation*}

\noindent
Equation \eqref{eq:Ricciphi} implies that
\begin{equation*}
\phi (\mathrm{R}^g(\xi, \phi(v))\xi)=-\frac{1}{2}\phi (\nabla_\xi \mathfrak{h}) (v)+\frac{1}{4}\phi^2(v)-\frac{1}{4}\mathfrak{h}^2(v)\, .
\end{equation*}

\noindent
Using now the previous formulae for $v$ of unit norm we obtain:
\begin{equation*}
g(\mathrm{R}^g (\xi,v) \xi,v)+g(\phi (\mathrm{R}^g (\xi, \phi(v)) \xi),v) =\frac{1}{2}-\frac{1}{4}\text{Tr}(\mathfrak{h}^2)=-\ric^g(\xi,\xi)\, ,
\end{equation*}

\noindent
where we used the fact that $g(\mathfrak{h}^2\phi (v), v)=-g(\mathfrak{h}^2(v),v)$, so $g(\mathfrak{h}^2(v),v)=\frac{1}{2}\text{Tr}(\mathfrak{h}^2)$. Hence
\begin{equation*}
\ric^g(\xi,\xi)=-\frac{1}{2}+\frac{1}{4}\text{Tr}(\mathfrak{h}^2)\, ,
\end{equation*}

\noindent
and we conclude.
\end{proof}

\begin{remark}
The equation proved in Proposition \ref{prop:ricxien0} differs from the one found usually in the literature by a factor of $\frac{1}{4}$. This difference can be traced back to the different convention used in this work for the exterior derivative and $\mathfrak{h}$ with respect to the conventions used in references \cite{Blair} and \cite{Calvaruso}.  
\end{remark}

\begin{prop}
\label{prop:sasakireeb}
An $\varepsilon\,$-contact structure $(g,\alpha,\varepsilon)$ with $\varepsilon\, \mathfrak{s}_g = 1$ is Sasakian if and only if: 
\begin{equation*}
\emph{Ric}^{g}(\xi,\xi)=\mathfrak{s}_g  \frac{\varepsilon}{2}\,.
\end{equation*}
\end{prop}

\begin{proof}
The \emph{only if} direction follows by setting $\mathfrak{h} = 0$ in Proposition \ref{prop:ricxien0}. Conversely, if $\ric^{g}(\xi,\xi)=\mathfrak{s}_g  \frac{\varepsilon}{2}$ then Proposition \ref{prop:ricxien0} implies that $\text{Tr}(\mathfrak{h}^2)=0$. Taking into account that $\phi$ is skew-adjoint with respect to $g$ and that $\mathfrak{h}$ is self-adjoint with respect to $g$, we obtain $g(\mathfrak{h}^2(X), \phi(X))=0$ for every vector field $X$. Since $\mathfrak{h}(\xi) = 0$ and $g$ is positive definite when restricted to $\ker(\alpha)$ (assuming $\varepsilon \s_g = 1$), then condition $\text{Tr}(\mathfrak{h}^2)=0$ implies that $\mathfrak{h}^2 = 0$. Since $\mathfrak{h}$ is self-adjoint, this yields $\mathfrak{h}=0$. 
\end{proof}

%\begin{remark}
%\label{remark:sasakireeb}
%Due to the conventions we use in the paper, a 3-dimensional $\varepsilon\,$-contact structure $(g,\alpha)$ with $\varepsilon \neq 0$ is Sasakian (or, equivalently, K-contact) if and only if 
%\begin{equation*}
%\ric^g=\mathfrak{s_g}  \frac{\varepsilon}{2}\,.
%\end{equation*}
%\end{remark}

%%%%%%%%%%%%%%%%%%%%%%%%%%%%%%%%%%%%%%%%%%%%%%%%%%%%%%%%%%%%%%%

\subsection{Globally hyperbolic $\varepsilon\,$-contact metric three-manifolds}
\label{sec:globhyperL}

%%%%%%%%%%%%%%%%%%%%%%%%%%%%%%%%%%%%%%%%%%%%%%%%%%%%%%%%%%%%%%%

In this section we describe $\varepsilon\,$-contact metric structures on globally hyperoblic Lorentzian manifolds. This class of $\varepsilon\,$-contact metric three-manifolds is specially relevant for our purposes, since they can be used to construct globally hyperbolic Lorentzian six-manifolds equipped with a Ricci-flat metric connection with totally skew-symmetric and closed torsion, as described in Section \ref{sec:6dsugrasolutions} and Theorem \ref{thm:pcontactsugrasolution}. At the same time, globally hyperbolic solutions of six-dimensional supergravity play a prominent role in the celebrated fuzzball proposal to describe the microscopic entropy of a black hole, see \cite{Bena:2007kg,Mathur:2009hf} and references therein for more details.

Let $(M,g)$ be a globally hyperbolic Lorentzian three-manifold. A celebrated theorem of A. Bernal and M. S\'anchez \cite{Bernal:2003jb} states that $(M,g)$ admits the following presentation:
\begin{equation*}
(M,g) = (\mathbb{R}\times \mathrm{X} , \,-\beta^2_t\,\dd t\otimes \dd t + q_t)\, ,
\end{equation*}

\noindent
where $t$ is a coordinate on $\mathbb{R}$, $\left\{ \beta_t \right\}_{t\in\mathbb{R}}$ is a family of nowhere vanishing functions on $\mathrm{X} \eqdef \left\{ 0\right\}\times\mathrm{X}$ and $\left\{ q_t\right\}$ is a family of complete Riemannian metrics on $X$. In this presentation, $(M,g)$ is oriented and time-oriented, which immediately fixes an orientation on $\mathrm{X}$. We denote by $\nu_{q_t}$ the Riemannian volume form associated to $q_t$. Let $\alpha$ be a one form on $M$. Set $e^0_t = \beta_t\, \dd t$. We write:
\begin{equation*}
\alpha = F_t\, e^0 + \alpha^{\perp}_t\, ,
\end{equation*}

\noindent
where $\left\{ F_t \right\}_{t\in\mathbb{R}}$ is a unique family of nowhere vanishing functions on $\mathrm{X}$ and $\left\{ \alpha_t \right\}_{t\in \mathbb{R}}$ is a unique family of one-forms on $X$. With these provisos in mind, the dual of $\alpha$ can be computed to be:
\begin{equation*}
\ast \alpha = -F_t\, \nu_{q_t} - e^0 \wedge \ast_{q_t}\alpha^{\perp}_t\, ,
\end{equation*}

\noindent
where $\ast_{q_t}$ denotes the Hodge-dual of $(X,q_t)$. On the other hand, the exterior derivative of $\alpha$ reads:
\begin{equation*}
\dd \alpha = \dd_{\mathrm{X}} F_t\wedge e^0_t  + \frac{F_t}{\beta_t} \dd_{\mathrm{X}}\beta_t\wedge e^0_t + \frac{1}{\beta_t} e^0_t \wedge \partial_t \alpha^{\perp}_t + \dd_{\mathrm{X}}\alpha^{\perp}_t\, ,
\end{equation*}

\noindent
where $\dd_{\mathrm{X}}$ is the exterior derivative on $\mathrm{X}$. Define $n_t \eqdef \frac{1}{\beta_t}\partial_t$. Altogether, the previous discussion implies the following characterization of $\varepsilon\,$-contact structures on $(M,g)$.

\begin{prop}
\label{prop:globhypervarepsilon}
Let $(M,g) = (\mathbb{R}\times \mathrm{X} , - \beta^2_t\,\dd t\otimes \dd t + q_t)$ be a globally hyperbolic Lorentzian three-manifold. A one-form $\alpha\in \Omega^1(M)$ defines an $\varepsilon\,$-contact metric structure on $(M,g)$ if and only if:
\begin{equation}
\label{eq:globhypervarepsilon}
\dd_{\mathrm{X}}\alpha^{\perp}_t = F_t\,\nu_{q_t}\, , \qquad \ast_{q_t}\alpha^{\perp}_t + \frac{1}{\beta_t}\dd_{\mathrm{X}}(\beta_t F_t) = \cL_{n_t}\alpha_t^{\perp}\, , \qquad \vert \alpha_t^{\perp}\vert^2_{q_t} = \varepsilon + F^2_t\, ,
\end{equation}

\noindent
where $\left\{ F_t\right\}_{t\in \mathbb{R}}$, $\left\{ \beta_t\right\}$ and $\left\{ \alpha^{\perp}_t\right\}_{t\in \mathbb{R}}$ are defined above and $\alpha = F_t\, e^0 + \alpha^{\perp}_t$.
\end{prop}

\begin{remark}
\label{remark:hypertuple}
The previous proposition yields a system of flow equations for a pair of families of functions $\left\{ F_t\right\}_{t\in\mathbb{R}}$, $\left\{ \beta_t\right\}_{t\in\mathbb{R}}$, a family of one-forms $\left\{ \alpha^{\perp}_t\right\}_{t\in\mathbb{R}}$ and a family of complete Riemannian metrics $\left\{ q_t\right\}_{t\in\mathbb{R}}$ on an oriented two-manifold $\mathrm{X}$. To the best of our knowledge, this system has not been studied in the literature. We hope to study it in more detail elsewhere. 
\end{remark}

\begin{definition}
Given a three-manifold $M = \mathbb{R}\times X$, a {\bf globally hyperbolic $\varepsilon\,$-contact structure} on $M$ is a tuple
\begin{equation*}
(\left\{ \beta_t \right\}_{t\in\mathbb{R}},\left\{ F_t \right\}_{t\in\mathbb{R}}, \left\{ \alpha^{\perp}_t \right\}_{t\in\mathbb{R}}, \left\{ q_t \right\}_{t\in\mathbb{R}})\, ,
\end{equation*}

\noindent
defined as specified in Remark \ref{remark:hypertuple}, which satisfies Equations \ref{eq:globhypervarepsilon}.
\end{definition}

\noindent
Given a globally hyperbolic $\varepsilon\,$-contact structure $(\left\{ \beta_t \right\}_{t\in\mathbb{R}},\left\{ F_t \right\}_{t\in\mathbb{R}}, \left\{ \alpha^{\perp}_t \right\}_{t\in\mathbb{R}}, \left\{ q_t \right\}_{t\in\mathbb{R}})$, it is clear from the previous discussion how to reconstruct $(g,\alpha, \varepsilon)$. We consider now a particular example in the case $\varepsilon = 1$.  

\begin{ep}
In the conditions of Proposition \ref{prop:globhypervarepsilon} set
\begin{equation*}
F_t = 0\, , \qquad \beta_t = 1\, .
\end{equation*}

\noindent
With these choices for $\left\{ F_t\right\}_{t\in \mathbb{R}}$ and $\left\{ \beta_t\right\}_{t\in \mathbb{R}}$, Equations \eqref{eq:globhypervarepsilon} read:
\begin{equation}
\label{eq:globhypervarepsilonII}
\dd_{\mathrm{X}}\alpha^{\perp}_t = 0\, , \qquad  \ast_{q_t}\alpha^{\perp}_t = \partial_t\alpha_t^{\perp}\, , \qquad \vert \alpha_t^{\perp}\vert^2_{q_t} = \varepsilon\, ,
\end{equation}

\noindent
which immediately implies $\varepsilon = 1$, corresponding to the case in which the Reeb vector field is space-like and the associated pair $(g,\alpha,\varepsilon = 1)$ defines a para-contact metric structure. Consider now
\begin{equation*}
(\mathrm{X},q_t) = (\mathbb{R}^2 , e^{2 U_t} (\dd x\otimes \dd x + \dd y\otimes\dd y))\, ,
\end{equation*}

\noindent
where $\left\{ U_t\right\}_{t\in \mathbb{R}}$ is a family of constant functions and $(x,y)$ are coordinates on $\mathbb{R}^2$. Let us write $\alpha^\perp=e^{U_t}( \alpha_1^\perp \dd x + \alpha_2^\perp \dd y)$. With these assumptions, Equations \eqref{eq:globhypervarepsilonII} are equivalent to:
\begin{equation*}
 \partial_t (\alpha^{\perp}_1 e^{U_t}) = - \alpha^{\perp}_2 e^{U_t}\, , \qquad  \partial_t (\alpha^{\perp}_2 e^{U_t}) = \alpha^{\perp}_1 e^{U_t}\, , \qquad (\alpha^{\perp}_1)^2  + (\alpha^{\perp}_2)^2 = 1\, .
\end{equation*}

\noindent
These equations are solved by:
\begin{eqnarray*}
\alpha^{\perp}_1 e^{U_t} = l_2 \cos(t) - l_1 \sin(t)\, , \qquad \alpha^{\perp}_2 e^{U_t} = l_1 \cos(t) + l_2 \sin(t)\, , \qquad  e^{2 U_t} = l_1^2 + l_2^2\, ,
\end{eqnarray*}

\noindent
for real constants $l_1$ and $l_2$ such that $l_1^2 + l_2^2 \neq 0$.
\end{ep}

%%%%%%%%%%%%%%%%%%%%%%%%%%%%%%%%%%%%%%%%%%%%%%%%%%%%%%%%%%%%%%%
%%%%%%%%%%%%%%%%%%%%%%%%%%%%%%%%%%%%%%%%%%%%%%%%%%%%%%%%%%%%%%%

\section{Null contact metric structures}
\label{sec:nullcontact}

%%%%%%%%%%%%%%%%%%%%%%%%%%%%%%%%%%%%%%%%%%%%%%%%%%%%%%%%%%%%%%%
%%%%%%%%%%%%%%%%%%%%%%%%%%%%%%%%%%%%%%%%%%%%%%%%%%%%%%%%%%%%%%%

The case of $\varepsilon\,$-contact metric structures with null Reeb vector field ($\varepsilon = 0$) seems to be new in the literature and thus deserves further attention. For simplicity in the exposition, we will refer to Lorentzian $\varepsilon\,$-contact metric structures with $\varepsilon = 0$ simply as null contact structures. The definition of null contact metric structure allows for the Reeb vector field to be identically zero. We will refer to this case as being \emph{trivial}. Unless otherwise specified, we will always consider non-trivial null contact metric structures.

\begin{remark}
Let $(g,\alpha)$ be a null contact structure. We have:
\begin{equation*}
\alpha \wedge \dd \alpha=-\alpha \wedge \ast \alpha = 0\, ,
\end{equation*}

\noindent
where we have used that $\vert \alpha \vert_g^2=0$. Therefore, the previous computation implies that the one-form $\alpha$ of a null contact structure is not a \emph{contact form} since it does not satisfy $\alpha\wedge \dd\alpha \neq 0$. Nevertheless, given that the definition of $\varepsilon\,$-contact structure encompasses Riemannian contact ($\varepsilon = 1$), Lorentzian contact ($\varepsilon = -1$) and para-contact ($\varepsilon = 1$) metric structures, we interpret the remaining case $\varepsilon = 0$ as a natural generalization of the formers but in which the Reeb vector field is null. Hence we will continue referring to the case $\varepsilon = 0$ as a null contact structure. Moreover, we will show later in this section that null contact structures admit meaningful notions of Sasakianity and K-contactness, analogously to the $\varepsilon\neq 0$ cases.
\end{remark}

\noindent
Having shed some light into the nature of null contact structures, we proceed to investigate their most relevant properties. We recall that Lemma \ref{lemma:phiproperties} applied to a null contact structure $(g,\alpha)$ implies:
\begin{equation*}
\phi^2 = -\xi\otimes \alpha\, ,
\end{equation*}

\noindent
from which the following lemma follows.

\begin{lemma}
\label{lemma:phicubiczero}
The characteristic endomorphism $\phi$ of a null contact metric structure $(g,\alpha)$ satisfies
\begin{equation*}
\phi^3 = 0\, .
\end{equation*}
	
\noindent
Hence $\phi$ is nilpotent.
\end{lemma}

\noindent
In the null contact case, the tensor field $\mathfrak{h}\eqdef\mathcal{L}_\xi \phi$ satisfies additional properties not listed in Proposition \ref{prop:equationscontact} (compare also with Remark \ref{remark:covxi}).

\begin{lemma}
\label{lemma:hphi0}
Let $(g,\alpha)$ be a null contact structure. Then $\phi \circ\mathfrak{h}=\mathfrak{h}\circ\phi=0$.
\end{lemma}
 
\begin{proof}
Let $v_1 , v_2 \in TM$. Combining the first and fifth equations in Lemma \ref{lemma:phiproperties}, we have
\begin{equation*}
\dd \alpha(v_1,\phi(v_2)) = -\alpha(v_1) \alpha(v_2)\,.
\end{equation*}
Applying the Lie derivative along the Reeb vector field $\xi$ to both sides of the previous equation, we obtain:
\begin{equation*}
\dd \alpha( \cL_{\xi} v_1, \phi(v_2))+\dd \alpha(v_1,\mathfrak{h}(v_2))+\dd \alpha(v_1, \phi(\cL_{\xi}v_2))=-\alpha(\cL_{\xi}v_1)\alpha(v_2)-\alpha(v_1)\alpha(\cL_{\xi}v_2)\, ,
\end{equation*}

\noindent
whence
\begin{equation*}
\dd \alpha(v_1,\mathfrak{h}(v_2))=0\, ,
\end{equation*}
which by the first equation in Lemma \ref{lemma:phiproperties} is equivalent to
\begin{equation*}
g(v_1,\phi\circ\mathfrak{h}(v_2))=0\, ,
\end{equation*}

\noindent
for every $v_1 , v_2 \in TM$. Hence $\phi\circ \mathfrak{h}=0$, which, combined with Proposition \ref{prop:eqfh} implies  $\mathfrak{h}\circ\phi=0$ and we conclude.
\end{proof}

\begin{prop}
\label{prop:hnulo}
The tensor field $\mathfrak{h}$ associated to a null contact structure $(g,\alpha)$ over $M$ can always be written as:
\begin{equation*}
\mathfrak{h}=\mu\, \xi\otimes \alpha \, ,
\end{equation*}

\noindent
for a function $\mu \in C^{\infty}(M)$. 
\end{prop}

\begin{remark}
\label{remark:hnull}
The previous Proposition implies that $\mathfrak{h}$ cannot have non-zero eigenvalues, in notable contrast with the situation occurring when $\varepsilon \neq 0$ \cite{Blair,Calvaruso,Zankoyov}, where certain eigenbundles of $\mathfrak{h}$ with non-zero eigenvalue play a crucial role in the classification of $\varepsilon\eta\,$-Einstein and $(\kappa,\mu)$ contact metric three-manifolds \cite{Blair,Calvaruso3d}.  
\end{remark}

\begin{proof}
Choose a light-cone frame $\{\xi, u, \phi(u)\}$. By Lemma \ref{prop:equationscontact}, $\phi(\xi)=0$ hence $\xi \in \ker(\phi)$. By Remark \ref{remark:lcb} we have $g(\phi(u), \phi(u) )=1$, whence $u\notin \ker(\phi)$. Furthermore, $\phi^2 (u)= -\xi$, so $\phi (u)\notin \ker(\phi)$ and we conclude that $\ker(\phi)=\text{Span}_{C^\infty}\left\{ \xi\right\}$. By Lemma \ref{lemma:hphi0} we have $\phi\circ\mathfrak{h}=0$, and thus $\Im(\mathfrak{h}) \subseteq \ker(\phi)$, whence
\begin{equation*}
\mathfrak{h} = \gamma\otimes \xi\, 
\end{equation*}

\noindent
for a certain one-form $\gamma$. Since $\mathfrak{h}(\xi) = 0$ (see Proposition \ref{prop:equationscontact}) we conclude that $\gamma = \mu\, \alpha + c \,\phi(u)^{\flat}$ for some functions $\mu$ and $c$. On the other hand, by Lemma \ref{lemma:hphi0} we have $\mathfrak{h}\circ \phi = 0$, which is equivalent to $\gamma \circ \phi = 0$. This implies that $c = 0$ after evaluating at $u$, that is, imposing $\gamma(\phi(u)) = 0$.
\end{proof}

\begin{remark}
Expressed in a light-cone frame $\left\{\xi,u,\phi(u)\right\}$, the endomorphism $\mathfrak{h}$ has the following point-wise matrix form:
	\[
\mathfrak{h} =
\left[ {\begin{array}{ccc}
	0 & \mu & 0 \\
	0 & 0 & 0  \\
	0 & 0 & 0  \\
	\end{array} } \right],
\]

\noindent
whence it has a unique type as an endomorphism of the tangent space. This shall be compared with the results of \cite{CalvarusoPC} where, in the para-contact case, the possible types of $\mathfrak{h}$ in an orthonormal special basis were classified.
\end{remark}

\noindent 
The additional properties satisfied by the tensor field $\mathfrak{h}$ allows us to obtain an explicit expression for the Lie brackets of a light-cone frame $\{\xi, u ,\phi (u) \}$.

\begin{lemma}
Let $(g,\alpha)$ be a null contact structure over $M$. For a light-cone frame $\{\xi, u ,\phi (u) \}$ we have:
\begin{equation*}
[\xi, u]=b\,\xi+c\,\phi(u)\, , \quad [\xi,\phi (u)]=(\mu-c)\,\xi \, , \quad [u,\phi(u)]=e\,\xi+ u  +f\,\phi(u),
\end{equation*}
where $\mu=g(u,\mathfrak{h}(u))$ and $b,c,e$ and $f$ are local functions.
\label{lemma:lblcf}
\end{lemma}

\begin{proof}
Equation  $\mathcal{L}_\xi \alpha=0$ implies that $\alpha([\xi, u])=0$, whence $[\xi, u]= b\,\xi+c\,\phi(u)$. Therefore:
\begin{equation*}
[\xi, \phi(u)]=\phi ([\xi, u])+\mathfrak{h}(u)=-c\,\xi+\mu\,\xi=(\mu-c)\,\xi\,.
\end{equation*}
Furthermore, by the $\varepsilon\,$-contact structure condition $\ast \alpha=-\dd \alpha$, we deduce that $\alpha([u, \phi (u)])=1$, so necessarily $ [u,\phi(u)]=e\,\xi+u+ f\,\phi(u)$. 
\end{proof}
 
\noindent
The following examples show that null contact metric structures are abundant.

\begin{ep}
	\label{ep:examplenull1}
	Take $M = \mathbb{R}^3$ and fix $g = \delta$ to be the standard Minkowski metric of signature $(-,+,+)$ in $\mathbb{R}^{3}$ with orthonormal coordinates $(t,x,y)$, $t$ being time-like. Then: 
	\begin{equation*}
	\delta = -\dd t\otimes \dd t + \dd x\otimes \dd x + \dd y\otimes \dd y\, , \qquad \alpha = e^{y}\,(\dd t - \dd x)\, 
	\end{equation*}
	is a null contact metric structure. The characteristic endomorphism of this null contact structure can be found to be
	\begin{equation*}
	\phi =  e^{y}\,\partial_y\otimes \dd t - e^{y}\,\partial_y\otimes \dd x + e^{y}\,(\partial_t + \partial_x)\otimes \dd y\, .
	\end{equation*}
	
	\noindent
	A direct computation shows that:
	\begin{equation*}
	\phi^2 = e^{2y} (\partial_t + \partial_x) \otimes \dd t - e^{2y} (\partial_t + \partial_x) \otimes \dd x = - \xi\otimes \alpha \, , \qquad \phi^3 = 0\, ,
	\end{equation*}
	
	\noindent
	as expected. The Reeb vector field is given by $\xi = \alpha^{\sharp} = - e^y\, (\partial_t + \partial_x)$ and we have
	\begin{equation*}
	\cL_{\xi} \delta = e^{y}\, (\dd t\odot \dd y - \dd x \odot\dd y )\, .
	\end{equation*}
	
	\noindent
	Hence $(\delta,\alpha)$ is not K-contact, and by Proposition \ref{prop:KSasakian} below, it cannot be Sasakian either. On the other hand, the endomorphism $\mathfrak{h}$ can be found to be:
	\begin{equation*}
	\mathfrak{h} \eqdef \cL_{\xi}\phi = -\xi \otimes \alpha\, ,
	\end{equation*}
	
	\noindent
	in agreement with Proposition \ref{prop:hnulo} with $\mu = - 1$.
\end{ep}

\begin{ep}
	\label{ep:Sl(2,R)}
	Consider the $M=\widetilde{\text{Sl}}(2,\mathbb{R})$, where $\widetilde{\text{Sl}}(2,\mathbb{R})$ denotes the universal cover group of $\text{Sl}(2,\mathbb{R})$. There exists a Lorentzian metric $g$ with left-invariant orthonormal global frame $\left\{ e^{0}, e^{1}, e^2\right\}$ on $\widetilde{\text{Sl}}(2,\mathbb{R})$ satisfying \cite{Rahmani}:
	\begin{equation*}
	\dd e^{0} = e^{1}\wedge e^2\, , \qquad \dd e^{1} =  e^{0}\wedge e^2 \, , \qquad \dd e^2 = - e^{0}\wedge e^{1}\, 
	\end{equation*}
	
	\noindent
	and whose associated Lorentzian metric $g$ is given by:
	\begin{equation*}
	g = - e^{0}\otimes e^{0} + e^{1}\otimes e^{1} + e^2\otimes e^2\, .
	\end{equation*}
	
	\noindent
	Hence $\ast e^{0} = - e^{1}\wedge e^2$, $\ast e^{1} = - e^{0}\wedge e^{2}$ and $\ast e^{2} = e^{0}\wedge e^{1}$. We take $\alpha$ to be a null and left-invariant one-form on $M$ and we expand it the basis $\left\{ e^{0}, e^{1}, e^2\right\}$,
	\begin{equation*}
	\alpha = \alpha_{0} e^{0} + \alpha_{1} e^{1} +\alpha_{2} e^{2}\, , 
	\end{equation*}
	
	\noindent
	for some real constant coefficients $\left\{\alpha_a\right\}$, $a = 0,1,2$, satisfying $\alpha_0^2 = \alpha_1^2 + \alpha_2^2$. We compute:
	\begin{equation*}
	\dd\alpha = \alpha_{0} e^{1}\wedge e^2 + \alpha_{1}  e^{0}\wedge e^2 - \alpha_{2}  e^{0}\wedge e^1  = -\ast\alpha\, .
	\end{equation*}
	
	\noindent
	Therefore, every null left-invariant one-form on $\widetilde{\text{Sl}}(2,\mathbb{R})$ equipped with the Lorentzian metric $g$ defines a null contact structure $(g,\alpha)$. We obtain now the characteristic endomorphism $\phi$ of $(g,\alpha)$ as well as its square. A direct computation yields the following matrix representation for $\phi$ in the basis $\left\{ e^0, e^1, e^2 \right\}$:
	
	\[
	\phi =
	\left[ {\begin{array}{ccc}
		0 & \alpha_2 & - \alpha_1 \\
		\alpha_2 & 0 & \alpha_0  \\
		-\alpha_1 & -\alpha_0 & 0  \\
		\end{array} } \right]
	\]
	
	\noindent
	For the square and cubic powers of $\phi$ we obtain:
	\[
	\phi^2 =
	\left[ {\begin{array}{ccc}
		\alpha_0^2 & \alpha_0\alpha_1 & \alpha_0 \alpha_2 \\
		-\alpha_0 \alpha_1 & -\alpha_1^2 & -\alpha_1 \alpha_2  \\
		-\alpha_0 \alpha_2 & -\alpha_1\alpha_2 & -\alpha_2^2  \\
		\end{array} } \right] = -\left[\xi\otimes \alpha\right]\, , \qquad \phi^3 = 0\, ,
	\]
	
	\noindent
	as expected from Lemma \ref{lemma:phicubiczero}.
\end{ep}

\begin{remark}
By the results of \cite{DumitrescuZeghib,Frances}, see in particular Theorem A in \cite{Frances}, Examples \ref{ep:examplenull1} and \ref{ep:Sl(2,R)} show that all closed Lorentzian three-manifolds admitting a non-compact isometry group carry null contact structures. 
\end{remark}

\noindent
We generalize now Example \ref{ep:Sl(2,R)} by classifying all simply-connected, Lorentzian, three-dimensional Lie groups admitting left-invariant null contact structures. For this, we will exploit the classification of this type of groups available in the literature \cite{Calvaruso3d,CorderoParker,Rahmani}, which we summarize for completeness in Appendix \ref{app:Lorentzgroups}.

\begin{prop}
\label{prop:nullcontactstructures}
A three-dimensional connected and simply connected Lie group $\mathrm{G}$ admits a left-invariant null contact structure $(g,\alpha)$ if and only if $(\mathrm{G},g,\alpha)$ is isomorphic, through a possibly orientation-reversing isometry, to one of the items listed in the following table in terms of the orthonormal frame $\left\{ e_0 , e_1 , e_2\right\}$ appearing in Theorem \ref{thm:LorentzGroups}:

\vspace{0.25cm}
 \renewcommand{\arraystretch}{1.5}
\begin{tabular}{ |P{2cm}|P{4.5cm}|P{4.5cm}|P{2cm}|  }
				\hline
			
				$\mathfrak{g}$  & \emph{Structure constants} $(s\in \mathbb{Z}_2)$  & $\alpha$ &  $\mathrm{G}$  \\
				\hline
				  $\mathfrak{g}_1$& $a \neq 0\, , \, b= s $ & $\alpha_1=0\, , \, \alpha_0=-\alpha_2$  & $\widetilde{\mathrm{Sl}}(2,\mathbb{R})$ \\ \hline

				 \multirow{4}*{$\mathfrak{g}_3$} & $a\neq 0\, ,\, b=c=s$ & \multirow{1}*{$\alpha_1=0\, ,\, \alpha_0^2=\alpha_2^2$} & $\widetilde{\mathrm{Sl}}(2,\mathbb{R})$   \\ \cline{2-2} \cline{3-3} \cline{4-4}&  $a=c=s\, ,b\neq 0$ &  \multirow{2}*{$\alpha_2=0\, , \,\alpha_0^2=\alpha_1^2$}  &  $\widetilde{\mathrm{Sl}}(2,\mathbb{R})$ \\ \cline{2-2}   \cline{4-4} &  $a=s\, , \, b =0\, ,\, c=s$  &  & $\widetilde{\mathrm{E}}(1,1)$ \\  \cline{2-2} \cline{3-3}  \cline{4-4}&  $a=b=c=s$ & $\alpha_0^2=\alpha_1^2+\alpha_2^2\, , \alpha_1, \alpha_2 \neq 0$ &  $\widetilde{\mathrm{Sl}}(2,\mathbb{R})$\\ \hline 
			 \multirow{2}*{$\mathfrak{g}_4$}    & $b=s + \mu \, ,\, a \neq 0\,  , \,\mu \in \mathbb{Z}_2$ & \multirow{2}*{$\alpha_1=0\, , \,\alpha_0=\mu\alpha_2$} & $\widetilde{\mathrm{Sl}}(2,\mathbb{R})$ \\ \cline{2-2} \cline{4-4} & $b=s + \mu\, , a=0\,  , \mu \in \mathbb{Z}_2$  & & $\widetilde{\mathrm{E}}(1,1)$\\	 
			 \hline
				\multirow{2}*{$\mathfrak{g}_6$}  & $\mu a = (b+s)\, , \, \mu d = (c+s)$  & \multirow{2}*{$\alpha_1=0\, , \, \alpha_0= -\mu \alpha_2 $}   & \multirow{2}*{$\mathfrak{G}_6$}\\  &  $b=c\, , \, a=d\neq 0$ & & \\  \hline
\end{tabular}
\renewcommand{\arraystretch}{1} 	
\end{prop}

\begin{proof}
Every connected and simply-connected three-dimensional Lorentzian Lie group is isometric to one of the items listed in Theorem \ref{thm:LorentzGroups} through an isometry which may be orientation-reversing. Therefore we proceed on a case by case basis by solving the equation
\begin{equation*}
\ast \alpha = -s \,\dd\alpha\, , \qquad s\in\left\{ - 1 , + 1\right\}
\end{equation*}

\noindent
in each of the cases listed in the table appearing in Theorem \ref{thm:LorentzGroups}. The sign $s$ is introduced because the $\varepsilon\,$-contact condition is not invariant under orientation reversing morphisms. If a solution is found with $s = -1$ then an orientation reversing isometry yields an isomorphic Lorentzian Lie algebra admitting a null-contact structure satisfying equation $\ast \alpha = - \dd\alpha$. Let $\{e_0, e_1, e_2\}$ be the frame appearing in Theorem \ref{thm:LorentzGroups} and let $\{e^0, e^1, e^2\}$ be its dual frame. We fix the volume form to be $\nu = e^0\wedge e^1 \wedge e^2$. For the rest of the proof we will write:
\begin{equation*}
\alpha = \alpha_0 \, e^0 + \alpha_1 \, e^1 + \alpha_2 \, e^2\, ,
\end{equation*}

\noindent
where $\alpha_0$, $\alpha_1$ and $\alpha_2$ are real coefficients satisfying $\alpha_0^2 = \alpha_1^2 + \alpha_2^2$ (hence $\alpha_0 \neq 0$ in order for $\alpha$ to be non-zero). If $\{e^0, e^1, e^2\}$ denotes the dual co-frame, we compute:
\begin{equation*}
\ast \alpha=-\alpha_0 e^1 \wedge e^2 - \alpha_1 e^0\wedge e^2 +\alpha_2 e^0 \wedge e^1 \,. 
\end{equation*}

\begin{itemize}
	\item {\bf Case} $\mathfrak{g}_1$. In the orthonormal frame for $\mathfrak{g}_1$ given in Appendix \ref{app:Lorentzgroups}, Equation $\ast \alpha = -s\,\dd\alpha$ is equivalent to
	
	\begin{eqnarray*}
    &(-a \alpha_1 - b \alpha_2)\, e^0\wedge e^1 + (a \alpha_0 +b \alpha_1 + a \alpha_2)\, e^0\wedge e^2 + (b \alpha_0 - a \alpha_1)\, e^1\wedge e^2 \\
    & = s\, ( \alpha_0\, e^1\wedge e^2 + \alpha_1\, e^0\wedge e^2  - \alpha_2\, e^0\wedge e^1)\, ,
    \end{eqnarray*}
  \noindent  
    A solution $(a,b)$ exists for a non-zero $\alpha$ if and only if $b=s$. Using that $a \neq 0$ by Theorem \ref{app:Lorentzgroups}, if $b=s$ we obtain $\alpha_0=-\alpha_2$, so $\alpha_1=0$.

	\item {\bf Case} $\mathfrak{g}_2$. In the orthonormal frame for $\mathfrak{g}_2$ given in Appendix \ref{app:Lorentzgroups}, Equation $\ast \alpha = -s\,\dd\alpha$ is equivalent to
	\begin{eqnarray*}
	& a \alpha_1 e^0 \wedge e^2+(c \alpha_0-b \alpha_2)\,  e^0 \wedge e^1 + (c \alpha_2 +b \alpha_0)\,  e^1 \wedge e^2\\
	& = s \alpha_0 \, e^1 \wedge e^2 +s \alpha_1 \, e^0 \wedge e^2-s \alpha_2\, e^0 \wedge e^1\, .
	\end{eqnarray*}
	\noindent
	The previous equations imply the condition $c^2+(b-s)^2=0$, whose unique solution is $c=0$ and $b=s$. However the value $c=0$ is forbidden for $\mathfrak{g}_2$ algebras, so no null contact structures exist on this type of algebras. 
	
	\item {\bf Case} $\mathfrak{g}_3$. In the orthonormal frame for $\mathfrak{g}_3$ given in Appendix \ref{app:Lorentzgroups}, Equation $\ast \alpha = -s\,\dd\alpha$ is equivalent to
	\begin{equation*}
	c \alpha_0\,  e^1\wedge e^2+a \alpha_1 \, e^0 \wedge e^2- b \alpha_2 \, e^0 \wedge e^1=s\alpha_0 e^1 \wedge e^2+ s\alpha_1 e^0 \wedge e^2 -s\alpha_2 e^0 \wedge e^1\,. 
	\end{equation*}
	\noindent
	Since $\alpha_0 \neq 0$ if $\alpha$ is non-vanishing, we have $c = s$. If $\alpha_1=0$, then $\alpha_0^2= \alpha_2^2$ and $b= s$. If $\alpha_2=0$, then $\alpha_0^2= \alpha_1^2$ and $a=s$. If $\alpha_1, \alpha_2 \neq 0$, then $a=b=s$.

	\item {\bf Case} $\mathfrak{g}_4$. In the orthonormal frame for $\mathfrak{g}_4$ given in Appendix \ref{app:Lorentzgroups}, Equation $\ast \alpha = -s\,\dd\alpha$ is equivalent to
	\begin{eqnarray*}
     & a \alpha_1 \, e^0 \wedge e^2 +(-(2\mu-b)\alpha_0+\alpha_2) e^1 \wedge e^2+(\alpha_0-b \alpha_2 ) \, e^0 \wedge e^1\\
     & =s \alpha_0 \, e^1 \wedge e^2+s \alpha_1\, e^0 \wedge e^2-s \alpha_2\, e^0 \wedge e^1 \, .
	\end{eqnarray*}
	
\noindent 	
This equation admits non-trivial solutions only if $\alpha_2 \neq 0$, $b\neq s$ (which is required in order to have $\alpha_0\neq 0$) and
\begin{equation*}
b^2-2b(s+\mu)+2(s\mu+ 1)=0\, .
\end{equation*}

\noindent
The previous equation has the unique solution $b=\mu + s$, which satisfies $b\neq s$. Setting $b=\mu + s$ we obtain $\alpha_0 = \mu \alpha_2$, whence $\alpha_1=0$. 
	
\item  {\bf Case} $\mathfrak{g}_5$. In the orthonormal frame for $\mathfrak{g}_5$ given in Appendix \ref{app:Lorentzgroups}, Equation $\ast \alpha = -s\,\dd\alpha$ is equivalent to
\begin{equation*}
(a \alpha_1 + b \alpha_2 )\, e^0\wedge e^1 + (c\alpha_1  + d \alpha_2 )\, e^0\wedge e^2 = s\,(\alpha_0\, e^1\wedge e^2 + \alpha_1\, e^0\wedge e^2  - \alpha_2\, e^0\wedge e^1)\, .
\end{equation*}

\noindent	 
This equation implies $\alpha_0=0$, which in turn yields $\alpha=0$. Therefore, $\mathfrak{g}_5$ does not admit null contact structures. 
	
\item {\bf Case} $\mathfrak{g}_6$. In the orthonormal frame for $\mathfrak{g}_6$ given in Appendix \ref{app:Lorentzgroups}, Equation $\ast \alpha = -s\,\dd\alpha$ is equivalent to
\begin{equation*}
(d\alpha_0+c\alpha_2) \, e^0 \wedge e^1 +(-b \alpha_0 -a\alpha_2) \, e^1 \wedge e^2 = s\,(\alpha_0 \, e^1 \wedge e^2 +\alpha_1 \, e^0 \wedge e^2-\alpha_2 \, e^0 \wedge e^1)\, .
\end{equation*}

\noindent	
We obtain $\alpha_1=0$, so that $\alpha_0= - \mu \alpha_2$ for $\mu \in \mathbb{Z}_2$. Consequently, $\mu a = s+b$ and $\mu d = s+c$. The remaining conditions follow from equations $ac = bd$ and $a+d \neq 0$, which must hold by Theorem \ref{thm:LorentzGroups}.
	
\item {\bf Case} $\mathfrak{g}_7$. In the orthonormal frame for $\mathfrak{g}_7$ given in Appendix \ref{app:Lorentzgroups}, Equation $\ast \alpha = -s\,\dd\alpha$ is equivalent to
\begin{equation*}
\begin{split}
(b \alpha_0+a \alpha_1+b \alpha_2) \, e^0 \wedge e^1&+ (d\alpha_0+c \alpha_1 +d \alpha_2) \, e^0 \wedge e^2 +(b \alpha_0+a \alpha_1+b \alpha_2) \, e^1 \wedge e^2\\&= s(\alpha_0 e^1 \wedge e^2+\alpha_1 e^0 \wedge e^2 -\alpha_2 e^0 \wedge e^1) \, .
\end{split}
\end{equation*}

\noindent
with $a+d\neq 0$ and $ac = 0$. Imposing $ac = 0$ implies $\alpha_0=0$, whence a Lorentzian Lie group with Lie algebra isomorphic to $\mathfrak{g}_7$ does not admit non-trivial null contact structures.
\end{itemize}
\end{proof}

\subsection{Sasakian null contact structures}

%%%%%%%%%%%%%%%%%%%%%%%%%%%%%%%%%%%%%%%%%%%%%%%%%%%%%%%%%%%%%%%%%%%%%%%%%%%%%%%%%%%%%%%%%%%%%%%%%%%%%%%%%%%%%%%%%%%%%%%%%%%%%%
%%%%%%%%%%%%%%%%%%%%%%%%%%%%%%%%%%%%%%%%%%%%%%%%%%%%%%%%%%%%%%%%%%%%%%%%%%%%%%%%%%%%%%%%%%%%%%%%%%%%%%%%%%%%%%%%%%%%%%%%%%%%%%

The Sasakian condition of a Riemannian or Lorentzian (para) contact-metric structure with non-null Reeb vector field can be defined in terms of the integrability of certain endomorphism defined from $\phi$ on $T(M\times \mathbb{R})$ \cite{Blair,Calvaruso3d,ACGLT}. In the null case $\varepsilon = 0$, we proceed in analogous way and we introduce the following endomorphism, which mimics the definition occurring in the case $\varepsilon \neq 0$:

%We explore now to which extent this still holds in the null case $\varepsilon = 0$. In order to do this, we introduce the following endomorphism, which mimics the definition occurring in the case $\varepsilon \neq 0$:
\begin{equation*}
J: T(M \times \mathbb{R}) \rightarrow T(M\times \mathbb{R})\, ,  \qquad (v, c\,\partial_q) \mapsto (\phi(v) + c\, \xi, \alpha(v) \partial_q)\, ,
\end{equation*}

\noindent
where $q$ is the fixed canonical coordinate on $\mathbb{R}$ and $c\in \mathbb{R}$. A direct computation shows that $J^2=0$.  We clarify first the relevant notion of \emph{integrability} of a field of endomorphisms.

%We will prove that the \emph{integrability} of $J$ is equivalent to $(g,\alpha)$ is being a Sasakian null contact structure.

\begin{definition}
Let $E\in \Gamma(TN\otimes T^{\ast}N)$ be a field of endomorphisms on a manifold $N$. Then, $E$ is said to be {\bf integrable} if around every point $n\in N$ there exists a coordinate system on which the local matrix representation of $E$ has constant coefficients.
\end{definition}

\noindent
For almost complex structures the previous definition is equivalent to the standard definition of integrability in terms of existence of holomorphic charts. Necessary and sufficient conditions for a nilpotent endomorphism (such as $J$) to be integrable have been studied in the literature, see \cite{Kobayashi,Lejeune,Turiel,Thompson,Boubel} and \cite{BKM} for a thorough exposition of this and related topics. By a result of G. Thomson \cite[Theorem 2]{Thompson}, we have that $J\in \Gamma(\End(T(M\times\mathbb{R})))$ is integrable if and only if the following three conditions hold simultaneously:
\begin{itemize}
	\item The Nijenhuis torsion tensor of $J$ vanishes.
	
	\item $J$ is a zero-deformable field of endomorphisms. %The invariant factors of $J$ are constant.
	
	\item The distribution $\mathrm{Ker}(J)\subset T(M\times \mathbb{R})$ is involutive.
\end{itemize}

\noindent
Recall that the Nijenhuis torsion tensor associated to $J$ is defined as
\begin{equation*}
\cN_J(v_1,v_2) = [J(v_1),J(v_2)] - J[v_1,J(v_2)] - J[J(v_1),v_2] + J^2[v_1,v_2]\, .
\end{equation*}

\noindent
Note that since $J^2 = 0$, the last term in the right hand side identically vanishes. Also, we remind the reader that a certain field of endomorphisms  is said to be zero-deformable if around every point of $M$ there exists a frame relative to which the Jordan form of this endomorphism field is constant.  

\begin{lemma}
\label{lemma:constantfactors}
The endomorphism $J\in \Gamma(\End(T(M\times \mathbb{R})))$ is zero-deformable.%has constant invariant factors.
\end{lemma}

\begin{proof}
Fix a point in $M\times \mathbb{R}$ and consider the tangent-space basis $\left\{\xi,u,\phi(u),\partial_q\right\}$, where $\left\{\xi,u,\phi(u)\right\}$ is a light-cone basis. In this basis $J$ has the following matrix representation:
 \[
J =
\left[ {\begin{array}{cccc}
	0 & 0 & -1 & 1 \\
	0 & 0 & 0 & 0 \\
	0 & 1 & 0 & 0 \\
	0 & 1 & 0 & 0 \\
	\end{array} } \right].
\]

\noindent
Since the basis $(\xi,u,\phi(u),\partial_q)$ exists at every point in $M\times \mathbb{R}$, we conclude that $J$ is always locally conjugate to the same constant Jordan form.
\end{proof}
  
\begin{lemma}
\label{lemma:kerJinvolutive}
The distribution $\mathrm{Ker}(J)\subset T(M\times \mathbb{R})$ is involutive.
\end{lemma}
\begin{proof}
Fix a local frame $\left\{ \xi,u,\phi(u),\partial_q\right\}$ on $M\times \mathbb{R}$, where $\{ \xi,u,\phi(u)\}$ is a light-cone frame for $M$. The kernel of $J$ is locally spanned by
\begin{equation*}
\mathrm{Ker}(J) = \mathrm{Span}_{C^{\infty}}(\xi,\phi(u) + \partial_q)\, .
\end{equation*}

\noindent
Lemma \ref{lemma:lblcf} implies now:
\begin{equation*}
[\xi,\phi(u) + \partial_q] = [\xi,\phi(u)] = (\mu-c)\, \xi\, ,
\end{equation*} 

\noindent
and we conclude.
\end{proof}
 
\begin{prop}
A null contact metric structure $(g,\alpha)$ is Sasakian if and only if its associated endomorphism $J\colon T(M\times\mathbb{R})\to T(M\times \mathbb{R})$ is integrable. 
\end{prop}
 
\begin{proof}
Assume first that $(g,\alpha)$ is Sasakian, that is, $\mathfrak{h} = 0$. By Lemma \ref{lemma:constantfactors} the endomorphism $J$ is zero-deformable and by Lemma \ref{lemma:kerJinvolutive} its kernel $\mathrm{Ker}(J)$ is involutive. Therefore by Theorem \cite[Theorem 2]{Thompson} we only need to show that $\cN_J$ vanishes to prove that $J$ is integrable.  Since $\cN_J$ is a tensor, it is enough to prove that it vanishes on a light cone frame $\left\{ \xi,u,\phi(u),\partial_q\right\}$. We compute:
\begin{eqnarray*}
&\cN_J(\xi,u) = - J[\xi,\phi(u)] = 0\, , \quad \cN_J(\xi,\phi(u)) = - J[\xi,J(\phi(u))] = 0\, , \\ 
&\cN_J(u,\phi(u)) = - J[J(\phi(u)),J(\phi(u))] = 0\, , \quad  \cN_J(\xi,\partial_q) = -J[\xi,J(\partial_q)] = 0\, ,\\
&\cN_J(u,\partial_q) =  [\phi(u),\xi] - J[u,\xi] = - \cL_{\xi}(\phi(u)) + J(\cL_{\xi}u) = - \cL_{\xi}(\phi(u)) + \phi(\cL_{\xi}u) = -\mathfrak{h}(u) = 0\, , \\
&\cN_J(\phi(u),\partial_q) = [J(\phi(u)), J(\partial_q)] - J[\phi(u), J(\partial_q)] = [\phi^2(u), \xi] - J[\phi(u), \xi] = \phi^2(\cL_{\xi} u) = 0\, ,
\end{eqnarray*}

\noindent
whence $J$ is integrable. The converse follows now directly by applying the previous formulae, upon use of Lemmas \ref{lemma:constantfactors} and \ref{lemma:kerJinvolutive} and the fact that $\mathfrak{h}(u) = 0$ if and only if $\mathfrak{h}=0$.
\end{proof}

\noindent

We proceed now to classify all left-invariant Sasakian null contact structures on simply-connected, Lorenztian, three-dimensional Lie groups by exploiting Proposition \ref{prop:nullcontactstructures}. 

\begin{prop}
\label{prop:liegroupsasakinc}
A three-dimensional connected and simply connected Lie group $\mathrm{G}$ admits a left-invariant Sasakian null contact structure $(g, \alpha)$ if and only if $(\mathrm{G},g,\alpha)$ is isomorphic, through a possibly orientation-reversing isometry, to one of the items listed in the following table in terms of the orthonormal frame $\left\{ e_0 , e_1 , e_2\right\}$ appearing in Theorem \ref{thm:LorentzGroups}:

\vspace{0.25cm}
 \renewcommand{\arraystretch}{1.5}
\begin{tabular}{ |P{2cm}|P{4.8cm}|P{4.5cm}|P{2cm}|  }
				\hline
			
				$\mathfrak{g}$  & \emph{Structure constants} $(s\in \mathbb{Z}_2)$  &  $\alpha$ &  $\mathrm{G}$  \\
				\hline
				  $\mathfrak{g}_1$& $a \neq 0\, , \, b= s $ & $\alpha_1=0\, , \,\alpha_0=-\alpha_2$  & $\widetilde{\mathrm{Sl}}(2,\mathbb{R})$ \\ \hline

				 $\mathfrak{g}_3$ & $a=b=c=s$ & $ \alpha_0^2=\alpha_1^2+\alpha_2^2$ & $\widetilde{\mathrm{Sl}}(2,\mathbb{R})$  \\  \hline 
			$\mathfrak{g}_4$    & $b= s + \mu \, , \, a = s\,  , \, \mu \in \mathbb{Z}_2$ & $\alpha_1=0\, , \, \alpha_0=\mu\alpha_2$ & $\widetilde{\mathrm{Sl}}(2,\mathbb{R})$ \\ 
			 \hline
				
				$\mathfrak{g}_6$  & $b=c=-\frac{s}{2}\, , a=d= \frac{\mu s}{2}\, ,$  & $\alpha_1=0\, ,$ $\alpha_0= - \mu \alpha_2 $   & $\mathfrak{G}_6$   \\  \hline
				
				%\multirow{2}*{$\mathfrak{g}_6$}  & $c=\pm a\, , d=-a \mp 1\, ,$  & \multirow{2}*{$\alpha_1=0\, ,$ $\alpha_0=\pm \alpha_2 $}   & \\  &  $b=\mp a-1 \, , b \neq -\frac{1}{2}$ & & \\  \hline
				
			\end{tabular}
			
			 \renewcommand{\arraystretch}{1}
\end{prop}

\begin{proof}
We proceed by verifying which cases in Proposition \ref{prop:nullcontactstructures} satisfy $\mathfrak{h} = 0$. By Proposition \ref{prop:hnulo}, it is enough to prove that $\mathfrak{h}(u)=0$, where $u\in \left\{\xi,u,\phi(u)\right\}$ belongs to a light-cone frame. In the following computations it will be very convenient to use that the matrix expression of the characteristic endomorphism $\phi$ in the orthonormal basis $\{e_0,e_1,e_2\}$ used at Theorem \ref{thm:LorentzGroups} is given by
\begin{equation*}
\phi =
	\left[ {\begin{array}{ccc}
		0 & \alpha_2 & - \alpha_1 \\
		\alpha_2 & 0 & \alpha_0  \\
		-\alpha_1 & -\alpha_0 & 0  \\
		\end{array} } \right].
\end{equation*}

\noindent
Furthermore, the Reeb vector field $\xi$ is given by $\xi=-\alpha_0 e_0+\alpha_1 e_1+\alpha_2 e_2$.

\begin{itemize}
\item {\bf Case} $\mathfrak{g}_1$. Since $\alpha_1=0$ by Proposition  \ref{prop:nullcontactstructures}, then the Reeb vector field reads $\xi=-\alpha_0 (e_0+e_2)$. A direct computation shows that $(\xi, u = \frac{1}{2 \alpha_0} (e_0-e_2), -e_1)$ is a light-cone frame (where $\phi(u)=-e_1$). Hence
\begin{equation*}
\mathfrak{h}(u)= \alpha_0 [e_0,e_1] + \alpha_0 [e_2,e_1]+ \phi ([e_2,e_0])=0\, ,
\end{equation*}

\noindent
implying that every null contact structure on $\mathfrak{g}_1$ is Sasakian.

\item {\bf Case} $\mathfrak{g}_3$. According to Proposition \ref{prop:nullcontactstructures} we distinguish between the cases $\alpha_1=0$ with $\alpha_2 \neq 0$, $\alpha_2 = 0$ with $\alpha_1\neq 0$ and $\alpha_1, \alpha_2 \neq 0$.  If $\alpha_1=0$, then $\xi = \alpha_0 (-e_0 + \mu e_2)$ and $(\xi, u =\frac{1}{2 \alpha_0} (e_0 + \mu e_2), \mu e_1)$ is a light-cone frame. We obtain, after using that $b=c=s$:
\begin{equation*}
\mathfrak{h}(u)=\alpha_0 (s-a)e_0 + \mu \alpha_0 (a-s )e_2.
\end{equation*}

\noindent
Hence, $\mathfrak{h}=0$ if and only if $a=s$. Similarly, if $\alpha_2=0$,  then  $\xi=\alpha_0 (-e_0 + \mu e_1)$ and  $\{\xi, u=\frac{1}{2\alpha_0} (e_0 + \mu e_1), \mu e_2\}$ is a light-cone frame (with $\phi(u)=\mu e_2$). Since $a=c=s$, we compute:
\begin{equation*}
\mathfrak{h}(u)= \alpha_0 (s-b)e_0 + \mu\alpha_0 (b-s) e_1\, .
\end{equation*}

\noindent
Hence $\mathfrak{h} = 0$ if and only if $b=s$. The case $\alpha_1, \alpha_2 \neq 0$ follows along similar lines.

\item {\bf Case} $\mathfrak{g}_4$. In this case we have $\xi=\alpha_0 (-e_0 + \mu e_2)$ and get a light-cone frame $(\xi, u=\frac{1}{2 \alpha_0}(e_0+\mu e_2), \mu e_1)$. We obtain:
\begin{equation*}
\mathfrak{h}(u)=\alpha_0 (-(2\mu-b) + \mu-a)e_0+\alpha_0 (1 - \mu b + \mu a) e_2\, ,
\end{equation*}

\noindent
which vanishes if and only if
\begin{equation*}
1 - \mu b + \mu a =0\, .
\end{equation*}

\noindent
This is in turn equivalent to the constraint $a=s$.

\item  {\bf Case} $\mathfrak{g}_6$. In this case we have $\xi= - \alpha_0 (e_0 + \mu e_2)$ and $\{ \xi, u=\frac{1}{2\alpha_0}(e_0 - \mu e_2), -\mu e_1\}$ is a light-cone frame, where $\phi(u)= -\mu e_1$. We find:
\begin{equation*}
\mathfrak{h}(u)=-\alpha_0\mu (d+ \mu b) e_0 - \alpha_0\mu (c + \mu a) e_2\, ,
\end{equation*}

\noindent
whence $c= -\mu a$ and $d=-\mu b$. Taking into account now the constraints stated in Proposition \ref{prop:nullcontactstructures} we conclude.

\end{itemize}
\end{proof}

%%%%%%%%%%%%%%%%%%%%%%%%%%%%%%%%%%%%%%%%%%%%%%%%%%%%%%%%%%%%%%%%%%%%%%%%%%%%%%%%%%%%%%%%%%%%%%%%%%%%%%%%%%%%%%%%%%%%%%%%%%%%%%
%%%%%%%%%%%%%%%%%%%%%%%%%%%%%%%%%%%%%%%%%%%%%%%%%%%%%%%%%%%%%%%%%%%%%%%%%%%%%%%%%%%%%%%%%%%%%%%%%%%%%%%%%%%%%%%%%%%%%%%%%%%%%%

\subsection{Null K-contact structures}

%%%%%%%%%%%%%%%%%%%%%%%%%%%%%%%%%%%%%%%%%%%%%%%%%%%%%%%%%%%%%%%%%%%%%%%%%%%%%%%%%%%%%%%%%%%%%%%%%%%%%%%%%%%%%%%%%%%%%%%%%%%%%%
%%%%%%%%%%%%%%%%%%%%%%%%%%%%%%%%%%%%%%%%%%%%%%%%%%%%%%%%%%%%%%%%%%%%%%%%%%%%%%%%%%%%%%%%%%%%%%%%%%%%%%%%%%%%%%%%%%%%%%%%%%%%%%

For three-dimensional Riemannian contact structures, Lorentzian contact structures and para-contact structures, the K-contact and Sasakian conditions are equivalent. However, this fails to be the case for null contact structures, as the following Example shows.

\begin{ep}
\label{ep:SasakiannoK}
Take $M$ to be a connected and simply connected Lie group admitting a left-invariant global co-frame $\{e^{+},e^{-},e^2\}$ satisfying:
\begin{eqnarray*}
\dd e^{+} =-a\, e^{+} \wedge e^{-}  -e^{+}\wedge e^2 \, , \quad \dd e^{-}=e^{-} \wedge e^2\, , \quad \dd e^2=e^{+}\wedge e^{-} - a\, e^{-}\wedge e^2\, ,  
\end{eqnarray*}

\noindent
where $a \in \mathbb{R}\backslash\left\{ 0\right\}$. We denote by $\{e_{+},e_{-},e_2\}$ the frame dual to $\{e^{+},e^{-},e^2\}$. We equip $M$ with the Lorentzian metric
\begin{equation*}
g=e^{+} \odot e^{-} +e^2 \otimes e^2\, 
\end{equation*}

\noindent
and we fix the volume form to be $\nu =  e^{-}\wedge e^{+}\wedge e^2$. Set $\alpha \eqdef  e^{-}$, whence $\xi = e_{+}$. Then, $(g,\alpha)$ defines a null contact structure on $M$ since $g(\xi, \xi)=0$ and
\begin{equation*}
\ast \alpha=- \alpha \wedge e^2 = - \dd \alpha\, .
\end{equation*}

\noindent
The characteristic endomorphism $\phi$ is given by:
\begin{equation*}
\phi(\xi)=0\, , \quad \phi(e_{-})=(\iota_{e_{-}} \ast \alpha)^\sharp=-e_2\, , \quad \phi(e_2)=(\iota_{e_{2}} \ast \alpha)^\sharp=\xi\, .
\end{equation*}

\noindent
Therefore, $\{\xi,e_-,-e_2\}$ yields a light-cone frame. Similarly, $\phi^2$ can be shown to satisfy: 
\begin{equation*}
\phi^2(\xi)=0\, , \quad \phi^2 (e_{-})=-\phi (e_2)=-\xi\, , \quad
 \phi^2(e_2)=0\, .
\end{equation*}

\noindent
Hence we obtain $\phi^2=-\xi\otimes\alpha$ and $\phi^3=0$, as required. The fact that $(g,\alpha)$ is Sasakian follows now from the following computation:
\begin{equation*}
\mathfrak{h}(e_{-})=[\xi, \phi(e_{-})]-\phi([\xi, e_{-}])=- \xi +  \xi=0\, .
\end{equation*}

\noindent
which, together with Proposition \ref{prop:hnulo} implies that $\mathfrak{h}$ vanishes identically. On the other hand, $\xi$ cannot be a Killing vector field, because:
\begin{equation*}
(\mathcal{L}_{\xi} g)(e_{-},e_{-})=-2\, g([\xi,e_{-}],e_{-})=-2a\,g(\xi,e_{-})=-2a \neq 0\, 
\end{equation*}

\noindent
and $a\neq 0$ by assumption. Therefore $(g,\alpha)$ is a Sasakian null contact structure which fails to be K-contact.
\end{ep}

\noindent
Hence, the Sasakian condition does not imply the K-contact condition.  
 
\begin{prop}
\label{prop:KSasakian}
Every three-dimensional null K-contact structure $(g,\alpha)$ is Sasakian.
\end{prop}

\begin{proof}
Let $\xi$ denote the Reeb vector field associated to $(g,\alpha)$ and choose a light-cone frame $\{\xi,u,\phi (u)\}$. From the K-contact condition we have:
\begin{equation*}
0 = - ( \mathcal{L}_\xi g)(u,\phi (u) )=  g(\cL_{\xi}u,\phi (u) ) + g(u,\phi(\cL_{\xi}u) ) + g(u, \mathfrak{h}(u) ) =  g(u, \mathfrak{h}(u) )\, ,
\end{equation*}
where we have used that $\phi$ is skew-adjoint with respect to $g$. On the other hand, Corollary \ref{prop:hnulo} implies that $\mathfrak{h}(u)=\mu \, \xi$ for some function $\mu \in C^\infty(M)$. Since $g(u,\xi) = 1$, equation $g(u, \mathfrak{h}(u)) = 0$ is equivalent to $\mu = 0$ and thus $\mathfrak{h}=0$.
\end{proof}
 
\begin{prop}
\label{prop:wsikc}
Let $\xi$ denote the Reeb vector field associated to a Sasakian null-contact structure $(g,\alpha)$, and let $\{\xi,u,\phi (u)\}$ be a light-cone frame. Then, $(g,\alpha)$ is K-contact if and only if:
\begin{equation*}
g(\cL_{\xi} u,u)=0\, ,
\end{equation*}
%\noindent
%where $\cL$ denotes Lie derivative.
\end{prop}

\begin{proof}
We evaluate $\cL_{\xi}g$ on a light-cone frame $\left\{ \xi, u ,\phi(u) \right\}$. A direct computation using Lemma \ref{lemma:lblcf} shows that the only non-trivial term is
\begin{equation*}
(\cL_{\xi} g)(u,u) = - 2 g(\cL_{\xi} u , u)\, 
\end{equation*}

\noindent
and we conclude.
\end{proof}

\begin{remark}
Indeed, not every Sasakian null contact structure satisfies that $g(\cL_{\xi} u,u)=0$. For instance, in Example \ref{ep:SasakiannoK}, we have that:
\begin{equation*}
g(\cL_{\xi}u,u) = g([e_{+}, e_{-}],e_{-})= a \neq 0\, ,
\end{equation*}

\noindent
as expected. Therefore we learn that the Sasakian condition for null contact structures is weaker than the K-contact condition. Interestingly enough, this is in sharp contrast to what occurs for non-null contact structures: for such structures these conditions are equivalent in three dimensions whereas in higher dimensions the Sasakian condition is stronger than the K-contact condition \cite{Blair,Calvaruso}. 
\end{remark}

\noindent
Using the classification of simply connected, Lorentzian, three-dimensional Lie groups admitting left-invariant Sasakian null contact structures presented in Proposition \ref{prop:liegroupsasakinc}, we obtain in the following an analogous classification for null K-contact structures. 

\begin{prop}
\label{prop:liegroupkcontact}
A three-dimensional connected and simply connected Lie group $\mathrm{G}$ admits a left-invariant null K-contact contact structure $(g, \alpha)$ if and only if $(\mathrm{G},g,\alpha)$ is isomorphic, through a possibly orientation-reversing isometry, to one of the items listed in the following table in terms of the orthonormal frame $\left\{ e_0 , e_1 , e_2\right\}$ appearing in Theorem \ref{thm:LorentzGroups}:

\vspace{0.25cm}
 \renewcommand{\arraystretch}{1.5}
\begin{tabular}{ |P{2cm}|P{4.8cm}|P{4.5cm}|P{2cm}|  }
				\hline
			
				$\mathfrak{g}$  & \emph{Structure constants} $(s\in \mathbb{Z}_2)$  &  $\alpha$ &  $\mathrm{G}$ \\
				\hline

				 $\mathfrak{g}_3$ & $a=b=c=s$ & $ \alpha_0^2=\alpha_1^2+\alpha_2^2$ & $\widetilde{\mathrm{Sl}}(2,\mathbb{R})$  \\  \hline 
			$\mathfrak{g}_4$    & $b = s+\mu \, , a = s\,  , \, \mu \in \mathbb{Z}_2$ & $\alpha_1=0\, , \alpha_0=\mu\alpha_2$ & $\widetilde{\mathrm{Sl}}(2,\mathbb{R})$ \\ 
			 \hline
				
				$\mathfrak{g}_6$  & $b=c=-\frac{s}{2}\, ,\, a=d= \frac{\mu s}{2}$  & $\alpha_1=0\, ,$ $\alpha_0= -\mu \alpha_2 $   & $\mathfrak{G}_6$   \\  \hline

				%\multirow{2}*{$\mathfrak{g}_6$}  & $c=\pm a\, , \pm d=b\, ,$  & \multirow{2}*{$\alpha_1=0\, ,$ $\alpha_0=\pm \alpha_2 $}   & \\  &  $c+b+1=0 \, , b \neq -\frac{1}{2}$ & & \\  \hline
				%\multirow{2}*{$\mathfrak{g}_6$}  & $c=\pm 1\, , d=-a \mp 1\, ,$  & \multirow{2}*{$\alpha_1=0\, ,$ $\alpha_0=\pm \alpha_2 $}   & \\  &  $b=\mp a-1 \, , b \neq -\frac{1}{2}$ & & \\  \hline
				
			\end{tabular}
			
			 \renewcommand{\arraystretch}{1}
\end{prop}

\begin{proof}
\noindent 
Proposition \ref{prop:KSasakian} states that every null K-contact structure is Sasakian. Therefore, we proceed by checking which cases in Proposition \ref{prop:liegroupsasakinc} are in fact K-contact. By Proposition \ref{prop:wsikc}, we will test K-contactness by verifying if $g([\xi,u],u)=0$. We use the terminology introduced in Propositions \ref{prop:liegroupsasakinc} and \ref{prop:KSasakian}:

\begin{itemize}
\item {\bf Case} $\mathfrak{g}_1$. We obtain $[\xi, u]=[e_0,e_2]=-s e_1-ae_2-ae_0\neq 0$, where $a \neq 0$ by Theorem \ref{thm:LorentzGroups}. Hence:
\begin{equation*}
g([\xi, u],u) = \frac{a}{\alpha_{0}} \neq 0\, ,
\end{equation*}

\noindent
whence $\mathfrak{g}_1$ does not admit left-invariant null K-contact structures. In fact, Example \ref{ep:SasakiannoK} corresponds to a Lie group of type $\mathfrak{g}_1$.

\item {\bf Case} $\mathfrak{g}_3$. We distinguish between the cases $\alpha_1=0$ or $\alpha_2=0$ and $\alpha_1, \alpha_2 \neq 0$.  If $\alpha_1=0$, then $[\xi, u]=\mu [e_2,e_0]= \mu e_1=\phi(u)$. Hence $g([\xi, u],u) = 0$. Similarly, if $\alpha_2 =0$, then $[\xi, u]=\mu [e_1,e_0]=-\mu e_2=\phi(u)$ and $g([\xi, u],u) = 0$. For $\alpha_1, \alpha_2 \neq 0$ a direct computation shows that $(\mathcal{L}_\xi g)(e_i,e_j) = 0$ for every $i,j \in \{1,2,3\}$, whence $\mathcal{L}_\xi g=0$ and every left-invariant Sasakian null contact structure on $\mathfrak{g}_3$ is also K-contact.

\item Case $\mathfrak{g}_4$. We have $[\xi, u]= \mu [e_2,e_0]= \mu a e_1=a \phi (u)$, so $g([\xi, u],u) = 0$ and the Sasakian and K-contact conditions are equivalent on $\mathfrak{g}_4$.

\item  Case $\mathfrak{g}_6$. We have $[\xi, u]= -\mu  [e_2,e_0]=0$, so $g([\xi, u],u) = 0$. Hence the Sasakian and K-contact conditions are equivalent on  $\mathfrak{g}_6$.
\end{itemize}
\end{proof}

\noindent
The previous Proposition and Example \ref{ep:SasakiannoK} show that the theory of null K-contact structures and Sasakian null contact structures in three dimensions has the potential to be richer than its $\varepsilon\neq 0$ counterpart, where the Sasakian and K-contact conditions are equivalent. Further investigation of  this issue is beyond the scope of this manuscript.

%%%%%%%%%%%%%%%%%%%%%%%%%%%%%%%%%%%%%%%%%%%%%%%%%%%%%%%%%%%%%%%
%%%%%%%%%%%%%%%%%%%%%%%%%%%%%%%%%%%%%%%%%%%%%%%%%%%%%%%%%%%%%%%

\section{$\varepsilon\eta\,$-Einstein $\varepsilon\,$-contact metric manifolds}
\label{sec:Einsteinpcontact}

%%%%%%%%%%%%%%%%%%%%%%%%%%%%%%%%%%%%%%%%%%%%%%%%%%%%%%%%%%%%%%%
%%%%%%%%%%%%%%%%%%%%%%%%%%%%%%%%%%%%%%%%%%%%%%%%%%%%%%%%%%%%%%%

We introduce in this section the notion of $\varepsilon\eta\,$-Einstein $\varepsilon\,$-contact metric structure on an oriented three-manifold $N$, which is a particular case of the standard notion of $\eta\,$-Einstein Riemannian/Lorentzian contact metric structure when the Reeb vector field has non-vanishing norm. The definition is motivated by the structure of six-dimensional supergravity (coupled to a tensor multiplet) and its solutions, see Section \ref{sec:6dsugrasolutions} and Theorem \ref{thm:pcontactsugrasolution} for details and applications.

\begin{definition}
\label{def:Einsteinpcontact}
An $\varepsilon\,$-contact metric structure $(g,\alpha,\varepsilon)$ on $N$ is said to be {\bf $\varepsilon\eta\,$-Einstein} if the Ricci curvature tensor $\ric^g$ of $g$ satisfies
\begin{equation}
\label{eq:Einsteinpcontact}
\mathrm{Ric}^g = \frac{\mathrm{s}_g}{2}(\lambda^2 + \kappa\,\varepsilon)\,g - \mathrm{s}_g\,\kappa\, \alpha\otimes \alpha\, ,
\end{equation}
where $\mathrm{s}_g = 1$ if $g$ is Riemannian, $\mathrm{s}_g = -1$ if $g$ is Lorentzian and $\lambda,\kappa \in \mathbb{R}$ are real constants such that $\kappa \geq 0$ if $\mathrm{s}_g =-1$.
\end{definition}

\begin{remark}
Whenever there is no possible confusion, we may abbreviate notation and denote an $\varepsilon\eta\,$-Einstein contact metric structure $(g,\alpha,\varepsilon)$ just by $\varepsilon\eta\,$-Einstein contact structure.
\end{remark}

\begin{remark}
Recall that we denote by $\mathrm{Q}^g\in \Gamma(TM\otimes T^{\ast}M)$ the endomorphism associated to $\mathrm{Ric}^g$. Then, $(g,\alpha,\varepsilon)$ is $\varepsilon\eta\,$-Einstein if and only if
\begin{equation*}
\mathrm{Q}^g = \frac{\mathrm{s}_g}{2}(\lambda^2 + \kappa\,\varepsilon)\,\mathrm{Id} - \mathrm{s}_g\,\kappa\, \xi\otimes \alpha\, ,
\end{equation*}

\noindent
in which case
\begin{equation*}
\mathrm{Q}^g(\xi) = \frac{\mathrm{s}_g}{2} (\lambda^2 - \kappa\,\varepsilon) \, \xi\ \, .
\end{equation*}

\noindent
Therefore $\xi$ is an eigenvector of $\mathrm{Q}^g$ with eigenvalue $\rho_{\xi} = \frac{\mathrm{s}_g}{2} (\lambda^2 - \kappa\,\varepsilon)$.
\end{remark}

\begin{definition}
We denote by $\mathrm{PCont}^{\varepsilon\eta}(\varepsilon,\lambda^2,\kappa)$ the category of $\varepsilon\eta\,$-Einstein $\varepsilon\,$-contact structures with respect to $\lambda^2$ and $\kappa$ and whose Reeb vector field is of norm $\varepsilon$. Likewise, denote by $\mathrm{PCont}^{\varepsilon \eta}_L(\varepsilon,\lambda^2,\kappa)$ ($\mathrm{PCont}^{\varepsilon \eta}_R(\lambda^2,\kappa)$) the category of Lorentzian (Riemannian) $\varepsilon\eta\,$-Einstein $\varepsilon\,$-contact structures with respect to $\lambda^2$, $\kappa$ and whose Reeb vector field is of norm $\varepsilon$.
\end{definition}

\begin{remark}
\label{remark:etaysasaki}
Let $(g,\alpha,\varepsilon) \in \mathrm{PCont}^{\varepsilon \eta}(\varepsilon,\lambda^2,\kappa)$ with $\varepsilon\, \s_g = 1$. By Definition \ref{eq:Einsteinpcontact}, we have:
\begin{equation*}
\ric^g(\xi,\xi)= \frac{\mathfrak{s}_g}{2}(\lambda^2 \varepsilon  - \kappa)\, .
\end{equation*}
On the other hand, by Proposition \ref{prop:sasakireeb} $(g,\alpha,\varepsilon)$ is Sasakian if and only if $\ric^g(\xi,\xi)=\mathfrak{s}_g\frac{\varepsilon}{2}$. This is equivalent to
\begin{equation*}
\lambda^2= 1 + \kappa \varepsilon\, ,
\end{equation*}

\noindent
In other words, an $\varepsilon\eta\,$-Einstein $\varepsilon\,$-contact structure $(g,\alpha,\varepsilon)$ such that $\s_g \varepsilon = 1$ is Sasakian if and only if  $(g,\alpha,\varepsilon) \in \mathrm{PCont}^{\varepsilon \eta}(\varepsilon, 1+\varepsilon \kappa,\kappa)$. 
\end{remark}

\begin{lemma}
\label{lemmalexpression}
Let $(N,g,\alpha,\varepsilon)\in \mathrm{PCont}^{\varepsilon \eta}(\varepsilon,\lambda^2,\kappa)$. Then, the following equation holds:
\begin{equation*}
\mathrm{R}^g(v_1,v_2)(\xi) = \cK\, (\alpha(v_2)\, v_1 - \alpha(v_1)\, v_2)\, ,
\end{equation*}
	
\noindent
where
\begin{equation*}
\cK\eqdef \frac{\s_g (\lambda^2 - \varepsilon \kappa )}{4}\, , 
\end{equation*}
	
\noindent
and $\mathrm{R}^g$ denotes the Riemann tensor on $(N,g)$.
\end{lemma}

\begin{proof}
The results follows directly by plugging the $\varepsilon\eta\,$-Einstein condition in the expression for the Riemann tensor of a (pseudo-)Riemannian three-manifold in terms of its Ricci curvature.
\end{proof}

\begin{remark}
The previous Proposition implies that a Riemannian $\varepsilon\eta\,$-Einstein contact manifold $(M,g,\alpha,\varepsilon)\in \mathrm{PCont}^{\varepsilon\eta}_R(\lambda^2,\kappa)$ is a $(\kappa,\mu)$ manifold or, alternatively, a Riemannian contact three-manifold whose Reeb vector field belongs to the nullity distribution \cite{BlairKoufoII}.
\end{remark}

\noindent
We consider now the $\varepsilon\eta\,$-Einstein condition on globally hyperbolic Lorentzian three-manifolds, making contact with Section \ref{sec:globhyperL}, where the $\varepsilon\,$-contact condition for a one-form on a globally hyperbolic Lorentzian three-manifold was considered. Therefore, using the notation of Section \ref{sec:globhyperL}, we set:
\begin{equation}
\label{eq:globhyfromtuple}
(N,g) = (\mathbb{R}\times \mathrm{X}_t , \,g = -\beta^2_t\,\dd t\otimes \dd t + q_t)\, , \qquad \alpha = F_t\, e^0 + \alpha^{\perp}_t\, ,
\end{equation}

\noindent
where $\left\{ \beta_t \right\}_{t\in\mathbb{R}}$, $\left\{ F_t \right\}_{t\in\mathbb{R}}$, $\left\{ \alpha_t \right\}_{t\in \mathbb{R}}$  and $\left\{ q_t\right\}_{t\in\mathbb{R}}$ are parametric families on $\mathrm{X}_t$, $e^0 = \beta_t\, \dd t$ is the normalized time-like one-form induced by the globally hyperbolic presentation of $(N,g)$ and where we have considered, for every $t\in \mathbb{R}$,
\begin{equation*}
\mathrm{X}_t \eqdef \left\{ t\right\}\times \mathrm{X}\subset N\, 
\end{equation*}

\noindent
as an embedded manifold. We introduce the familiar Weingarten tensor $W_t$ and second fundamental form $\Theta_t$:
\begin{equation*}
W_t = -\nabla n_t\vert_{T\mathrm{X}_t}\in \Omega^1(T\mathrm{X}_t)\, ,\qquad \Theta_t = - \frac{1}{2\beta_t} \cL_{\partial_t} g\vert_{T\mathrm{X}_t\times T\mathrm{X}_t}\in \Gamma( T^{\ast}\mathrm{X}_t \odot T^{\ast}\mathrm{X}_t)\, ,
\end{equation*}

\noindent
associated to the embedding $X_t\subset N$. The trace of $\Theta_t$ with respect to $q_t$, which we denote by $\mathrm{Tr}_{q_t}(\Theta_t)$, is the \emph{mean curvature} of the embedded surface $X_t\subset N$.

\begin{prop}
\label{prop:globhypervarepsilonetaEinstein}
A tuple $(\left\{\beta_t \right\}_{t\in\mathbb{R}}, \left\{ F_t \right\}_{t\in\mathbb{R}}, \left\{ \alpha^{\perp}_t \right\}_{t\in \mathbb{R}}, \left\{ q_t\right\}_{t\in\mathbb{R}})$ on $\mathrm{X}_t$ defines an $\varepsilon\eta\,$-Einstein $\varepsilon\,$-contact structure as prescribed in \eqref{eq:globhyfromtuple} if and only if:
\begin{eqnarray}
\label{eq:globhypervarepsilonetaEinstein}
&\dd_{\mathrm{X}_t}\alpha^{\perp}_t = F_t\,\nu_{q_t}\, , \qquad \ast_{q_t}\alpha^{\perp}_t + \frac{1}{\beta_t}\dd_{\mathrm{X}} (\beta_t F_t) = \cL_{n_t}\alpha_t^{\perp}\, , \qquad \vert \alpha_t^{\perp}\vert^2_{q_t} = \varepsilon + F^2_t\, ,\nonumber\\
& \mathrm{R}^{q_t} - \vert\Theta_t\vert^2_{q_t} +  (\Tr_{q_t} \Theta)^2 + \mathfrak{c} = 2 \kappa\, F^2_t\, , \qquad \dd_{\mathrm{X}_t} \mathrm{Tr}_{q_t}(\Theta_t) + \mathrm{div}_{q_t}(\Theta_t) = \kappa\, F_t\, \alpha^{\perp}_t\\
&\mathrm{Ric}^{q_t} + \Tr_{q_t}(\Theta_t)\, \Theta_t - 2\, \Theta_t(\mathrm{Id}\otimes W_t) - \frac{1}{\beta_t} (\dot{\Theta}_t + \nabla^{q_t} \dd_{\mathrm{X}_t} \beta_t) = \kappa\, \alpha^{\perp}_t \otimes \alpha^{\perp}_t - \frac{1}{2} (\lambda^2 + \kappa \varepsilon) q_t\, , \nonumber
\end{eqnarray}

\noindent
where $\cL$ denotes Lie derivative, $\dd_{\mathrm{X}_t}$ denotes the exterior derivative on $\mathrm{X}_t$, $\Tr_{q_t}$ denotes trace with respect to $q_t$, $\nabla^{q_t}$ represents the Levi-Civita connection of $q_t$  and $\mathfrak{c} = \frac{1}{2} (5\lambda^2 + 3\kappa \varepsilon)$ is a constant.
\end{prop}

\begin{proof}
The first line of equations in \eqref{eq:globhypervarepsilonetaEinstein} is proven in Proposition \ref{prop:globhypervarepsilon}. A direct computation shows that:
\begin{eqnarray*}
&\mathrm{Ric}^g(n,n) = \vert\Theta_t\vert^2_{q_t} +  \frac{1}{\beta_t} ((\tr_{q_t}\dot{\Theta})^2 + \Delta_{q_t} \beta_t)\, , \qquad \mathrm{Ric}^g(n)\vert_{T\mathrm{X}_t} = \dd_{\mathrm{X}_t} \mathrm{Tr}_{q_t}(\Theta_t) + \mathrm{div}_{q_t}(\Theta_t)\, , \\
&\mathrm{Ric}^g\vert_{T\mathrm{X}_t\otimes T\mathrm{X}_t} = \mathrm{Ric}^{q_t} + \Tr(\Theta_t)\, \Theta_t - 2\, \Theta_t(\mathrm{Id}\otimes W_t) - \frac{1}{\beta_t} (\dot{\Theta}_t + \nabla^{q_t} \dd_{\mathrm{X}_t} \beta_t)\, ,
\end{eqnarray*}

\noindent
Plugging these equations in the $\varepsilon\eta\,$-Einstein condition \eqref{eq:Einsteinpcontact} and combining them with the trace of \eqref{eq:Einsteinpcontact} we obtain the second and third lines in  \eqref{eq:globhypervarepsilonetaEinstein}. For more details the reader is referred to \cite{Baum} and references therein. 
\end{proof}

\noindent
Therefore, the $\varepsilon\eta\,$-Einstein condition on a globally hyperbolic Lorentian three manifold becomes, as expected, a dynamical equation on the evolution of a pair of functions, a one-form and a metric on a two-dimensional oriented manifold. Note that the only equations in \eqref{eq:globhypervarepsilonetaEinstein} that contain \emph{time derivatives} correspond to the second equation in the first line and the equation in the third line. These become the evolution equations for the tuple $(\left\{\beta_t \right\}_{t\in\mathbb{R}}, \left\{ F_t \right\}_{t\in\mathbb{R}}, \left\{ \alpha^{\perp}_t \right\}_{t\in \mathbb{R}}, \left\{ q_t\right\}_{t\in\mathbb{R}})$, while the rest of equations can be considered as \emph{constraint equations} in order to formulate a Cauchy or initial value problem for Lorentzian $\varepsilon\eta\,$-Einstein $\varepsilon\,$-contact structures. Set:
\begin{equation*}
\mathrm{X} \eqdef \mathrm{X}_0\, , \qquad \Theta \eqdef \Theta_0\, ,  \qquad q \eqdef q_0\, , \qquad \alpha^{\perp} \eqdef \alpha^{\perp}_0\, , \qquad \beta \eqdef  \beta_0 \, , \qquad F \eqdef F_0\, ,
\end{equation*}

\noindent
and consider $\mathrm{X}$ as the Cauchy surface for the zero-time intial values $(q,\Theta,F,\alpha^{\perp})$ associated to a given tuple 
\begin{equation*}
(\left\{\beta_t \right\}_{t\in\mathbb{R}}, \left\{ F_t \right\}_{t\in\mathbb{R}}, \left\{ \alpha^{\perp}_t \right\}_{t\in \mathbb{R}}, \left\{ q_t\right\}_{t\in\mathbb{R}})\, 
\end{equation*}

\noindent
satisfying equations \eqref{eq:globhypervarepsilonetaEinstein}. The following result is a direct consequence Proposition \ref{prop:globhypervarepsilonetaEinstein}.

\begin{prop}
\label{prop:constraintequations}
If $(\left\{\beta_t \right\}_{t\in\mathbb{R}}, \left\{ F_t \right\}_{t\in\mathbb{R}}, \left\{ \alpha^{\perp}_t \right\}_{t\in \mathbb{R}}, \left\{ q_t\right\}_{t\in\mathbb{R}})$ is a solution of \eqref{eq:globhypervarepsilonetaEinstein} then $(q,\Theta,F,\alpha^{\perp})$ satisfies:
\begin{eqnarray}
\label{eq:constraintequations}
&\dd_{\mathrm{X}}\alpha^{\perp} = F\,\nu_{q}\, ,\qquad \vert \alpha^{\perp}\vert^2_{q} = \varepsilon + F^2\, , \\
& \mathrm{R}^{q} - \vert\Theta\vert^2_{q} +  (\Tr_{q} \Theta)^2 + \mathfrak{c} = 2\, \kappa\, F^2 \, , \qquad \dd_{\mathrm{X}} \mathrm{Tr}_{q}(\Theta) + \mathrm{div}_{q}(\Theta) = \kappa\, F\, \alpha^{\perp}\, , \nonumber
\end{eqnarray}

\noindent
where $\Theta$ is a symmetric bilinear form on $\mathrm{X}$.
\end{prop}

\begin{remark}
\label{remark:constraintequations}
Equations \eqref{eq:constraintequations}, to which we will refer as the \emph{constraint equations} of an $\varepsilon\eta\,$-Einstein structure, generalize the well-known constraint equations of General Relativity coupled to a cosmological constant in $(2+1)$ dimensions via the coupling of an $\varepsilon\,$-contact structure. If we take $F = 0$ then, the second equation in the first line of \eqref{eq:constraintequations} forces $\varepsilon = 1$ and the whole system decouples. The second line in \eqref{eq:constraintequations} corresponds in this case to the constraint equations of General Relativity coupled to a cosmological constant. This system has been studied in the literature, see \cite{Moncrief:1989dx,Andersson:1996dr} and references therein. In fact, it would be interesting to explore if a Hamiltonian formulation on (the cotangent space of) Teichm\"uller space, in the lines of the one presented in \cite{Moncrief:1989dx,Andersson:1996dr}, can be developed also for $\varepsilon\eta\,$-Einstein structures. Note that, under the assumption $F=0$ the first line reduces to the condition of $\alpha^{\perp}$ being closed and of constant norm, a condition which if $\mathrm{X}$ is compact can only be satisfied on the torus.  
\end{remark}

\noindent
Proving the converse of Proposition \ref{prop:constraintequations}, that is, proving that for every solution of \eqref{eq:constraintequations} there exists a tuple $(\left\{\beta_t \right\}_{t\in\mathbb{R}}, \left\{ F_t \right\}_{t\in\mathbb{R}}, \left\{ \alpha^{\perp}_t \right\}_{t\in \mathbb{R}}, \left\{ q_t\right\}_{t\in \mathbb{R}})$ fulfilling \eqref{eq:globhypervarepsilonetaEinstein}, would solve the initial value problem of $\varepsilon\eta\,$-Einstein $\varepsilon\,$-contact structures. Addressing this problem is beyond the scope of this manuscript and will be considered elsewhere. Having said this, we expect the converse to hold due to the fact that $\varepsilon\eta\,$-Einstein $\varepsilon\,$-contact condition arises in a supergravity theory, which is expected to pose consistent initial value problems due to their supersymmetric structure \cite{Choquet}.  

We distinguish the cases $\kappa = 0$ and $\kappa \neq 0$. Due to its complexity, we plan to investigate the case $\kappa \neq 0 $ in a separate publication, so we focus instead on the case $\kappa = 0$ in the following, with the goal of showing existence of particular solutions. If $\kappa = 0$ then Equations \eqref{eq:constraintequations} decouple again (see Remark \ref{remark:constraintequations}) and the second line in \eqref{eq:constraintequations} corresponds to the constraint equations of General Relativity coupled to a cosmological constant. Hence, we focus on the first line in \eqref{eq:constraintequations}, which we rewrite as follows:
\begin{equation*}
\ast\,\dd_X\alpha^{\perp} + \sqrt{\vert \alpha^{\perp}\vert^2_{q} - \varepsilon} = 0\, ,
\end{equation*}

\noindent
with variable given by a one-form $\alpha^{\perp}\in \Omega^1(\mathrm{X})$ on a complete Riemann surface $(\mathrm{X},q)$. We focus on the null-contact case $\varepsilon = 0$, since it seems to be new in the literature. In this case, the previous equation reduces simply to
\begin{equation}
\label{eq:nulldecoupled}
\ast\,\dd_X\alpha^{\perp} + \vert \alpha^{\perp}\vert_{q}  = 0\, .
\end{equation}
 
\noindent
We introduce now local isothermal coordinates $(x,y)$ on $\mathrm{X}$, in which the metric $q$ reads $q = F^2\,( \dd x^2 + \dd y^2)$ for a local function $F$. In this coordinates, Equation \eqref{eq:nulldecoupled} reads
\begin{equation*}
\partial_x \alpha_y - \partial_y \alpha_x = F\, (\alpha_x^2 + \alpha_y^2)^{\frac{1}{2}}\, ,
\end{equation*}

\noindent
where we have written $\alpha^{\perp} = \alpha_x \dd x + \alpha_y \dd y$. Assuming $\alpha_x= 0$ the previous equation admits the solution
\begin{equation*}
\alpha_y = e^{\int F\,\dd x + f(y)}\, , 
\end{equation*}
 
\noindent
where $f(y)$ is a local function depending only on $y$. Assuming on the other hand $\alpha_y= 0$ the previous equation admits the solution
\begin{equation*}
\alpha_x = e^{-\int F\,\dd y + f(x)}\, , 
\end{equation*}

\noindent
where $f(x)$ is a local function depending only on $x$. Hence, we obtain the following result.

\begin{prop}
Every solution to the constraint equations of General Relativity coupled to a cosmological constant is, at least locally, a solution to the constraint equations of an $\varepsilon\eta\,$-Einstein null contact structure with $\kappa = 0$.
\end{prop}  

\noindent
We study now the $\varepsilon\eta\,$-Einstein condition on a case by case basis by distinguishing the signature of $g$ and the the causal character of the Reeb vector field.

%%%%%%%%%%%%%%%%%%%%%%%%%%%%%%%%%%%%%%%%%%%%%%%%%%%%%%%%%%%%%%%%%%%%%%%%%%%%%%%%%%%%%%%%%%%%%%%%%%%%%%%%%%%%%%%%%%%%%%%%%%%%%
%%%%%%%%%%%%%%%%%%%%%%%%%%%%%%%%%%%%%%%%%%%%%%%%%%%%%%%%%%%%%%%%%%%%%%%%%%%%%%%%%%%%%%%%%%%%%%%%%%%%%%%%%%%%%%%%%%%%%%%%%%%%%

\subsection{Riemannian $\varepsilon\eta\,$-Einstein contact structures}

%%%%%%%%%%%%%%%%%%%%%%%%%%%%%%%%%%%%%%%%%%%%%%%%%%%%%%%%%%%%%%%%%%%%%%%%%%%%%%%%%%%%%%%%%%%%%%%%%%%%%%%%%%%%%%%%%%%%%%%%%%%%%
%%%%%%%%%%%%%%%%%%%%%%%%%%%%%%%%%%%%%%%%%%%%%%%%%%%%%%%%%%%%%%%%%%%%%%%%%%%%%%%%%%%%%%%%%%%%%%%%%%%%%%%%%%%%%%%%%%%%%%%%%%%%%

We briefly review in this section the classification $\varepsilon\eta\,$-Einstein Riemannian contact structures, with the goal of constructing solutions of supergravity coupled to a tensor multiplet, as explained in Sections \ref{sec:6dsugrasolutions} and \ref{sec:RicciFlatTorsion}. Riemannian contact structures on three-manifolds have been extensively studied in the literature, see \cite{BlairKoufoII, Koufo,Geiges} and references therein. Adapting and refining the results of \cite{BlairKoufoII, Koufo,Geiges} to our situation and conventions we obtain the following Theorem.
 
\begin{thm}\cite{BlairKoufoII, Koufo, Geiges}
\label{thm:rcs}
Let $(M,g,\alpha)$ be a three-dimensional complete and simply connected $\varepsilon\eta\,$-Einstein Riemannian contact metric manifold, $(M,g,\alpha) \in \mathrm{PCont}_R^{\varepsilon \eta} (\lambda^2,\kappa)$. Then one of the following holds:
\begin{itemize}
\item $(M,g,\alpha)$ is Sasakian. If $(M,g,\alpha)$  is in addition a Lie group equipped with a left-invariant $\varepsilon\eta\,$-Einstein Sasakian structure, then it is isomorphic\footnote{Two $\varepsilon\,$-contact structures are isomorphic in the sense given at Definition \ref{def:morphpcont}.} to a left-invariant $\varepsilon\eta\,$-Einstein structure on:
\begin{enumerate}
\item $\mathrm{SU}(2)$ if $\lambda^2=1+\kappa\,$, $\kappa > -1$.
\item $\mathrm{H}_3$ if $\lambda^2=0$, $\kappa=-1$.
\end{enumerate}
% $\mathrm{SU}(2)$ if $\lambda^2=1+\kappa\,$, $\kappa > -1$ or isomorphic to a left-invariant $\varepsilon\eta\,$-Einstein structure on $\mathrm{H}_3$ if $\lambda^2=0$, $\kappa=-1$ ($\kappa <-1$ is not allowed).
\item $(M,g,\alpha)$ is non-Sasakian and isomorphic to a left-invariant $\varepsilon\eta\,$-Einstein structure on:
\begin{enumerate}
\item $\mathrm{SU}(2)$ if $\lambda^2=-\kappa=\frac{1}{2}-\frac{1}{2}\mu^2$, with $0< \mu<1$ the positive eigenvalue of the tensor field $\mathfrak{h}$.
\item $\widetilde{\mathrm{E}}(2)$ if $\lambda^2=\kappa=0$.
\end{enumerate}
 %$\mathrm{SU}(2)$ equipped with a left-invariant $\varepsilon\eta\,$-Einstein structure when $\lambda^2=-\kappa=\frac{1}{2}-\frac{1}{2}\mu^2$, where $0< \mu<1$ is the positive eigenvalue of the tensor field $\mathfrak{h}$, or isomorphic to a left-invariant $\varepsilon\eta\,$-Einstein structure on $\widetilde{\mathrm{E}}(2)$ with $\lambda^2=\kappa=0$.
\end{itemize}
\end{thm}

\begin{proof}
If $(M,g,\alpha)$ is Sasakian and not isomorphic to a Lie group equipped with a left-invariant structure $(g,\alpha)$ then we are done. On the other hand, reference \cite{Koufo} proves that a complete and simply connected $\varepsilon\eta\,$-Einstein non-Sasakian Riemannian contact three-manifold has a Lie group structure respect to which its $\varepsilon\eta\,$-Einstein structure is left-invariant. Therefore, we use the classification \cite{Milnor} of three-dimensional Riemannian Lie algebras to proceed on a case by case basis evaluating the $\varepsilon\eta\,$-Einstein condition. We distinguish between the Sasakian and the non-Sasakian cases.

\begin{itemize}
\item Sasakian case. First, if $(M,g,\alpha)$ is a Sasakian unimodular Lie group then there exists an orthonormal left-invariant frame $\{e_1,e_2,e_3\}$ whose associated Lie brackets satisfy \cite{Milnor}:
\begin{equation*}
[e_2,e_3]=\mu_1 e_1\, , \quad [e_3,e_1]=\mu_2 e_2 \, , \quad [e_1,e_2]=\mu_3 e_3\, ,
\end{equation*}

\noindent
for some real constants $\mu_1,\mu_2,\mu_3 \in \mathbb{R}$. This immediately implies:
\begin{equation*}
\dd e^1 = - \mu_1 e^2\wedge e^3\, , \qquad \dd e^2 = - \mu_2 e^3\wedge e^1\, , \qquad \dd e^3 = - \mu_3 e^1\wedge e^2\, ,
\end{equation*}

\noindent
where $\{e^1,e^2,e^3\}$ is the coframe dual to $\{e_1,e_2,e_3\}$. Expressing $\alpha=\alpha_0 e^0+\alpha_1 e^1 +\alpha_2 e^2$, the contact condition (we include the sign $s \in \mathbb{Z}_2$ in order to take into account orientation-reversing isometries) reads
\begin{equation*}
\ast  \alpha=s \dd \alpha\, ,
\end{equation*}

\noindent
what is equivalent to:
\begin{equation*}
\alpha_1=-s \mu_1 \alpha_1\, , \quad \alpha_2=-s \mu_2 \alpha_2\, , \quad \alpha_3=-s \mu_3 \alpha_3\, .
\end{equation*}

\noindent
Furthermore, the non-zero components of the Ricci curvature tensor $\ric^g$ read:
\begin{eqnarray*}
& \ric^g(e_1,e_1) =\frac{1}{2}(\mu_1^2-\mu_2^2-\mu_3^2+2\mu_2 \mu_3 )\, , \quad \ric^g(e_2,e_2) = \frac{1}{2}(-\mu_1^2+\mu_2^2-\mu_3^2+2\mu_1 \mu_3 ) \, , \quad \\ 
& \ric^g(e_3,e_3) =\frac{1}{2}(-\mu_1^2-\mu_2^2+\mu_3^2+2\mu_1 \mu_2 )\, , \quad \ric^g(e_1,e_2)= \ric^g(e_1,e_3)= \ric^g(e_2,e_3)=0\, .
\end{eqnarray*}

\noindent
We distinguish now the following subcases of the Sasakian unimodular case:

\begin{itemize}
\item Assume $\alpha_1,\alpha_2,\alpha_3 \neq 0$. Then $\mu_1=\mu_2=\mu_3=-s$ and $\ric^g=\frac{1}{2}g$, which follows from $\lambda^2=1$ and $\kappa=0$. Therefore, choosing $s=-1$ we conclude that $(M,g)$ is isometric to $\mathrm{SU}(2)$ equipped with a left-invariant metric. 

\item Assume $\alpha_1,\alpha_2 \neq 0$ and $\alpha_3=0$. Again, $\mu_1,\mu_2=-s$. Since $\ric(e_1,e_2)=0$, we obtain $\kappa=0$. Consequently, $\lambda^2=1$ which in turn implies, by equating $\ric^g(e_1,e_1)=\ric^g(e_2,e_2)=\ric^g(e_3,e_3)$, that $\mu_3=-s$. Taking $s=-1$, we recover the previous case and we conclude that $(M,g)$ is isometric to $\mathrm{SU}(2)$. A similar analysis holds when $\alpha_1,\alpha_3 \neq 0$ and $\alpha_2=0$ and when $\alpha_2,\alpha_3 \neq 0$ and $\alpha_1=0$. 

\item Assume $\alpha_2=\alpha_3=0$ and $\alpha_1^2=1$. Then, $\xi= \pm e_1$ and we obtain $\ric^g(e_1,e_1)=\frac{1}{2}$ by the Sasaki condition. However, owing the fact that $\alpha_1 \neq 0$ implies that $\mu_1=-s$, then we have that
\begin{equation*}
\ric^g(e_1,e_1)=\frac{1}{2}-\frac{1}{2}(\mu_2-\mu_3)^2=\frac{1}{2}\,.
\end{equation*}
Therefore $\mu_2=\mu_3=\mu \in \mathbb{R}$. Now, since $\ric^g(e_2,e_2)=\ric^g(e_3,e_3)=\frac{1}{2}(2\lambda^2-1)$, we get the constraint
\begin{equation*}
\lambda^2=-s\mu\,.
\end{equation*}
Taking $s=-1$, we conclude that for $\mu>0$ $(M,g)$ is isometric to $\mathrm{SU}(2)$ and that for $\mu=0$ $(M,g)$ is isometric to $\mathrm{H}_3$. A completely similar analysis holds when $\alpha_2^2 = 1$ and $\alpha_3^2=1$.
\end{itemize}

\noindent
Secondly, if $(M,g,\alpha)$ is a Sasakian non-unimodular Lie group then there exists an orthonormal left-invariant frame $\{e_1,e_2,e_3\}$ whose associated Lie brackets satisfy \cite{Milnor}:
\begin{equation*}
[e_1 , e_2] = a e_2 + b e_3\, , \quad [e_1 , e_3] = c e_2 + f e_3\, , \quad [e_2 , e_3] = 0\, ,
\end{equation*}

\noindent
for real numbers $a, b, c, f \in \mathbb{R}$ satisfying that $a+f\neq 0$. These Lie brackets immediately imply:
\begin{equation*}
\dd e^1 = 0\, , \qquad \dd e^2 = - a e^1\wedge e^2 - c e^1\wedge e^3\, , \qquad \dd e^3 = - b e^1 \wedge e^2 - f e^1\wedge e^3\, .
\end{equation*}
Therefore the contact condition imposes the following constraints:
\begin{equation*}
c \alpha_2 + f \alpha_3=s\alpha_2\, , \quad -a\alpha_2-b \alpha_3=s\alpha_3\, , \quad \alpha_1=0\,.
\end{equation*}
The components of the Ricci curvature tensor in this orthonormal basis read:
\begin{eqnarray*}
& \ric^g(e_1,e_1) =-a^2-f^2-\frac{1}{2}(b+c)^2\, , \quad \ric^g(e_2,e_2) = -a^2-\frac{b^2}{2}+\frac{c^2}{2}-fa \, , \quad \\ 
& \ric^g(e_3,e_3) =-f^2+\frac{b^2}{2}-\frac{c^2}{2}-af\, , \quad \ric^g(e_1,e_2)= \ric^g(e_1,e_3)=0\, , \quad  \ric^g(e_2,e_3)=-ac-bf\, .
\end{eqnarray*}

\noindent
In particular, we have: 
\begin{equation*}
\ric^g(e_2,e_2)+\ric^g(e_3,e_3)=-a^2-2af-f^2=-(a+f)^2\, .
\end{equation*}

\noindent
Imposing the $\varepsilon\eta\,$-Einstein condition, and taking into account that $\alpha_2^2+\alpha_3^2=1$ by the contact condition, we obtain $\ric^g(e_2,e_2)+\ric^g(e_3,e_3)=\lambda^2+\kappa-\kappa \alpha_2^2-\kappa \alpha_3^2=\lambda^2$, which in turn yields $\lambda^2=-(a+f)^2 <0$. Hence, non-unimodular Riemannian groups admit no left-invariant $\varepsilon\eta\,$-Einstein structures.

\item If $(M,g,\alpha)$ is non-Sasakian, then \cite{Koufo} shows that the $\varepsilon\eta\,$-Einstein condition implies $\lambda^2=-\kappa$. Furthermore, \cite{BlairKoufoII,Koufo} prove that there exists a frame $\{\xi, X, \phi(X)\}$, where $X$ is an eigenvector of $\mathfrak{h}$ with positive eigenvalue, whose Lie brackets satisfy:
\begin{equation*}
[\xi, X]=\frac{1}{2}(1+\mu) \phi(X) \, , \quad [\xi, \phi(X)]=-\frac{1}{2}(1-\mu) X \, , \quad [X, \phi(X)]=\xi\, ,
\end{equation*}

\noindent
where $\mu$ is the positive eigenvalue of $\mathfrak{h}$. Comparing the previous Lie brackets to Milnor's classification, we obtain that $M$ must be isometric to $\mathrm{SU}(2)$ when $\mu<1$ and isometric to $\widetilde{\mathrm{E}}(2)$ when $\mu=1$. Finally, using Proposition \ref{prop:ricxien0}, we find the relation $\mu^2=1-2\lambda^2$.
\end{itemize}
\end{proof}
 
%%%%%%%%%%%%%%%%%%%%%%%%%%%%%%%%%%%%%%%%%%%%%%%%%%%%%%%%%%%%%%%%%%%%%%%%%%%%%%%%%%%%%%%%%%%%%%%%%%%%%%%%%%%%%%%%%%%%%%%%%%%%%%
%%%%%%%%%%%%%%%%%%%%%%%%%%%%%%%%%%%%%%%%%%%%%%%%%%%%%%%%%%%%%%%%%%%%%%%%%%%%%%%%%%%%%%%%%%%%%%%%%%%%%%%%%%%%%%%%%%%%%%%%%%%%%%

\subsection{Lorentzian $\varepsilon\eta\,$-Einstein structures with time-like Reeb vector field}

%%%%%%%%%%%%%%%%%%%%%%%%%%%%%%%%%%%%%%%%%%%%%%%%%%%%%%%%%%%%%%%%%%%%%%%%%%%%%%%%%%%%%%%%%%%%%%%%%%%%%%%%%%%%%%%%%%%%%%%%%%%%%%
%%%%%%%%%%%%%%%%%%%%%%%%%%%%%%%%%%%%%%%%%%%%%%%%%%%%%%%%%%%%%%%%%%%%%%%%%%%%%%%%%%%%%%%%%%%%%%%%%%%%%%%%%%%%%%%%%%%%%%%%%%%%%%

Let $(M,g,\alpha,-1)\in \mathrm{PCont}_L^{\varepsilon \eta} (-1,\lambda^2,\kappa)$ be an $\varepsilon\,$-contact metric manifold $M$ with time-like vector field. Note that, by Definition \ref{def:Einsteinpcontact}, $(g,\alpha,\varepsilon)$ is $\varepsilon\eta\,$-Einstein if and only if it satisfies
\begin{equation}
\label{eq:etaEinsteinRiemannian}
\mathrm{Ric}^g = \frac{1}{2}(- \lambda^2 + \kappa )\,\chi +\kappa\, \alpha\otimes \alpha\, ,
\end{equation} 

\noindent
where $\lambda \in \mathbb{R}\, $ and $\kappa \geq 0$.  
 
\begin{remark}
\label{remark:perfectfluid}
The Lorentzian $\varepsilon\eta\,$-Einstein condition with time-like Reeb vector field  has a natural physical interpretation in the context of General Relativity. Indeed, it is an \emph{Einstein-like equation}, that is, it can be written as follows:
\begin{equation*}
\mathrm{G}(g) = \mathrm{T}(g,\alpha)\, ,
\end{equation*}
	
\noindent
where $\mathrm{G}(g) = \mathrm{Ric}^g - \frac{1}{2} \mathrm{R}^g g$ denotes the Einstein tensor and $\mathrm{T}(g,\alpha)\in \Gamma(T^{\ast}N\odot T^{\ast}N)$ is the \emph{energy-momentum tensor} of Einstein's equations, given in this case by
\begin{equation*}
\mathrm{T}(g,\alpha) = \frac{1}{4} (\lambda^2 + \kappa)\, \chi + \kappa\, \alpha\otimes \alpha\, .
\end{equation*}
	
\noindent
Interestingly enough, $\mathrm{T}(g,\alpha)$ corresponds with the energy-momentum tensor of a perfect fluid whose speed is prescribed by $\xi$ and whose pressure $p$ and rest-frame mass density $o$ are constant and given by:
\begin{equation*}
p = \frac{\lambda^2 + \kappa}{4}\, , \qquad o = \frac{3\,\kappa - \lambda^2}{4}\, .
\end{equation*}
	
\noindent
Hence, the $\varepsilon\eta\,$-Einstein condition in Lorentzian signature with $\varepsilon = -1$ corresponds with the Einstein's General Relativity equations for a Lorentzian metric coupled to a perfect fluid with velocity prescribed by the Reeb vector field. This interpretation allows to apply the extensive literature dedicated to the study of perfect-fluid space-times, see for instance \cite{GarciaDiaz} and references therein, to the study of Lorentzian $\varepsilon\eta\,$-Einstein $\varepsilon\,$-contact structures. 
\end{remark}

\noindent
We prove now that a Lorentzian $\varepsilon\eta\,$-Einstein contact structure $(g,\alpha,-1)$ on a connected and simply connected complete\footnote{We say that an $\varepsilon\,$-contact metric three-manifold $(N,\chi,\alpha,\varepsilon = -1)$ is \emph{complete} if all elements of any $\varepsilon\,$-contact frame are complete on $N$.} three-dimensional Lorentzian manifold $(M,g)$ is either Sasakian or isometric to a Lie group equipped with a Lorentzian left-invariant contact structure $(g,\alpha,-1)$. We proceed analogously to the Riemannian case \cite{BlairKoufoII}. We prove first existence of a special type of $\varepsilon\,$-contact frame, particularly convenient for computations. 
  
\begin{lemma}
\label{lemma:phihbasisb}
Let $(M,g,\alpha,-1) \in \mathrm{PCont}_L^{\varepsilon \eta}(-1,\lambda^2,\kappa)$ be a simply connected Lorentzian $\varepsilon\eta\,$-Einstein $\varepsilon\,$-contact metric three-manifold. Then, there exists a global orthonormal frame $\{\xi, X, \phi(X)\}$ such that $\mathfrak{h}(X)=\mu X$ and $\mathfrak{h}(\phi(X))=-\mu \phi(X)$, where:
\begin{equation*}
\mu=\sqrt{1-(\lambda^2+\kappa)}\, .
\end{equation*}

\noindent
In particular, $\lambda^2+\kappa \leq 1$.
\end{lemma}

\begin{proof}
First notice that since $M$ is simply connected and three-dimensional it is parallelizable, whence it admits nowhere vanishing vector fields. If  $(g,\alpha)$ is Sasakian then $\mathfrak{h} = 0$ and $\lambda^2+\kappa = 1$, see Remark \ref{remark:etaysasaki}, and the statement is trivial. We consider thus the non-Sasakian case. For every nowhere vanishing space-like vector field $Y \in \mathfrak{X}(M)$ of unit norm, $\left\{ \xi , Y , \phi(Y)\right\}$ is a global $\varepsilon\,$-contact frame. Furthermore:
\begin{equation*}
g(\mathfrak{h}^2(Y),\xi)=0\, , \qquad g(\mathfrak{h}^2(Y),\phi(Y))=0\, ,
\end{equation*}

\noindent
since $g(\mathfrak{h}^2(Y),\phi(Y))=g(Y,\mathfrak{h}^2(\phi(Y)) )=g(Y,\phi(\mathfrak{h}^2(Y)))=-g(\mathfrak{h}^2(Y),\phi(Y))$. Therefore:
\begin{equation*}
\mathfrak{h}^2(Y)=\sigma^2 \, Y=-\sigma^2 \phi^2(Y)\, ,
\end{equation*}

\noindent
for some $\sigma \in C^\infty(M)$. Therefore, $\text{Tr}(\mathfrak{h}^2)=2\sigma^2$. Since $(M,g,\alpha,-1)\in \mathrm{PCont}_L^{\varepsilon \eta}(-1,\lambda^2,\kappa)$, combining Proposition \ref{prop:ricxien0} and Remark \ref{remark:etaysasaki} we obtain that $\sigma=\sqrt{1-(\lambda^2+\kappa)}$. Since $(M,g,\alpha,-1)$ is not Sasakian we have $\lambda^2+\kappa < 1 $ whence $\mu$ is strictly positive and corresponds with the positive eigenvalue of $\mathfrak{h}$. 
\end{proof}

\noindent
Working in the special global $\varepsilon\,$-contact frame introduced in Lemma \ref{lemma:phihbasisb} we prove that every non-Sasakian $\varepsilon\eta\,$-Einstein connected and simply-connected Lorentzian three-dimensional manifold is a Lie group equipped with a left-invariant $\varepsilon\eta\,$-Einstein structure. 

\begin{prop}
\label{prop:cdnslc}
Let $(M,g,\alpha,-1) \in \mathrm{PCont}_L^{\varepsilon \eta} (-1,\lambda^2,\kappa)$ be simply connected with $\lambda^2 +\kappa < 1$ (that is, $(M,g,\alpha)$ is not Sasakian). Then, the following Lie brackets hold in the special $\varepsilon\,$-contact frame  $\{\xi, X,\phi(X)\}$ described in Lemma $\ref{lemma:phihbasisb}:$
\begin{eqnarray*}
[\xi, X]=\frac{1}{2}(1+\sqrt{1-2\lambda^2}) \phi(X)\, , \quad [\xi,\phi(X)]=\frac{1}{2}(-1+\sqrt{1-2\lambda^2}) X \, , \quad [X,\phi(X)]=-\xi\, .
\end{eqnarray*}
Furthermore, $\lambda^2=\kappa$ and $0 \leq \lambda^2 = \kappa < \frac{1}{2}$ since $\kappa \geq 0$.
\end{prop}

\begin{proof}
Propositions \ref{prop:eqfh} and \ref{prop:equationscontact} imply that:
\begin{equation*}
\nabla_\xi \xi=0\, , \quad \nabla_X \xi=-\frac{1}{2}(1+\mu) \phi(X)\, , \quad \nabla_{\phi(X)}\xi =\frac{1}{2}(1-\mu) X\, .
\end{equation*}

\noindent
On the other hand, we have that $g(\nabla_X X, \xi)=-g(X,\nabla_X \xi)=0$ and
\begin{equation*}
g(\nabla_{\phi(X)} \phi(X),\xi) =-g(\phi(X),\nabla_{\phi(X)} \xi)=0\, .
\end{equation*}

\noindent
Therefore, $\nabla_X X=c \phi(X)$ and $\nabla_{\phi(X)} \phi(X)=e X$ for some functions $c,e \in C^\infty(M)$. Similarly, $g(\nabla_\xi X,\xi)=g(\nabla_\xi( \phi(X) ),\xi)=0$, whence $\nabla_\xi X=\beta \phi(X)$ and $\nabla_\xi \phi(X)=-\beta X$ for a function $\beta \in C^\infty (M)$. Furthermore, we have: 
\begin{eqnarray*}
& g(\nabla_X \phi(X), \xi)=-g(\phi(X), \nabla_X \xi) = \frac{1}{2} (1+\mu)\, ,\quad g(\nabla_{\phi(X)} X, \xi)=-g(X,\nabla_{\phi(X)} \xi) = \frac{1}{2} (\mu-1)\, , \\ 
& g(\nabla_X \phi(X), X)=-g(\phi(X),\nabla_X X) = -c\, , \quad g(\nabla_{\phi(X)} X, \phi(X))=-g(X, \nabla_{\phi(X)} \phi(X) ) = - e\, .
\end{eqnarray*}

\noindent
Summarizing:
\begin{eqnarray*}
& \nabla_{\xi} \xi=0\, , \quad \nabla_{X} \xi=-\frac{1}{2}(1+\mu) \phi(X) \, , \quad \nabla_{\phi(X)} \xi=\frac{1}{2}(1-\mu) X\, , \\
& \nabla_\xi X=\beta \phi(X)\, , \quad \nabla_X X=c \phi(X)\, , \quad  \nabla_{\phi (X)} X= \frac{1}{2} (1-\mu) \xi-e \phi(X) \, ,\\ 
&\nabla_\xi \phi(X)=-\beta X\, , \quad \nabla_{X} \phi(X)=- \frac{1}{2}(1+\mu) \xi-cX \, , \nabla_{\phi(X)} \phi(X)=e X \, ,
\end{eqnarray*}

\noindent
which implies:
\begin{equation*}
[\xi,X]=\left (\beta+\frac{1}{2}+\frac{\mu}{2} \right ) \phi(X)\, , \quad [\xi,\phi(X)]=\left (-\beta-\frac{1}{2}+\frac{\mu}{2} \right ) X\, , \quad  [X,\phi(X)]=-\xi-c X+e \phi(X)\, ,
\end{equation*}

\noindent
by using the torsion-free property of $\nabla$. Making use of the previous equations, we compute:  
\begin{equation*}
\mathrm{R}^g(\xi,X)\xi=\left (  \mu \beta -\frac{1}{4}+\frac{\mu^2}{4} \right ) X \, , \quad \mathrm{R}^g(X, \phi(X))\xi=e\mu X -c \mu \phi(X)\, .
\end{equation*}

\noindent
By comparing with Lemma \ref{lemmalexpression}, we conclude that $\mu \beta=\mu c=\mu e=0$. Since $\mu \neq 0$ by assumption, then $\beta=c=e=0$. Finally, by $\ric^g(X,X)=\ric^g(\phi(X),\phi(X))=0$ we get $\lambda^2=\kappa$. Hence, $\mu=\sqrt{1-2\lambda^2}$ and we conclude.
\end{proof} 

\begin{remark}
\label{remark:covariantprescription}
From the proof of Proposition \ref{prop:cdnslc} we extract the following covariant derivatives of the $\varepsilon\,$-contact frame of special type described in Lemma \ref{lemma:phihbasisb}:
\begin{eqnarray*}
& \nabla_{\xi} \xi=0\, , \quad \nabla_{X} \xi=-\frac{1}{2}(1+\sqrt{1-2\lambda^2}) \phi(X) \, , \quad \nabla_{\phi(X)} \xi=\frac{1}{2}(1-\sqrt{1-2\lambda^2}) X\, , \\
& \nabla_\xi X=0\, , \quad \nabla_X X=0\, , \quad  \nabla_{\phi (X)} X= \frac{1}{2} (1-\sqrt{1-2\lambda^2}) \xi \, ,\\ &\nabla_\xi \phi(X)=0\, , \quad \nabla_{X} \phi(X)=- \frac{1}{2}(1+\sqrt{1-2\lambda^2}) \xi \, , \quad \nabla_{\phi(X)} \phi(X)=0 \, .
\end{eqnarray*}
\end{remark}

\begin{prop}
\label{prop:claslns}
Let $(M,g,\alpha,-1) \in \mathrm{PCont}_L^{\varepsilon \eta}(-1,\lambda^2,\kappa)$ be a complete and simply connected three-dimensional Lorentz $\varepsilon\eta\,$-Einstein contact manifold. Then, one of the following holds:
\begin{itemize}
\item $(M,g,\alpha,-1)$ is Sasakian.
\item $(M,g,\alpha,-1)$ is non-Sasakian and isomorphic to one of the following Lie groups equipped with a left-invariant $\varepsilon\eta\,$-Einstein structure $(g,\alpha,-1)$:
\begin{enumerate}
\item $\widetilde{\mathrm{Sl}}(2,\mathbb{R})$ when $\frac{1}{2}>\lambda^2=\kappa>0$.
\item  $\widetilde{\mathrm{E}}(1,1)$ when $\lambda^2=\kappa=0$. 
\end{enumerate}
\end{itemize}
\end{prop}

\begin{proof}
Assume $(M,g,\alpha,-1)$ is non-Sasakian. By Proposition \ref{prop:cdnslc}, there exists a global orthonormal frame $\{\xi, X, \phi(X)\}$ such that:
\begin{equation*}
[\xi, X]=\frac{1}{2}(1+\sqrt{1-2\lambda^2}) \phi(X)\, , \quad [\xi,\phi(X)]=\frac{1}{2}(-1+\sqrt{1-2\lambda^2}) X \, , \quad [X,\phi(X)]=-\xi\, ,
\end{equation*}
such that $\lambda^2=\kappa$ and $0 \leq \lambda^2 = \kappa < \frac{1}{2}$. Using \cite[Proposition 1.9]{Tricerri}, the fact that $M$ admits a global frame with constant structure functions implies that $(M,g,\alpha)$ has a Lie group structure (canonical after fixing an identity point) respect to which $\{\xi, X, \phi(X)\}$ is left-invariant. Using the classification of connected and simply connected three-dimensional Lie groups summarized in Appendix \ref{app:Lorentzgroups} we conclude that $(M,g)$ is of type $\mathfrak{g}_3$ when $\frac{1}{2}> \lambda^2=\kappa>0$ by identifying:
\begin{equation*}
\xi = e_0\, , \qquad e_1 = X\, , \qquad e_2 = \phi(X)\, .
\end{equation*}

\noindent
In particular, we have:
\begin{equation*}
a = \frac{1}{2}(1 - \sqrt{1-2\lambda^2})\, , \qquad b = \frac{1}{2}(1 + \sqrt{1-2\lambda^2})\, , \qquad c= 1\, .
\end{equation*}

\noindent
where $a$, $b$ and $c$ are the real parameters appearing in case $\mathfrak{g}_3$ of Theorem \ref{thm:LorentzGroups}. Hence $(M,g)$ is isometric to $\widetilde{\mathrm{Sl}}(2,\mathbb{R})$ equipped with a left-invariant metric. On the other hand, when $\lambda^2=\kappa=0$, identifying:
\begin{equation*}
e_0=-\xi\, , \qquad e_1=-\phi(X) \, , \qquad e_2=-X\, , 
\end{equation*}

\noindent
we obtain that $(M,g)$ is isometric to $\widetilde{\mathrm{E}}(1,1)$ endowed with a left-invariant metric. 
\end{proof}

\noindent
For Sasakian structures, we obtain the following result.

\begin{lemma}
\label{lemma:claslsas}
Let $(\G,g,\alpha,-1) \in \mathrm{PCont}_L^{\varepsilon \eta}(-1,\lambda^2,1-\lambda^2)$ be a left-invariant Sasakian Lorentzian $\varepsilon\eta\,$-Einstein contact structure on a simply connected Lie group $\G$. Then, according to the classification of connected and simply connected 3-dimensional Lie groups given in Theorem \ref{thm:LorentzGroups}, one of the following holds:
\begin{itemize}
\item $\G$ is of type $\mathfrak{g}_3$. In particular, we have that $(\G,g,\alpha,-1)$ is isomorphic to a left-invariant $\varepsilon\eta\,$-Einstein structure on:
\begin{enumerate}
\item  $\widetilde{\mathrm{Sl}}(2,\mathbb{R})$ if $1 \geq \lambda^2>0$.
\item $\mathrm{H}_3$ if $\lambda^2=0$.
\end{enumerate}
%$\widetilde{\mathrm{Sl}}(2,\mathbb{R})$ equipped with a left-invariant $\varepsilon\eta\,$-Einstein structure when $1 \geq \lambda^2>0$ or isomorphic to $\mathrm{H}_3$ equipped with a left-invariant $\varepsilon\eta\,$-Einstein structure when $\lambda^2=0$.
\item $\G$ is of type $\mathfrak{g}_6$ and $1\geq \lambda^2 >0$. 
\end{itemize}
\end{lemma}

\begin{proof}
Let $\{\xi, X, \phi(X)\}$ be a left-invariant $\varepsilon\,$-contact frame on $(\G,\xi,\alpha)$. The proof of Proposition \ref{prop:cdnslc} shows that the following holds: 
\begin{eqnarray*}
& \nabla_{\xi} \xi=0\, , \quad \nabla_{X} \xi=-\frac{1}{2} \phi(X) \, , \quad \nabla_{\phi(X)} \xi=\frac{1}{2} X\, , \\
& \nabla_\xi X=\beta \phi(X)\, , \quad \nabla_X X=c \phi(X)\, , \quad  \nabla_{\phi (X)} X= \frac{1}{2} \xi-e \phi(X) \, ,\\ & \nabla_\xi \phi(X)=-\beta X\, , \quad \nabla_{X} \phi(X)=- \frac{1}{2} \xi-cX \, , \quad \nabla_{\phi(X)} \phi(X)=e X \, , \\
& [\xi,X]=\left (\beta+\frac{1}{2}\right ) \phi(X)\, , \quad [\xi,\phi(X)]= - \left (\beta + \frac{1}{2} \right ) X\, , \quad  [X,\phi(X)]=-\xi-cX+e \phi(X)\, ,
\end{eqnarray*}

\noindent
where $\beta,c,e \in \mathbb{R}$. Imposing the Jacobi identity, we obtain the constraint
\begin{equation*}
e\left (\beta+\frac{1}{2} \right )X+ c \left (\beta+\frac{1}{2}\right )\phi(X) =0\, .
\end{equation*}

\noindent
Hence, either $c=e=0$ or $\beta=-\frac{1}{2}$. We consider these cases separately.

\begin{itemize}
\item Assume $c=e=0$. Imposing that $(\G,g,\alpha) \in \mathrm{PCont}_L^{\varepsilon \eta}(-1,\lambda^2,1-\lambda^2)$, we obtain the condition $\ric^g(X,X)=\ric^g(\phi(X),\phi(X))=\frac{1}{2}-\lambda^2$, which is equivalent to
\begin{equation*}
\beta=\lambda^2-\frac{1}{2}\, .
\end{equation*}

\noindent
Hence, the Lie brackets of $\{\xi, X, \phi(X)\}$  reduce to:
\begin{equation*}
 [\xi,X]=\lambda^2 \phi(X)\, , \quad [\xi,\phi(X)]=-\lambda^2 X\, , \quad  [X,\phi(X)]=-\xi\,.
\end{equation*}

\noindent
For $\lambda^2>0$, if $\{e_0,e_1,e_2\}$ denotes the orthonormal basis used at Theorem \ref{thm:LorentzGroups}, identifying $e_0=\xi$, $e_1=X$ and $e_2=\phi(X)$, we conclude that $(\G,g)$ is isometric to $\widetilde{\mathrm{Sl}}(2,\mathbb{R})$ endowed with a left-invariant metric. Similarly, for $\lambda^2=0$, setting $e_0=-\xi$,  $e_1=X$, $e_2=\phi(X)$ we conclude that $(\G,g)$ is isometric to H$_3$ equipped with a left-invariant metric.

\item If $\beta=-\frac{1}{2}$,  the only non-trivial constraint from the $\varepsilon\eta\,$-Einstein condition is given by $\ric^g(X,X)=\ric^g(\phi(X),\phi(X))=\frac{1}{2}-\lambda^2$, which is equivalent to
\begin{equation*}
c^2+e^2=\lambda^2\, .
\end{equation*}

\noindent
Consequently, we obtain:
\begin{equation*}
[\xi,X]=0\, , \quad [\xi,\phi(X)]=0\, , \quad  [X,\phi(X)]=-\xi-c\phi(X)+eX\, ,
\end{equation*}

\noindent
with $c^2+e^2=\lambda^2$. We assume $\lambda^2 \neq 0$, since otherwise we return to the previous bullet point. Parametrizing $e=-\vert \lambda\vert  \cos \theta$ and $c=\vert \lambda \vert  \sin \theta$ for some angle $\theta \in \mathbb{R}$, we consider the following orthogonal change of basis: 
\begin{equation*}
[\xi, \bar{X}, \phi(\bar{X}) ]^T=\begin{bmatrix}
1 & 0 & 0 \\ 0 & \cos \theta & \sin \theta \\ 0 & -\sin \theta & \cos \theta \\
\end{bmatrix} [\xi, X, \phi(X) ]^T\, ,
\end{equation*}

\noindent
finding:
\begin{equation*}
[\xi,\bar{X}]=0\, , \quad [\xi,\phi(\bar{X})]=0\, , \quad  [\bar{X},\phi(\bar{X})]=-\xi-\vert \lambda \vert  \bar{X}\, .
\end{equation*}

\noindent
Defining $e_0=\xi$, $e_1=\phi(\bar{X})$ and $e_2=\bar{X}$ we conclude, via Theorem \ref{thm:LorentzGroups}, that the previous Lie brackets correspond to those of a Lie algebra type $\mathfrak{g}_6$.
 
\end{itemize}
\end{proof}
\noindent
We may summarize the information provided in Proposition \ref{prop:claslns} and Lemma \ref{lemma:claslsas} in the following Theorem.
\begin{thm}
A three-dimensional connected and simply connected Lie group $G$ admits a left- invariant $\varepsilon\eta\,$-Einstein contact structure $(g,\alpha)$ with time-like Reeb vector field if and only if $(\mathrm{G},g,\alpha)$ is isomorphic, through a possibly orientation-reversing isometry, to one of the items listed in the following table in terms of the orthonormal frame $\left\{ e_0 , e_1 , e_2\right\}$ appearing in Theorem \ref{thm:LorentzGroups}: 

\vspace{0.25cm}
\renewcommand{\arraystretch}{1.5}
\begin{center}
\begin{tabular}{ |P{0.5cm}|P{4.5cm}|P{1.3cm}|P{4cm}|P{1cm} | P{1.5cm}|}
\hline
		
$\mathfrak{g}$  & \emph{Structure constants } & $\alpha$ & $\eta\,$\emph{-Einstein constants} & $G$ & \emph{Sasakian} \\
\hline
	
\multirow{4}*{$\mathfrak{g}_3$} & $\frac{1}{2}>a=1-b > 0\, , c=1$ & $\alpha=e^0$ &  $\lambda^2= \kappa=2b(1-b)$ & $\widetilde{\mathrm{Sl}}(2,\mathbb{R})$ & \emph{No} \\ \cline{2-6} &   $a=c=1\, , b=0$ & $\alpha=-e^0$ & $\lambda^2=\kappa=0$ & $\widetilde{\mathrm{E}}(1,1)$ & \emph{No}  \\ \cline{2-6} & $1 \geq a=b>0\, , c=1$ & $\alpha=e^0$ & $\lambda^2=1-\kappa=a\,$ & $\widetilde{\mathrm{Sl}}(2,\mathbb{R})$ & \emph{Yes} \\ \cline{2-6} & $ a=b=0\, , c=-1$ & $\alpha=-e^0$ & $\lambda^2=0\, , \kappa=1\,$ & $\mathrm{H}_3$ & \emph{Yes} \\ \hline $\mathfrak{g}_6$ & $b=1\, , c=d=0\, , 1 \geq a^2 > 0\,$ & $\alpha=e^0$ & $\lambda^2=1-\kappa=a^2\, ,$ &$\mathfrak{G}_6$  & \emph{Yes} \\ \hline
\end{tabular}
\end{center}

\renewcommand{\arraystretch}{1}

\noindent
Furthermore, if $(M,g,\alpha,-1) \in \mathrm{PCont}_L^{\varepsilon \eta}(-1,\lambda^2,\kappa)$ is complete and not Sasakian then it is a Lie group equipped with a left-invariant Lorentzian contact structure and isomorphic to a left-invariant $\varepsilon\eta\,$-Einstein structure on either $\widetilde{\mathrm{Sl}}(2,\mathbb{R})$ when $\frac{1}{2}> \lambda^2=\kappa>0$ or on $\widetilde{\mathrm{E}}(1,1)$ when $\lambda^2=\kappa=0$.
\end{thm}
\begin{proof}
Follows from Proposition \ref{prop:claslns} and Lemma \ref{lemma:claslsas} upon use of Theorem \ref{thm:LorentzGroups}.
\end{proof}

%%%%%%%%%%%%%%%%%%%%%%%%%%%%%%%%%%%%%%%%%%%%%%%%%%%%%%%%%%%%%%%

\subsection{Lorentzian case with space-like Reeb vector field}
\label{sec:etaparacontact}

%%%%%%%%%%%%%%%%%%%%%%%%%%%%%%%%%%%%%%%%%%%%%%%%%%%%%%%%%%%%%%%

In this subsection we classify all left-invariant $\varepsilon\eta\,$-Einstein para-contact structures on three-dimensional simply-connected Lie groups. Let $(N,\chi,\alpha,\varepsilon = 1)\in \mathrm{PCont}^{\varepsilon \eta}_L(\varepsilon = 1,\lambda,\kappa)$ be an oriented and time-oriented Lorentzian $\varepsilon\,$-contact metric manifold. By Definition \ref{def:Einsteinpcontact}, $(\chi,\alpha,\varepsilon = 1)$ is $\varepsilon\eta\,$-Einstein if and only if it satisfies:
\begin{equation*}
\mathrm{Ric}^{\chi} = - \frac{1}{2}(\lambda^2 + \kappa )\,\chi +  \kappa\,\alpha\otimes \alpha\, ,
\end{equation*} 

\noindent
for real constants $\lambda\in \mathbb{R}$ and $\kappa \geq 0$.  

\begin{remark}
Note that, in contrast to the case $\s_g \varepsilon = 1$, the endomorphism $\mathfrak{h}$ associated to a para-contact metric structure may not be diagonalizable, whence the techniques usted to classify $\varepsilon\eta\,$-Einstein $\varepsilon\,$-contact metric three manifolds with $\s_g \varepsilon = 1$ are a priori not applicable here.
\end{remark}

\begin{thm}
\label{thm:pcetae}
A three-dimensional connected and simply connected Lie group $G$ admits a left- invariant $\varepsilon\eta\,$-Einstein para-contact structure $(g,\alpha)$ if and only if $(\mathrm{G},g,\alpha)$ is isomorphic, through a possibly orientation-reversing isometry, to one of the items listed in the following table in terms of the orthonormal frame $\left\{ e_0 , e_1 , e_2\right\}$ appearing in Theorem \ref{thm:LorentzGroups}:

\vspace{0.25cm}
 \renewcommand{\arraystretch}{1.5}
 \begin{center}
\begin{tabular}{ |P{0.3cm}|P{4.4cm}|P{2.9cm}|P{3.2cm}|P{1cm} | P{1.2cm}|}
				\hline
			
				$\mathfrak{g}$  & \emph{Structure constants} $(s\in \mathbb{Z}_2)$  & $\alpha$ & $\eta\,$\emph{-Einstein constants} & $\mathrm{G}$ & \emph{Sasakian} \\
				\hline

				  \multirow{6}*{$\mathfrak{g}_3$} & $a=s\, , \, b =c\, , \, sc\geq 1 $ & \multirow{1}*{$ \alpha=\pm e^1$} & $\lambda^2= s c\, , \, \kappa=  s c-1$& $\widetilde{\mathrm{Sl}}(2, \mathbb{R})$ & \emph{Yes} \\ \cline{2-6} & $b=s\, , \, a=c \, , \, \, sc\geq 1$ & \multirow{1}*{$ \alpha=\pm e^2$} & $\lambda^2=sc\, , \, \kappa=  s c-1$ & $\widetilde{\mathrm{Sl}}(2, \mathbb{R})$ & \emph{Yes} \\   \cline{2-6}& $b=0\, , a=c=s$ & $-\alpha_0^2+\alpha_1^2=1$ &  $\lambda^2=0\, , \kappa=0$ & $\widetilde{\mathrm{E}}(1,1)$ & \emph{No} \\ \cline{2-6}& $c= 0\, ,\, a=b=s$ & $\alpha_1^2+\alpha_2^2=1$ & $\lambda^2= 0\, , \kappa=0$  & $\widetilde{\mathrm{E}}(2)$ & \emph{No }\\ \cline{2-6}& $a=b=c=s$ & $-\alpha_0^2+\alpha_1^2+\alpha_2^2=1$ & $\lambda^2=1\, , \kappa=0$ & $\widetilde{\mathrm{Sl}}(2, \mathbb{R})$ & \emph{Yes} \\ \hline
				  
				  \multirow{3}*{$\mathfrak{g}_6$} & $a=b=0\, ,$ $d^2\geq 1\, , c=-s$ & $\alpha_2^2=1$ & $\lambda^2=d^2\, , \kappa=d^2-1$ & \multirow{3}*{$\mathfrak{G}_6$} &  	\emph{Yes} \\ \cline{2-4} \cline{6-6}
				  & $b=-\mu a \neq 0 \, , \, d=-a+\mu s \, , \,$ & $a \alpha_2=(-s+\mu a ) \alpha_0$ & \multirow{2}*{$\lambda^2=1\, ,\, \kappa=0\, $} & & \multirow{2}*{\emph{Yes}} \\  & $c=-s+\mu a \, , \, \mu \in \mathbb{Z}_2\,$ & $ \alpha_2^2=1+\alpha_0^2\, $ &  & & \\ \hline

			\end{tabular}
			\end{center}
			
			 \renewcommand{\arraystretch}{1}

\end{thm}

\begin{proof}
We proceed on a case by case basis by checking which of the items appearing in Theorem \ref{thm:LorentzGroups} admits an $\varepsilon\eta\,$-Einstein para-contact metric structure $(g,\alpha)$. For this, we will exploit the formulae presented in Appendix \ref{app:CurvatureLorentzgroups}.

\begin{itemize}
	
\item {\bf Case} $\mathfrak{g}_1$. In the orthonormal frame given in Appendix \ref{app:Lorentzgroups}, Equation $\ast \alpha = -s\,\dd\alpha$ is equivalent to
\begin{eqnarray*}
&(-a \alpha_1 - b \alpha_2)\, e^0\wedge e^1 + (a \alpha_0 +b \alpha_1 + a \alpha_2)\, e^0\wedge e^2 + (b \alpha_0 - a \alpha_1)\, e^1\wedge e^2 \\
& = s\,(\alpha_0\, e^1\wedge e^2 + \alpha_1\, e^0\wedge e^2  - \alpha_2\, e^0\wedge e^1)\, ,
\end{eqnarray*}

\noindent  
Non-trivial solutions for $\alpha$ exist only if $b=s$, which implies $\alpha_1=0$, since $a\neq 0$. This implies in turn $\alpha_0+\alpha_2=0$, which is incompatible with the constraint $-\alpha_0^2+\alpha_2^2=1$. Hence $\mathfrak{g}_1$ does not admit para-contact metric structures.

\item {\bf Case} $\mathfrak{g}_2$. In the orthonormal frame given in Appendix \ref{app:Lorentzgroups}, Equation $\ast \alpha = -s\,\dd\alpha$ is equivalent to
\begin{eqnarray*}
& a \alpha_1 e^0 \wedge e^2+(c \alpha_0-b \alpha_2)\,  e^0 \wedge e^1 + (c \alpha_2 +b \alpha_0)\,  e^1 \wedge e^2\\
& = s\,(\alpha_0 \, e^1 \wedge e^2 +\alpha_1 \, e^0 \wedge e^2-\alpha_2\, e^0 \wedge e^1)\,.
\end{eqnarray*}

\noindent
This system of equations implies $c^2+(b-s)^2=0$, which is equivalent to $c=0$ and $b=s$. Since the value $c=0$ is forbidden for $\mathfrak{g}_2$ Lie algebras, we conclude that there are no para-contact metric structures on this type of Lie algebras. 
 
\item {\bf Case} $\mathfrak{g}_3$. In the orthonormal frame given in Appendix \ref{app:Lorentzgroups}, Equation $\ast \alpha = -s\,\dd\alpha$ is equivalent to
\begin{equation*}
c \alpha_0\,  e^1\wedge e^2+a \alpha_1 \, e^0 \wedge e^2- b \alpha_2 \, e^0 \wedge e^1 = s\,(\alpha_0 \, e^1 \wedge e^2+ \alpha_1 e^0 \wedge e^2 -\alpha_2 e^0 \wedge e^1)\, . 
\end{equation*}

\noindent
If $\alpha_0,\alpha_1,\alpha_2 \neq 0$, then we have $a=b=c=s$ and $-\alpha_0^2+\alpha_1^2+\alpha_2^2=1$. If $\alpha_0=0$ and $\alpha_1, \alpha_2 \neq 0$, then $\alpha_1^2+\alpha_2^2=1$, $a=b=s$ and $c$ is unconstrained. If $\alpha_1=0, \alpha_0,\alpha_2 \neq 0$ (resp. $\alpha_2=0, \alpha_0,\alpha_1 \neq 0$) we have $-\alpha_0^2+\alpha_2^2=1$, $b=c=s$ and $a$ unconstrained (resp. $-\alpha_0^2+\alpha_1^2=1$, $a=c=s$ and $b$ unconstrained). If $\alpha_0=\alpha_1=0$ (resp. $\alpha_0=\alpha_2=0$), then $\alpha^2_2=  1$, $b=s$ and $a,c$ are unconstrained (resp. $\alpha^2_1 = 1$, $a=s$ and $b,c$ unconstrained). Finally, we remark that $a=b=0$ is never allowed, since it implies $\alpha_1=\alpha_2=0$, whence $\alpha_0^2=-1$.

We compute now the Ricci curvature. We obtain:
\begin{eqnarray}
\label{eq:Riccicomputationpara}
\nonumber\ric^g(e_0,e_0)=\frac{c^2}{2}+ba-\frac{b^2}{2}-\frac{a^2}{2}\, , \quad \ric^g(e_0,e_1)=0\, , \\
\ric^g(e_1,e_1)=\frac{b^2}{2}+\frac{c^2}{2}-\frac{a^2}{2}-cb\, , \quad \ric^g(e_0,e_2)=0\, , \\
\nonumber \ric^g(e_2,e_2)=\frac{a^2}{2}-ac+\frac{c^2}{2}-\frac{b^2}{2}\, , \quad \ric^g(e_1,e_2)=0\, .
\end{eqnarray}
 
\noindent
We proceed now on a case by case basis:

\begin{enumerate}
	\item If $\alpha_0,\alpha_1,\alpha_2 \neq 0$, and hence $a=b=c=s$, the $\varepsilon\eta\,$-Einstein condition reduces to
	\begin{equation*}
	\lambda^2 = 1\, , \qquad \kappa = 0\, .
	\end{equation*}
	
	\noindent
	which follows by direct computation from \eqref{eq:Riccicomputationpara}.
	
	\item If $\alpha_0 = 0$ and $\alpha_1,\alpha_2 \neq 0$, we have $a = b = s$ and $c$ unconstrained. Then $\ric^g(e_1,e_2) = 0$ implies $\kappa = 0$, which solves all off-diagonal $\varepsilon\eta\,$-Einstein equations. The diagonal components of the $\varepsilon\eta\,$-Einstein equations are equivalent to
	\begin{equation*}
	\lambda^2 = c^2\, \qquad c^2 = s c\, .
	\end{equation*}
	
	\noindent
	Hence, either $c\neq  0$, which implies $c=s$ and we are back to point (1), or $c=0$, in which case $\lambda = 0$. The cases $\alpha_1 = 0$ ($\alpha_2 = 0$) and $\alpha_0,\alpha_2 \neq 0$ ($\alpha_0,\alpha_1 \neq 0$) follow analogously.
	
%	\item If $\alpha_1 = 0$ and $\alpha_0,\alpha_2 \neq 0$, whence $c = b = s$ and $a$ unconstrained, then $\ric^g(e_1,e_2) = 0$ implies $\kappa = 0$, which solves all off-diagonal $\varepsilon\eta\,$-Einstein equations. The diagonal components of the $\varepsilon\eta\,$-Einstein equations imply:
	%\begin{equation*}
	%\lambda^2 = a^2\, \qquad a^2 = s a\, .
	%\end{equation*}
	
	%\noindent
	%Hence, either $a\neq  0$, which implies $a=s$ and we are back to point (1), or $a=0$, in which case $\lambda = 0$. The case $\alpha_2 = 0$ and $\alpha_0,\alpha_1 \neq 0$ follows identically.
	
	\item We consider now $\alpha_0 = \alpha_1 = 0$ and $\alpha_2^2 = 1$, so that $b=s$ and $a,c$ unconstrained. In this case, using the formulae of Appendix \ref{app:CurvatureLorentzgroups} the $\varepsilon\eta\,$-Einstein condition can be found to imply:
	\begin{equation*}
	(c-a) (s- (c+a)) = 0\, ,
	\end{equation*}
	
	\noindent
	whence either $c=a$ or $a+c = s$. If $a=c$ then the $\varepsilon\eta\,$-Einstein equations are equivalent to $a = s(1+\kappa) = c$ and $\lambda^2 = 1 + \kappa$. Therefore, $\lambda^2 = sc$. On the other hand, if $a+c = s$ then the $\varepsilon\eta\,$-Einstein condition implies $\kappa = -\lambda^2$. However, $\kappa \geq 0$ by the definition of $\varepsilon\eta\,$-Einstein structure, whence $\lambda= \kappa = 0$ and the $\varepsilon\eta\,$-Einstein condition reduces to $a(a-s) = 0$. Since $a=0$ is not allowed if $b=s$ by the type of algebra $\mathfrak{g}_3$, we conclude that $a=s$. The case $\alpha_0 = \alpha_2 = 0$ and $\alpha_1^2 = 1$ follows now along similar lines.
\end{enumerate}

\item  {\bf Case} $\mathfrak{g}_4$. In the orthonormal frame given in Appendix \ref{app:Lorentzgroups}, Equation $\ast \alpha = -s\,\dd\alpha$ is equivalent to
\begin{eqnarray*}
& a \alpha_1 \, e^0 \wedge e^2 +(-(2\mu - b)\alpha_0+\alpha_2) \,  e^1 \wedge e^2+(\alpha_0-b \alpha_2 ) \, e^0 \wedge e^1 \\
& = s\,(\alpha_0 \, e^1 \wedge e^2+\alpha_1\, e^0 \wedge e^2-\alpha_2\, e^0 \wedge e^1)\, .
\end{eqnarray*}

\noindent 	
We distinguish the cases $\alpha_2 = 0$ and $\alpha_2 \neq 0$. If $\alpha_2 = 0$ then we must have $\alpha_0 = 0$ and $\alpha^2_1 = 1$ and $a=s$ is the unique solution. If $\alpha_2 \neq 0$ and $b=s$ then $\alpha_0 = 0$ and $\alpha_2 = 0$, a contradiction. Assume then that $\alpha_2 \neq 0$ and $b\neq s$. It follows that $\alpha_0 \neq 0$ and
\begin{equation*}
b^2-2b(s+\mu)+2(s\mu+ 1)=0\, .
\end{equation*}

\noindent
The previous equation has the unique solution $b=\mu + s$, which indeed satisfies $b\neq s$. Then, the solutions in this case are given by:
\begin{equation*}
\alpha^2_1 = 1\, , \qquad a = s\, , \qquad \alpha_0 = \mu \alpha_2\, .
\end{equation*}

\noindent
To classify which para-contact structures are also $\varepsilon\eta\,$-Einstein, we proceed as in the previous cases by direct computation. We obtain (imposing $a=s$):
\begin{eqnarray*}
& \ric^g(e_0,e_0)=\frac{(2\mu-b)^2}{2} - \frac{(b-s)^2}{2}\, , \qquad \ric^g(e_0,e_1)=0\, , \\
& \ric^g(e_1,e_1)= -\frac{1}{2}\, , \qquad \ric^g(e_0,e_2)= s + 2 (\mu-b)\, , \\
& \ric^g(e_2,e_2)=\frac{(2\mu-b)^2}{2} + \frac{1}{2} + s (2\mu-b) - \frac{b^2}{2}\, , \qquad \ric^g(e_1,e_2)= 0\,.
\end{eqnarray*}

\noindent
We distinguish again between the cases $\alpha_2 = 0$ and $\alpha_2\neq 0$ with $b\neq s$.

\begin{enumerate}
	\item If $\alpha_2 = 0$ then $\alpha_0 = 0$ and $a = s$. The only non-trivial off-diagonal component of the $\varepsilon\eta\,$-Einstein condition is $b = \mu + \frac{s}{2}$. Imposing this condition in the \emph{time-like} diagonal component we obtain
	\begin{equation*}
	\ric^g(e_0,e_0) = 0  = \frac{1}{2} (\lambda^2 + \kappa)\, , 
	\end{equation*}
	
	\noindent
	whence $\kappa = - \lambda^2$. Since $\kappa \geq 0$ we conclude that $\lambda = \kappa = 0$. However, this is incompatible with $\ric^g(e_1,e_1)= -\frac{1}{2}$.
	
	\item If $\alpha_2 \neq 0$ and $b=\mu+ s$ then equation $\ric^g(e_1,e_2)=0$ implies $\kappa =  0$ since $\alpha_1 \neq 0$ and $\alpha_2 \neq 0$. However, then we would need $\ric^g(e_0,e_2)=-s$ to vanish, what is not possible. 
\end{enumerate}

\item {\bf Case} $\mathfrak{g}_5$. In the orthonormal frame given in Appendix \ref{app:Lorentzgroups}, Equation $\ast \alpha = -s\,\dd\alpha$ is equivalent to
\begin{equation*}
(a \alpha_1 + b \alpha_2 )\, e^0\wedge e^1 + (c\alpha_1  + d \alpha_2 )\, e^0\wedge e^2 = s\, (\alpha_0\, e^1\wedge e^2 + \alpha_1\, e^0\wedge e^2  - \alpha_2\, e^0\wedge e^1)\, ,
\end{equation*}

\noindent
which immediately implies $\alpha_0 = 0$. The conditions for a para-contact structure to exist are $\alpha_1^2+\alpha_2^2=1$ and $ad=(b+s)(c-s)$, together with the conditions on the coefficients required by the algebra type $\mathfrak{g}_5$, which are $ac+bd=0$ and $a+d\neq 0$.

\noindent
We verify now which para-contact structures on $\mathfrak{g}_5$ are $\varepsilon\eta\,$-Einstein. Using Appendix \ref{app:CurvatureLorentzgroups} we obtain the following components for the Ricci tensor:
\begin{eqnarray*}
& \ric^g(e_0,e_0)=-a^2-d^2-cb-\frac{b^2}{2}-\frac{c^2}{2}\, , \quad \ric^g(e_0,e_1)=0\, , \\
& \ric^g(e_1,e_1)=a^2+ad-\frac{c^2}{2}+\frac{b^2}{2}\, , \quad \ric^g(e_0,e_2)=0\, , \\
& \ric^g(e_2,e_2)=d^2+da+\frac{c^2}{2}-\frac{b^2}{2}\, , \quad \ric^g(e_1,e_2)=ac+bd  \,.
\end{eqnarray*}

\noindent
We have, $\ric^g(e_1,e_2) =0$ automatically by the definition of algebra of type $\mathfrak{g}_5$. Hence, the $\varepsilon\eta\,$-Einstein condition evaluated in $e_1$ and $e_2$ implies:
\begin{equation*}
\alpha_1 \alpha_2 \kappa = 0\, .
\end{equation*}

\noindent
On the other hand, since $\alpha^2_1 +\alpha^2_2 = 1$, the $\varepsilon\eta\,$-Einstein equation implies:
\begin{equation*}
\ric^g(e_1,e_1) + \ric^g(e_2,e_2) = (a+d)^2 = -\lambda^2\, .
\end{equation*}

\noindent
Since $a+d\neq 0$ by the coefficient conditions of the algebra of type $\mathfrak{g}_5$, the previous equation admits no solutions and therefore an algebra of type $\mathfrak{g}_5$ does not admit $\varepsilon\eta\,$-Einstein para-contact structures.

\item {\bf Case} $\mathfrak{g}_6$. In the orthonormal frame given in Appendix \ref{app:Lorentzgroups}, Equation $\ast \alpha = -s\,\dd\alpha$ is equivalent to
\begin{equation}
\label{eq:caseg6para}
(d\alpha_0+c\alpha_2) \, e^0 \wedge e^1 +(-b \alpha_0 -a\alpha_2) \, e^1 \wedge e^2= s\,(\alpha_0 \, e^1 \wedge e^2 +\alpha_1 \, e^0 \wedge e^2-\alpha_2 \, e^0 \wedge e^1)\, ,
\end{equation}

\noindent	
which immediately implies $\alpha_1=0$. From the previous (linear) equations we obtain that the conditions to have a non-trivial para-contact structures are $-\alpha_0^2+\alpha_2^2=1$ and $ad=(b+s)(c+s)$, together with the conditions $ac-bd=0$ and $a+d \neq 0$ required by the algebra type $\mathfrak{g}_6$. The components of the Ricci tensor read:
\begin{eqnarray*}
& \ric^g(e_0,e_0)=d^2+ad+\frac{b^2}{2}-\frac{c^2}{2}\, , \quad \ric^g(e_0,e_1)=0\, , \\
& \ric^g(e_1,e_1)=-a^2-d^2-bc+\frac{c^2}{2}+\frac{b^2}{2}\, , \quad \ric^g(e_0,e_2)=-ac+bd = 0\, , \\
& \ric^g(e_2,e_2)=-a^2-da-\frac{c^2}{2}+\frac{b^2}{2}\, , \quad \ric^g(e_1,e_2)=0\, .
\end{eqnarray*}

\noindent
We have $\ric^g(e_0,e_2) = 0$ identically by the conditions on the coefficients of an algebra of type $\mathfrak{g}_6$. Evaluating the $\varepsilon\eta\,$-Einstein condition on $e_0$ and $e_2$ we obtain
\begin{equation*}
\kappa \,\alpha_0 \alpha_2 = 0\, .
\end{equation*}

\noindent
Hence, either $\alpha_0 = 0$ or $\kappa= 0$ since $\alpha_2 = 0$ is not allowed by the para-contact condition. If $\alpha_0 = 0$ then $\alpha_2^2 = 1$, which implies $a = 0$ and $c = - s$. Taking into account that $ac-bd=0$ and $a+d \neq 0$, we further obtain that $b=0$, which in turn implies:
\begin{equation*}
\lambda^2 = d^2 = \kappa + 1\, , 
\end{equation*}

\noindent
whence $d^2\geq 1$ since we must have $\kappa \geq 0$. Altogether these conditions solve the $\varepsilon\eta\,$-Einstein equations of a para-contact structure on $\mathfrak{g}_6$. On other hand, if $\kappa=0$, a combination of the $\varepsilon\eta\,$-Einstein equations implies:
\begin{equation*}
ad=(b+s)(c+s)\, ,\quad ac-bd=0\, , \quad a^2-d^2=b^2-c^2\, , \quad (a-d)^ 2 = (b-c)^2\, .
\end{equation*}

\noindent
The last equation above reduces to
\begin{equation*}
(c-b)=\mu(a-d)\,
\end{equation*} 

\noindent
where $\mu \in \mathbb{Z}_2$. We distinguish now two cases:

\begin{enumerate}
	\item If $a = d \neq 0$ (using that $a+d\neq 0$) we obtain $c=b$ and the $\varepsilon\eta\,$-Einstein equations are equivalent to
	\begin{equation*}
	a^2 = \frac{\lambda^2}{4}\, .
	\end{equation*}
	
	\noindent
	Likewise, the coefficient conditions required by the algebra of type $\mathfrak{g}_6$ are
	\begin{equation*}
    a^2 = (b+s)^2 \, ,
	\end{equation*}
	
	\noindent
	implying $a = \sigma (b+s)$ for a sign $\sigma \in \mathbb{Z}_2$. Note that $b\neq s$ since $a+d =2a\neq 0$. Plugging $a = \sigma (b+s)$ in Equation \eqref{eq:caseg6para} we obtain:
	\begin{equation*}
	\alpha_0^2 = \alpha_2^2\, ,
	\end{equation*}
 
	\noindent
	which is incompatible with the para-contact condition $-\alpha_0^2 + \alpha_2^2 = 1$. 
	
	\item If $a \neq d$, equation $ a^2-d^2=b^2-c^2$ is equivalent to
	\begin{equation*}
	a + d = - \mu (b+c)\, .
	\end{equation*}
	
	\noindent
	Combining this equation with $(c-b)=\mu(a-d)$ we obtain $a =-\mu b$ and $d = -\mu c$. The constraints on the coefficients required by the algebra of type $\mathfrak{g}_6$ are found to reduce to
	\begin{equation*}
	1+s(b+c) = 0\, ,
	\end{equation*}
	
	\noindent
	whereas the $\varepsilon\eta\,$-Einstein equations are tantamount to
	\begin{equation*}
	\lambda^2 = (b+c)^2\, .
	\end{equation*}
	
	\noindent
	Therefore, using that $1+s(b+c) = 0$ we obtain $\lambda^2 = 1$. Equation \eqref{eq:caseg6para} is solved by:
	\begin{equation*}
	b = - \mu a\, , \quad c = -s + \mu a\, , \quad d = \mu s - a\, , \quad (-s+\mu a)\alpha_0=a \alpha_2\, , \quad \alpha^2_2 = 1 + \alpha_0^2\, .
	\end{equation*}
\end{enumerate}

\item {\bf Case} $\mathfrak{g}_7$. In the orthonormal frame given in Appendix \ref{app:Lorentzgroups}, Equation $\ast \alpha = -s\,\dd\alpha$ is equivalent to
\begin{equation*}
\begin{split}
(b \alpha_0+a \alpha_1+b \alpha_2) \, e^0 \wedge e^1&+ (d\alpha_0+c \alpha_1 +d \alpha_2) \, e^0 \wedge e^2 +(b \alpha_0+a \alpha_1+b \alpha_2) \, e^1 \wedge e^2\\&= s\,(\alpha_0 e^1 \wedge e^2+\alpha_1 e^0 \wedge e^2 -\alpha_2 e^0 \wedge e^1)\, .
\end{split}
\end{equation*}

\noindent
which immediately implies $\alpha_0+\alpha_2=0$, whence $\alpha^2_1= 1$. Write $\alpha_1 = \sigma$, with $\sigma \in \mathbb{Z}_2$. With these assumptions, the previous equations boil down to
\begin{equation*}
\alpha_0 = s \sigma a \, , \qquad c = s\, , \qquad \alpha_2^2 = a^2\, ,
\end{equation*}

\noindent
where $\mu \in \mathbb{Z}_2$. Since $c\neq 0$, then $a=0$ from the condition $ac = 0$ required by the algebra of type $\mathfrak{g}_7$. Hence $\alpha_0 = \alpha_2 = 0$. With these provisos in mind, the Ricci curvature reads:
\begin{eqnarray*}
& \ric^g(e_0,e_0)=-bs-\frac{1}{2}\, , \quad \ric^g(e_0,e_1)=0\, , \\
& \ric^g(e_1,e_1)=-\frac{1}{2}\, , \quad \ric^g(e_0,e_2)= s b\, , \\
& \ric^g(e_2,e_2)=\frac{1}{2}-sb\, , \quad \ric^g(e_1,e_2)=0\,.
\end{eqnarray*}

\noindent
However, since $\alpha=\sigma\, e^1$, we obtain $b=0$, which implies $\kappa=-1$,  a value that is not permitted by Definition \ref{def:Einsteinpcontact}. 
\end{itemize}

\noindent
Finally, to verify which of the $\varepsilon\eta\, $-Einstein para-contact structures are Sasakian we apply first Proposition \ref{prop:ricxien0}, which states that $\varepsilon \eta\,$-Einstein para-contact structures on $\widetilde{\mathrm{E}}(1,1)$ or $\widetilde{\mathrm{E}}(2)$ can never be Sasakian. Also, a direct computation shows that $\mathfrak{h}$ vanishes for every $\varepsilon \eta\,$-Einstein para-contact structure on $\widetilde{\mathrm{Sl}}(2, \mathbb{R})$ and $\mathfrak{g}_6$.
\end{proof}

%%%%%%%%%%%%%%%%%%%%%%%%%%%%%%%%%%%%%%%%%%%%%%%%%%%%%%%%%%%%%%%

\subsection{Lorentzian case with null Reeb vector field}
\label{sec:etanull}

%%%%%%%%%%%%%%%%%%%%%%%%%%%%%%%%%%%%%%%%%%%%%%%%%%%%%%%%%%%%%%%

When $\varepsilon = 0$ the $\varepsilon\eta\,$-Einstein condition for an $\varepsilon\,$-contact metric structure $(\chi,\alpha)$ reduces to
\begin{equation*}
\mathrm{Ric}^{\chi} = - \frac{\lambda^2}{2}\,\chi +  \kappa\,\alpha\otimes \alpha\, ,
\end{equation*}

\noindent
with $\kappa \geq 0$. 

\begin{remark}
To the best of our knowledge, this equation has not been considered in the literature. In particular, the methods and techniques used in \cite{BlairKoufoII,Koufo} to classify $(\kappa,\mu)$ and $\varepsilon\eta\,$-Einstein contact three manifolds do not seem to apply in this case, due to the fact that $\phi^2$ is not an isomorphism when restricted to the kernel of $\alpha$ and that $\mathfrak{h}$ cannot have non-zero eigenvalues, see Remarks \ref{remark:covxi} and \ref{remark:hnull}.
\end{remark}

\noindent
The goal of this subsection is to classify all left-invariant $\varepsilon\eta\,$-Einstein null contact structures on a simply connected three-dimensional Lie group $\G$. In order to do this, we will make use of the following lemma.

\begin{lemma}
\label{lemma:etaefnc}
Let $(\chi, \alpha)$ be a null contact metric structure and let $\{\xi, u, \phi(u)\}$ be a light-cone frame. Then $(\xi,\alpha)$ is $\varepsilon\eta\,$-Einstein if and only if
\begin{eqnarray*}
& \emph{Ric}^\chi(\xi, \xi)=\emph{Ric}^\chi(\xi, \phi(u))=\emph{Ric}^\chi(u,\phi(u))=0\, , \\ 
& \emph{Ric}^\chi(\xi, u)=\emph{Ric}^\chi(\phi(u),\phi(u))=-\frac{\lambda^2}{2}\, , \quad \emph{Ric}^\chi (u,u)=\kappa\, .
\end{eqnarray*}
\end{lemma}
 \begin{proof}
Follows by direct computation.
 \end{proof}
\begin{thm}
\label{thm:etanullcontactstructures}
A three-dimensional connected and simply connected Lie group $\mathrm{G}$ admits a left-invariant $\varepsilon\eta\,$-Einstein null contact structure $(g,\alpha)$ if and only if $(\mathrm{G},g,\alpha)$ is isomorphic, through a possibly orientation-reversing isometry, to one of the items listed in the following table in terms of the orthonormal frame $\left\{ e_0 , e_1 , e_2\right\}$ appearing in Theorem \ref{thm:LorentzGroups}:

\vspace{0.25cm}
 \renewcommand{\arraystretch}{1.5}
 \begin{center}
\begin{tabular}{ |P{0.5cm}|P{4.4cm}|P{2cm}|P{3.7cm}|P{1cm} |P{1.4cm}| }
				\hline
			
				$\mathfrak{g}$  & \emph{Structure constants} $(s\in \mathbb{Z}_2)$  & $\alpha$ & $\eta\,$\emph{-Einstein constants} & $\mathrm{G}$ & \emph{Sasakian}\\
				\hline

				 \multirow{3}*{$\mathfrak{g}_3$} & $a=b=c=s$ & $ \alpha_0^2=\alpha_1^2+\alpha_2^2$ & $\lambda^2=1\, , \, \kappa=0$ &  $\widetilde{\mathrm{Sl}}(2,\mathbb{R})$  & \emph{Yes} \\ \cline{2-6} &  \multirow{2}*{$a=c=s\, , \, b=0$} &  $\alpha_2=0\, ,$   & \multirow{2}*{$\lambda^2=0\, , \, \kappa=0$} &  \multirow{2}*{$\widetilde{\mathrm{E}}(1,1)$}  & \multirow{2}*{\emph{No}}   \\ & &  $\alpha_0^2=\alpha_1^2$ & & & \\  \hline
			 %\multirow{2}*{$\mathfrak{g}_4$}    & $b=1 + z \, , a=1\,  , z \in \mathbb{Z}_2$ & \multirow{2}*{$\alpha_1=0\, , \alpha_0=z\alpha_2$} & $\lambda^2=1\, , \kappa=-2\sigma^2 z$ & $\widetilde{\mathrm{Sl}}(2,\mathbb{R})$ \\ \cline{2-2} \cline{4-5} & $b=1 + z\, , a=0\,  , z \in \mathbb{Z}_2$  & &$\lambda^2=0\, , \kappa=-4\sigma^2 z$ & $\widetilde{\mathrm{E}}(1,1)$\\	 
			 \multirow{2}*{$\mathfrak{g}_4$}    & $b=0 \, , \, a=s\, $ &$\alpha_1=0\, , \,$ & $\lambda^2=1\, , \, \alpha_0^2 \kappa=1$ & $\widetilde{\mathrm{Sl}}(2,\mathbb{R})$ &  \emph{Yes}\\ \cline{2-2} \cline{4-6} & $b=0\, , a=0\,  , \, $  & $\alpha_0=\mu  \alpha_2$  &$\lambda^2=0\, , \,\alpha_0^2 \kappa=2 $ & $\widetilde{\mathrm{E}}(1,1)$ & \emph{No} \\	 
			 \hline

				\multirow{2}*{$\mathfrak{g}_6$}  & $a=d\neq 0\, , b=c$  &$\alpha_1=0\, ,$    & $\lambda^2=4a^2\, ,$ & \multirow{2}*{$\mathfrak{G}_6$} & \multirow{2}*{\emph{If} $a=\frac{\mu s}{2}$} \\ & $a=\mu(b+s)\, , \mu \in \mathbb{Z}_2 \,$ & $\alpha_0=-\mu \alpha_2 $ & $\kappa=0$ & & \\ \hline
				
			\end{tabular}
			\end{center}
			
			 \renewcommand{\arraystretch}{1}
			
\end{thm}

\

\begin{proof}
Proposition \ref{prop:nullcontactstructures} classifies all simply-connected Lorentzian Lie groups admitting left-invariant null contact structures. Hence, we will proceed by verifying which of the cases appearing in Proposition \ref{prop:nullcontactstructures} satisfy the $\varepsilon\eta\,$-Einstein equation. For this, we will use the formulae presented in Appendix \ref{app:CurvatureLorentzgroups}, where the Ricci tensor is computed on a global orthonormal frame. 

\begin{itemize}
	
\item {\bf Case} $\mathfrak{g}_1$. The Reeb vector is given by $\xi = - \alpha_0 (e_0 + e_2)$. A direct computation using Appendix \ref{app:CurvatureLorentzgroups} shows that $\ric^\chi (u, \phi(u))=-\frac{s a}{\alpha_0}$. Since $a \neq 0$ by the definition of algebra of type $\mathfrak{g}_1$, Lemma \ref{lemma:etaefnc} implies that $\mathfrak{g}_1$ does not admit null $\varepsilon\eta\,$-Einstein structures.

\item {\bf Case} $\mathfrak{g}_3$. We distinguish between the cases $\alpha_1=0$, $\alpha_2=0$ and $\alpha_1, \alpha_2 \neq 0$. If $\alpha_1=0$, a light-cone frame is given by $\xi= \alpha_0 (-e_0 + \mu e_2), u=\frac{1}{2\alpha_0}(e_0 + \mu e_2)$ and $\phi(u)=\mu e_1$. We obtain: 
\begin{eqnarray*}
& \ric^\chi(\xi,\xi)=\ric^\chi(\xi, \phi(u))=\ric^\chi(u, \phi(u))=0 \\ 
& \ric^\chi(\xi,u)=-a+\frac{a^2}{2}\, , \quad \ric^\chi(\phi(u), \phi(u))=-\frac{a^2}{2}\, , \quad \ric^\chi(u,u)=0\, .
\end{eqnarray*}

\noindent
Hence, the $\varepsilon\eta\,$-Einstein implies $a=0$ or $a=s$. Since $a=0$ is not allowed, we conclude $a=s$, which in turn implies $\lambda^2=1$ and $\kappa=0$.

If $\alpha_2 =0$, a similar analysis follows. In this case, a light-cone frame is given by $\xi=\alpha_0 (-e_0 + \mu e_1), u=\frac{1}{2 \alpha_0}(e_0 + \mu e_1)$ and $\phi(u)= -\mu e_2$. We obtain:
\begin{eqnarray*}
& \ric^\chi(\xi,\xi)=\ric^\chi(\xi, \phi(u))=\ric^\chi(u, \phi(u))=0 \\ 
& \ric^\chi(\xi,u)=-bs+\frac{b^2}{2}\, , \quad \ric^\chi(\phi(u), \phi(u))=-\frac{b^2}{2}\, , \quad \ric^\chi(u,u)=0\, ,
\end{eqnarray*}

\noindent
which implies either $b=s$ or $b=0$. For $b=s$, $\lambda^2=1$ and $\kappa=0$, and for $b=0$, $\lambda^2=\kappa=0$.

\noindent
Finally, if $\alpha_1, \alpha_2 \neq 0$, we directly obtain $\ric^\chi=-\frac{1}{2}\chi$ and every null contact structure of this type is $\varepsilon\eta\,$-Einstein (with $\lambda^2=1$ and $\kappa=0$).

\item {\bf Case} $\mathfrak{g}_4$. We choose a light-cone frame $\{\xi, u, \phi(u)\}$ with  $\xi= \alpha_0 (-e_0+\mu e_2), u=\frac{1}{2\alpha_0}(e_0 +\mu e_2)$ and $\phi(u)= \mu e_1$. We compute:
\begin{eqnarray*}
& \ric^\chi(\xi,\xi)=\ric^\chi(\xi, \phi(u))=\ric^\chi(u, \phi(u))=0 \\ 
& \ric^\chi(\xi,u)=\frac{a^2}{2}-as\, , \quad \ric^\chi(\phi(u), \phi(u))=-\frac{a^2}{2}\, , \quad \ric^\chi(u,u) =\frac{\mu}{\alpha_0^2} (a-2s)\, .
\end{eqnarray*} 

\noindent
These conditions are satisfied if and only if $a=0$ or $a=s$. For $a=s$, $\lambda^2=1$ and $\kappa=-\frac{s\mu}{\alpha_0^2}$, and for $a=0$, $\lambda^2=0$ and $\kappa=-2\frac{s\mu}{\alpha_0^2}$. Since $\kappa$ must be non-negative by definition, we must require that $s \mu=-1$, which implies that $b=0$.

\item {\bf Case}  $\mathfrak{g}_6$. We can choose the light-cone frame $\{\xi, u, \phi(u)\}$ with  $\xi= -\alpha_0 (e_0 + \mu e_2), u=\frac{1}{2\alpha_0}(e_0 -\mu  e_2)$ and $\phi(u)=-\mu e_1$. Imposing the constraints found in Proposition \ref{prop:nullcontactstructures} for null contact structures on $\mathfrak{g}_6$, we get the following components for the Ricci curvature: 
\begin{eqnarray*}
& \ric^\chi(\xi, \phi(u))=\ric^\chi(u, \phi(u))=0\, , \quad \ric^\chi(\xi,\xi)=0\, ,  \\ 
& \ric^\chi(\xi,u)=-2a^2\, , \quad \ric^\chi(\phi(u), \phi(u))=-2a^2\, , \quad \ric^\chi(u,u)=0\, .
\end{eqnarray*}
These equations yield an $\varepsilon\eta\,$-Einstein structure with $\lambda^2=4a^2 \neq 0$ and $\kappa=0$. 
\end{itemize}
Finally, in order to determine when the different $\varepsilon\eta\,$-Einstein structures obtained are Sasakian, we just have to make use of Proposition \ref{prop:liegroupsasakinc}. 
\end{proof}
 
%%%%%%%%%%%%%%%%%%%%%%%%%%%%%%%%%%%%%%%%%%%%%%%%%%%%%%%%%%%%%%%
%%%%%%%%%%%%%%%%%%%%%%%%%%%%%%%%%%%%%%%%%%%%%%%%%%%%%%%%%%%%%%%

\section{Six-dimensional supergravity and $\varepsilon\,$-contact structures}
\label{sec:6dsugrasolutions}

%%%%%%%%%%%%%%%%%%%%%%%%%%%%%%%%%%%%%%%%%%%%%%%%%%%%%%%%%%%%%%%
%%%%%%%%%%%%%%%%%%%%%%%%%%%%%%%%%%%%%%%%%%%%%%%%%%%%%%%%%%%%%%%

In the following, let $M$ be an oriented and spin six-dimensional manifold.

\begin{definition}
The {\bf bosonic configuration space} of six-dimensional minimal supergravity coupled to a tensor multiplet with constant dilaton on $M$ is defined as the following set:
\begin{equation*}
\Conf(M) \eqdef \left\{ (\g,\H)\in \mathrm{Lor}(M)\times\Omega^3(M)\right\}\, ,
\end{equation*}

\noindent
where $\mathrm{Lor}(M)$ denotes the set of Lorentzian metrics on $M$.
\end{definition} 

\noindent
Given $(\g,\H)\in \Conf(M) $ we define $\nabla^{\H}$ to be the unique metric-compatible connection on $(M,\g)$ with totally skew-symmetric torsion given by $\H\in \Omega^3(M)$. In more explicit terms we have
\begin{equation*}
\nabla^{\H} = \nabla + \frac{1}{2} \g^{-1} \H\, ,
\end{equation*}

\noindent
where $\nabla$ denotes the Levi-Civita connection associated to $\g$.

\begin{definition}
A pair $(\g,\H)\in \Conf(M)$ is a {\bf bosonic solution} of six-dimensional minimal supergravity coupled to a tensor multiplet with constant dilaton on $M$ if:
\begin{equation}
\label{eq:q:Riccif}
\mathrm{Ric}(\nabla^{\H}) = 0\, , \qquad \dd \H = 0 \, , \qquad \dd\ast_\g \H = 0\, , \qquad \vert \mathrm{H}\vert^2_\g = 0\, ,
\end{equation}	

\noindent
where $\mathrm{Ric}(\nabla^\H)\in \Gamma( \Sym^2(T^{\ast}M))$ is the Ricci curvature tensor of $\nabla^\H$ and $\ast_\g\colon \Omega^{3}(M)\to \Omega^3(M)$ denotes the Hodge dual associated to $\g$. We denote by $\Sol(M)\subset \Conf(M)$ the set of solutions on $M$.
\end{definition}

\begin{remark}
In six Lorentzian dimensions the Hodge dual $\ast_\g$ on three-forms squares to the identity. Hence, we obtain the splitting
\begin{equation*}
\Lambda^3(M) = \Lambda^3_{+}(M)  \oplus \Lambda^3_{-}(M) \, ,
\end{equation*}

\noindent
in terms of self dual $\Lambda^3_{+}(M)$ and anti-self dual $\Lambda^3_{-}(M)$ three-forms. Using this decomposition, a particular class of solutions of Equations \eqref{eq:q:Riccif} is obtained by requiring $\H$ to be self-dual, that is:
\begin{equation*}
\mathrm{Ric}(\nabla^{\H}) = 0\, , \qquad \dd \H = 0 \, , \qquad \ast_\g \H = \H\, ,
\end{equation*}

\noindent
This set of equations define the bosonic equations of six-dimensional \emph{minimal supergravity}. Equations \eqref{eq:q:Riccif} are more general and allow, for instance, $\H$ to be a section of a fixed lagrangian distribution of $\Lambda^3(M)$. The possibility of generalizing minimal supergravity to this situation was proposed in \cite{Garcia-Fernandez:2015lsa} and it remains, to the best of our knowledge, up for debate.
\end{remark}

\begin{thm}
\label{thm:pcontactsugrasolution}
Let:
\begin{equation*}
(N,\chi,\alpha_N,\varepsilon_N) \in \mathrm{PCont}^{\varepsilon \eta}_L(\varepsilon_N,\lambda^2,\kappa_N = l^2)\, , \qquad (X,h,\alpha_X)\in \mathrm{PCont}^{\varepsilon\eta}_R(\lambda^2,\kappa_X = \vert\alpha_N\vert^2\,l^2)\, .
\end{equation*}

\noindent
Then, the oriented Cartesian product manifold
\begin{equation*}
M = N\times X\, 
\end{equation*}

\noindent
carries a family of solutions $(\g,\H_{\lambda, l})\in \Sol(M)$ of six-dimensional minimal supergravity coupled to a tensor multiplet with constant dilaton given by
\begin{equation}
\label{eq:thmsolution}
\g = \chi \oplus h\, , \qquad \H_{\lambda, l} = \lambda\, \nu_{\chi}  + \frac{l}{3}\,(\ast_\chi \alpha_N)\wedge \alpha_X +  \frac{l}{3}\,\alpha_N\wedge (\ast_h \alpha_X)  + \lambda\, \nu_h\, 
\end{equation}

\noindent
and parametrized by $\, (\lambda,l)\in \mathbb{R}^2$. Equivalently, the oriented Cartesian Lorentzian product of $(N,\chi,\alpha_N,\varepsilon_N) \in \mathrm{PCont}^{\varepsilon \eta}_L(\varepsilon_N,\lambda,l)$ and $(X,h,\alpha_X)\in \mathrm{PCont}^{\varepsilon \eta}_R(\lambda,\kappa = \varepsilon_N,l)$ carries a bi-parametric family of metric-compatible, Ricci-flat, connections with totally skew-symmetric, isotropic, closed and co-closed torsion prescribed by $\H_{\lambda, l}$.
\end{thm}

%\begin{remark}
%Expression \eqref{eq:thmsolution} together with the specific definition of $\varepsilon\eta\,$-Einstein $\varepsilon\,$-contact structure may appear as being God-given. They were however obtained after a careful study of the equations of motion of six-dimensional supergravity coupled to a tensor multiplet with constant dilaton on a product manifold.
%\end{remark}

\begin{proof}
We first compute that $\H_{\lambda, l}$, as prescribed in the statement of the theorem, is closed:
\begin{eqnarray*}
\dd \H_{\lambda, l} = \frac{l}{3}\,\dd(\ast_\chi \alpha_N)\wedge \alpha_X + \frac{l}{3} (\ast_\chi \alpha_N)\wedge \dd\alpha_X + \frac{l}{3}\, \dd\alpha_N\wedge (\ast_h \alpha_X) - \frac{l}{3}\, \alpha_N\wedge \dd (\ast_h \alpha_X) \\ = \frac{l}{3} (\ast_\chi \alpha_N)\wedge \dd\alpha_X + \frac{l}{3}\, \dd\alpha_N\wedge (\ast_h \alpha_X) = \frac{l}{3} \ast_\chi \alpha_N\wedge \ast_h\alpha_X - \frac{l}{3} \ast_{\chi}\alpha_N\wedge \ast_h \alpha_X = 0\, ,
\end{eqnarray*}

\noindent
where we have used that $\dd\alpha_N= -\ast_{\chi}\alpha_N$ and $\dd\alpha_X = \ast_{h}\alpha_X$ by the $\varepsilon\,$-contact condition. In addition, $\H_{\lambda, l}$ is co-closed:
\begin{eqnarray*}
\dd\ast \H_{\lambda, l} = \lambda\, \dd\ast\nu_{\chi}  + \frac{l}{3} \dd\ast (\ast_\chi \alpha_N\wedge \alpha_X) +  \frac{l}{3} \dd\ast (\alpha_N\wedge \ast_h \alpha_X)  + \lambda\,\dd\ast\nu_h\\ =  -\lambda\, \dd\nu_{h}  + \frac{l}{3}\, \dd(\alpha_N\wedge \ast_h\alpha_X) + \frac{l}{3} \dd(\ast_{\chi} \alpha_N\wedge\alpha_X)  - \lambda\,\dd\nu_{\chi} = 0\, ,
\end{eqnarray*}

\noindent
where we have used again the $\varepsilon\,$-contact condition and the fact that, for $\rho \in \Omega^q(N)$ and $\sigma \in \Omega^r(X)$, we have $\ast(\rho \wedge \sigma)=(-1)^{r (3-q)} \ast_\chi \rho \wedge \ast_h \sigma$. We verify now the norm of $\H_{\lambda, l}$ indeed vanishes:
\begin{eqnarray*}
&\vert\H_{\lambda, l}\vert^2_\g  = \lambda^2 \g(\nu_{\chi},\nu_{\chi}) +   \frac{l^2}{9}\g( (\ast_\chi \alpha_N)\wedge \alpha_X,(\ast_\chi \alpha_N)\wedge \alpha_X)  +  \frac{l^2}{9}\g(\alpha_N\wedge (\ast_h \alpha_X), \alpha_N\wedge (\ast_h \alpha_X)) \\ 
&+\lambda^2 \g(\nu_{h},\nu_{h}) = \lambda^2 \chi(\nu_{\chi},\nu_{\chi}) +   \frac{l^2}{3}\chi(\ast_\chi \alpha_N,\ast_\chi \alpha_N) h(\alpha_X,\alpha_X) +  \frac{l^2}{3}\chi(\alpha_N,\alpha_N) h(\ast_h \alpha_X , \ast_h \alpha_X) \\ 
&+\lambda^2 h(\nu_{h},\nu_{h}) =  -6 \lambda^2  -   \frac{2 l^2}{3}\chi(\alpha_N,\alpha_N) h(\alpha_X,\alpha_X) +  \frac{2 l^2}{3}\chi(\alpha_N,\alpha_N) h(\alpha_X , \alpha_X)  +6\lambda^2  = 0\, .
\end{eqnarray*}

\noindent
We verify next that the Einstein equations are satisfied for the specific choices of constants specified in the statement. We have that
\begin{eqnarray*}
&\H_{\lambda, l}\circ \H_{\lambda, l}\vert_{TN\otimes TN} = \lambda^2 \nu_N\circ \nu_N  + \frac{l^2}{9} (\ast_\chi \alpha_N\wedge \alpha_X)\circ (\ast_\chi \alpha_N\wedge \alpha_X)\vert_{TN\otimes TN} \\ 
&+ \frac{l^2}{9} (\alpha_N\wedge \ast_h \alpha_X)\circ  (\alpha_N\wedge \ast_h \alpha_X)\vert_{TN\otimes TN}\, .
\end{eqnarray*}

\noindent
We compute also the following:
\begin{eqnarray*}
& \lambda^2 \nu_N\circ \nu_N = -2\,\lambda^2 \chi\, , \quad (\alpha_N\wedge \ast_h \alpha_X)\circ  (\alpha_N\wedge \ast_h \alpha_X)\vert_{TN\otimes TN} = 18\,\alpha_N\otimes \alpha_N \\
& (\ast_\chi \alpha_N\wedge \alpha_X)\circ (\ast_\chi \alpha_N\wedge \alpha_X)\vert_{TN\otimes TN} = 18 (\alpha_N\otimes \alpha_N - \vert \alpha_N\vert^2_{\chi} \,\chi)\, .
\end{eqnarray*}

\noindent
This implies that 
\begin{equation}
\label{eq:rhsLorentzian}
\frac{1}{4}\H_{\lambda, l}\circ \H_{\lambda, l}\vert_{TN\otimes TN} = -\frac{\lambda^2}{2}\,\chi -  \frac{l^2}{2}\vert \alpha_N\vert^2_{\chi} \,\chi + l^2 \alpha_N\otimes \alpha_N\, .
\end{equation}

\noindent
Likewise we have
\begin{eqnarray*}
&\H_{\lambda, l}\circ \H_{\lambda, l}\vert_{TX\otimes TX} = \lambda^2 \nu_X\circ \nu_X  + \frac{l^2}{9} (\ast_\chi \alpha_N\wedge \alpha_X)\circ (\ast_\chi \alpha_N\wedge \alpha_X)\vert_{TX\otimes TX} \\ 
& +\frac{l^2}{9}(\alpha_N\wedge \ast_h \alpha_X)\circ  (\alpha_N\wedge \ast_h \alpha_X)\vert_{TX\otimes TX}\, ,
\end{eqnarray*}

\noindent
which in turn implies 
\begin{equation}
\label{eq:rhsRiemannian}
\frac{1}{4}\H_{\lambda, l}\circ \H_{\lambda, l}\vert_{TX\otimes TX} = \frac{\lambda^2}{2}\,h + \frac{l^2}{2}\vert \alpha_N\vert^2_{\chi} \, h - l^2 \vert \alpha_N\vert^2_{\chi} \alpha_X\otimes \alpha_X\, .
\end{equation}

\noindent
Finally, it can be checked the mixed components vanish identically,
\begin{equation*}
\H_{\lambda, l}\circ \H_{\lambda, l}\vert_{TN\otimes TX} = \H_{\lambda, l}\circ \H_{\lambda, l}\vert_{TX\otimes TN} = 0\, .
\end{equation*}

\noindent
From Definition \ref{def:Einsteinpcontact}, we obtain that $(N,\chi,\alpha_N,\varepsilon_N) \in \mathrm{PCont}^{\varepsilon \eta}_L(\varepsilon_N,\lambda^2,\kappa_N = l^2)$ and $(M,h,\alpha_X)\in \mathrm{PCont}^{\varepsilon\eta}_R(\lambda^2,\kappa_X = \vert\alpha_N\vert^2\,l^2)$ satisfy, respectively:
\begin{equation}
\label{eq:etaEinsteinproof}
\mathrm{Ric}^{\chi} = - \frac{1}{2}(\lambda^2 + l^2 \varepsilon )\,\chi + l^2 \alpha_N \otimes \alpha_N\, , \qquad \mathrm{Ric}^h = \frac{1}{2}(\lambda^2 + l^2\,\varepsilon )\,h -  l^2\,\varepsilon\, \alpha_X\otimes \alpha_X\, .
\end{equation}

\noindent
Since by definition $\varepsilon = \vert\alpha_N\vert^2_{\chi}$, the right hand sides of the equations appearing in \eqref{eq:etaEinsteinproof} coincide with the right hand sides of equations \eqref{eq:rhsLorentzian} and \eqref{eq:rhsRiemannian}, respectively. Hence, the tuple $(\g,\H_{\lambda, l})$ as defined in the statement of the Theorem satisfies:
\begin{equation*}
\mathrm{Ric}(\nabla^{\H_{\lambda, l}}) = 0\, , \qquad \dd \H_{\lambda, l} = 0 \, , \qquad \dd\ast_\g \H_{\lambda, l} = 0\, , \qquad \vert \mathrm{H}_{\lambda, l}\vert^2_\g = 0\, ,
\end{equation*} 

\noindent
whence it is a solution of 6-dimensional supergravity coupled to a tensor multiplet with constant dilaton.
\end{proof}

\begin{remark}
For ease of reference, we will refer to the solutions $(\g,\H)$ constructed in Theorem \ref{thm:pcontactsugrasolution} as \emph{$\varepsilon\,$-contact supergravity solutions} of type $(\varepsilon_N , \lambda , l)$, where $\varepsilon_{N}\in \left\{ -1, 0 ,1\right\}$ is the norm of the Reeb vector field of the Lorentzian $\varepsilon\,$-contact structure occurring in the given solution.
\end{remark}

\noindent
The holonomy of the Levi-Civita connection $\nabla$ of an $\varepsilon\,$-contact supergravity solution $(\g,\H)$ is clearly reducible, since the sub-bundles $TN\subset TM$ and $TX\subset TM$ are preserved by $\nabla$ by construction. In particular, $\nabla$ is a product connection on $TM = TN\times TX$. However, the connection with torsion $\nabla^{\H}$ is in general not a product connection. In particular neither $TN\subset TM$ nor $TX\subset TM$ are preserved by $\nabla^{\H}$ if $l\neq 0$. Therefore supergravity $\varepsilon\,$-contact solutions are in general not the direct product of a pair of three-dimensional pseudo-Riemannian manifolds with torsion.

Theorem \ref{thm:pcontactsugrasolution} allows to construct large classes of explicit solutions of 6-dimensional supergravity coupled to a tensor multiplet with constant dilaton by exploiting the extensive literature on $\varepsilon\eta\,$-Einstein Riemannian and Lorentzian (para) contact metric three-manifolds, as discussed in Section \ref{sec:Einsteinpcontact}, and by employing new null contact metric structures, as discussed in Sections \ref{sec:nullcontact} and \ref{sec:Einsteinpcontact}. In particular, the previous Theorem implies that the construction of examples and development of classification results on Riemannian and Lorentzian $\varepsilon\eta\,$-Einstein (para) contact metric three-manifolds can automatically be used to construct new Lorentzian six-manifolds equipped with a Ricci-flat metric-compatible connection with totally skew-symmetric, isotropic, closed and co-closed torsion.

%%%%%%%%%%%%%%%%%%%%%%%%%%%%%%%%%%%%%%%%%%%%%%%%%%%%%%%%%%%%%%%
%%%%%%%%%%%%%%%%%%%%%%%%%%%%%%%%%%%%%%%%%%%%%%%%%%%%%%%%%%%%%%%

\section{Ricci flat Lorentzian six-manifolds with closed self-dual torsion}
\label{sec:RicciFlatTorsion}

%%%%%%%%%%%%%%%%%%%%%%%%%%%%%%%%%%%%%%%%%%%%%%%%%%%%%%%%%%%%%%%
%%%%%%%%%%%%%%%%%%%%%%%%%%%%%%%%%%%%%%%%%%%%%%%%%%%%%%%%%%%%%%%

In this section we apply the results of the previous sections to the construction of new six-dimensional Lorentzian manifolds equipped with a Ricci-flat and metric-compatible connection with totally skew-symmetric, isotropic, closed and co-closed torsion, which in turn yields new solutions of minimal supergravity coupled to a tensor multiplet with constant dilaton in six dimensions. Excluding the Ricci-flat case, the simplest scenario where Theorem \ref{thm:pcontactsugrasolution} applies is obtained by taking $l=0$ and $\lambda \neq 0$. In this situation, the corresponding six-dimensional $\varepsilon\,$-contact solution is given by:
\begin{equation*}
\g = \chi \oplus h\, , \qquad \H_{\lambda, l} = \lambda\, (\nu_{\chi}   + \nu_h)\, , \quad \lambda \neq 0\, ,
\end{equation*}

\noindent
on $M = N\times X$, where $\chi$ and $h$ are Einstein with negative and positive Einstein constant, respectively. Assuming that both $(N,\chi)$ and $(X,h)$ are connected, simply connected and geodesically complete we conclude that $(X,h)$ is isometric to the round sphere and $(N,\chi)$ is isometric to $\widetilde{\Sl}(2,\mathbb{R})$ equipped with its Einstein metric. In particular:
\begin{equation*}
(M,g) = (\widetilde{\Sl}(2,\mathbb{R})\times S^3 , \chi\oplus h)\, ,
\end{equation*}

\noindent
is a solution of six-dimensional minimal supergravity, which corresponds with the well-known $\mathrm{AdS}_3\times S^3$ maximally supersymmetric solution of the theory \cite{Chamseddine:2003yy}. For $l\neq 0$, $\varepsilon\,$-contact supergravity solutions are not isomorphic to the previous solution\footnote{Another solution fitting the case $l=0$ is the \emph{non-supersymmetric} embedding of the Reissner-Nordstr\"om black hole presented in \cite{Cano:2019ycn}.}. Hence, intuitively speaking we can think of $\varepsilon\,$-contact supergravity solutions with $l\neq 0$ as being generically non-supersymmetric \emph{geometric and topological deformations} of the supersymmetric $\mathrm{AdS}_3\times S^3$ solution, with deformations parametrized by $l\in \mathbb{R}$, $\alpha_N \in \Omega^1(N)$ and $\alpha_X\in \Omega^1(X)$. In the following we consider $\varepsilon\,$-contact supergravity solutions of type $\varepsilon_N = -1$, $\varepsilon_N = 0$ and $\varepsilon_N = 1$ separately. We emphasise that in general the values of $\lambda^2$ and $\kappa$ do not uniquely determine the diffeomorphism type of $N$ or $X$ and that the same diffeomorphism type may admite several non-isomorphic $\varepsilon\,$-contact supergravity solutions. In this direction, it is a priori possible that there exist, at least when $\varepsilon_N \neq 1$ or in the Sasakian case for $\varepsilon_N=1$, $\varepsilon\,$-contact supergravity solutions for which $M$ is not diffeomorphic to a Lie group.

%%%%%%%%%%%%%%%%%%%%%%%%%%%%%%%%%%%%%%%%%%%%%%%%%%%%%%%%%%%%%%%
%%%%%%%%%%%%%%%%%%%%%%%%%%%%%%%%%%%%%%%%%%%%%%%%%%%%%%%%%%%%%%%

\subsection{Time-like case: $\varepsilon_N = -1$.}

%%%%%%%%%%%%%%%%%%%%%%%%%%%%%%%%%%%%%%%%%%%%%%%%%%%%%%%%%%%%%%%
%%%%%%%%%%%%%%%%%%%%%%%%%%%%%%%%%%%%%%%%%%%%%%%%%%%%%%%%%%%%%%%
 
Let:
\begin{equation*}
(N,\chi,\alpha_N,-1) \in \mathrm{PCont}_L^{\varepsilon \eta} (\lambda^2, l^2, -1)\, , \qquad (X,g,\alpha_X) \in \mathrm{PCont}_R^{\varepsilon \eta}(\lambda^2, - l^2)\, .
\end{equation*}

\noindent
Then, the product of the following pairs of three-manifolds carry $\varepsilon\,$-contact supergravity solutions as prescribed in Theorem \ref{thm:pcontactsugrasolution} for the specified parameters:

\vspace{0.25cm}
 \renewcommand{\arraystretch}{1.5}
 \begin{center}
\begin{tabular}{ |P{2cm}|P{1.5cm}|P{2cm}|P{1.5cm}|P{4cm} | }
				
			 \cline{1-4}
				\multicolumn{2}{|c|}{Lorentzian factor} &  \multicolumn{2}{|c|}{Riemannian factor}    \\
				
				\hline 
			
			$N$ & Sasakian & $X$ &  Sasakian & $(\lambda^2, l^2)$	  \\   \hline \hline
			$\widetilde{\mathrm{Sl}}(2, \mathbb{R})$ & Yes & $\mathrm{SU}(2)$ & Yes & $\lambda^2=1-l^2\, , 1> l^2 \geq 0\, $ \\ \hline
			$\mathfrak{G}_6$ & Yes & $\mathrm{SU}(2)$ & Yes & $\lambda^2=1-l^2\, , 1> l^2 \geq 0\, $ \\ \hline
			$\mathrm{H}_3$	& Yes & $\mathrm{H}_3$			 & Yes & $\lambda^2=0\, , l^2=1\,$	 \\ \hline 
			$\widetilde{\mathrm{Sl}}(2, \mathbb{R})$ & No & $\mathrm{SU}(2)$ & No & $\lambda^2=l^2\, , \frac{1}{2}> l^2 > 0\, $ \\ \hline
			$\widetilde{\mathrm{E}}(1,1)$	& No & $\widetilde{\mathrm{E}}(2)$			 & No & $\lambda^2=0\, , l^2=0\,$	 \\ \hline

			\end{tabular}
			\end{center}

%%%%%%%%%%%%%%%%%%%%%%%%%%%%%%%%%%%%%%%%%%%%%%%%%%%%%%%%%%%%%%%
%%%%%%%%%%%%%%%%%%%%%%%%%%%%%%%%%%%%%%%%%%%%%%%%%%%%%%%%%%%%%%%

\subsection{Space-like case: $\varepsilon_N = 1$.}

%%%%%%%%%%%%%%%%%%%%%%%%%%%%%%%%%%%%%%%%%%%%%%%%%%%%%%%%%%%%%%%
%%%%%%%%%%%%%%%%%%%%%%%%%%%%%%%%%%%%%%%%%%%%%%%%%%%%%%%%%%%%%%%

Let:
\begin{equation*}
(N,\chi,\alpha_N,1) \in \mathrm{PCont}_L^{\varepsilon\eta}(\lambda^2, l^2, 1)\, , \qquad (X,g,\alpha_X) \in \mathrm{PCont}_R^{\varepsilon\eta}(\lambda^2,  l^2)\, .
\end{equation*}

\noindent
The following direct products of three-manifolds can be endowed with $\varepsilon\,$-contact supergravity solutions as indicated in Theorem \ref{thm:pcontactsugrasolution} for the values of the parameters specified below:
%Then, the product of the following pairs of three-manifolds carry $\varepsilon\,$-contact supergravity solutions as prescribed in Theorem \ref{thm:pcontactsugrasolution} for the specified parameters: 

\vspace{0.25cm}
 \renewcommand{\arraystretch}{1.5}
 \begin{center}
\begin{tabular}{ |P{2cm}|P{1.5cm}|P{2cm}|P{1.5cm}|P{4cm} | }
				
			 \cline{1-4}
				\multicolumn{2}{|c|}{Lorentzian factor} &  \multicolumn{2}{|c|}{Riemannian factor}    \\
				
				\hline 
			
			$N$ & Sasakian & $X$ &  Sasakian & $(\lambda^2, l^2)$	  \\   \hline \hline
			$\widetilde{\mathrm{Sl}}(2, \mathbb{R})$ & Yes & $\mathrm{SU}(2)$ & Yes & $\lambda^2=1+l^2\, , l^2 \geq 0\, $ \\ \hline
			$\mathfrak{G}_6$ & Yes & $\mathrm{SU}(2)$ & Yes & $\lambda^2=1+l^2\, , l^2 \geq 0\, $ \\ \hline
		
			$\widetilde{\mathrm{E}}(1,1)$	& No & $\widetilde{\mathrm{E}}(2)$			 & No & $\lambda^2=0\, , l^2=0\,$	 \\ \hline 
			$\widetilde{\mathrm{E}}(2)$	& No & $\widetilde{\mathrm{E}}(2)$			 & No & $\lambda^2=0\, , l^2=0\,$	 \\ \hline

			\end{tabular}
			\end{center}

%\begin{remark}
%On the one hand, by Proposition \ref{prop:ricxien0}, it is clear that those $\varepsilon\,$-contact supergravity solutions constructed with the help of an $\varepsilon \eta\,$-Einstein para-contact structure in $\widetilde{\mathrm{E}}(1,1)$ or $\widetilde{\mathrm{E}}(2)$ can never be Sasakian. On the other hand, it can be shown by direct computation of $\mathfrak{h}$ that it vanishes for every $\varepsilon \eta\,$-Einstein para-contact structure in $\widetilde{\mathrm{Sl}}(2, \mathbb{R})$ and $\mathfrak{g}_6$-type Lie groups. 
%\end{remark}

%%%%%%%%%%%%%%%%%%%%%%%%%%%%%%%%%%%%%%%%%%%%%%%%%%%%%%%%%%%%%%%
%%%%%%%%%%%%%%%%%%%%%%%%%%%%%%%%%%%%%%%%%%%%%%%%%%%%%%%%%%%%%%%

\subsection{Null case: $\varepsilon_N = 0$.}

%%%%%%%%%%%%%%%%%%%%%%%%%%%%%%%%%%%%%%%%%%%%%%%%%%%%%%%%%%%%%%%
%%%%%%%%%%%%%%%%%%%%%%%%%%%%%%%%%%%%%%%%%%%%%%%%%%%%%%%%%%%%%%%

Let:
\begin{equation*}
(N,\chi,\alpha_N,0) \in \mathrm{PCont}_L^{\varepsilon \eta}(\lambda^2, l^2, 0)\, , \qquad (X,g,\alpha_X) \in \mathrm{PCont}_R^{\varepsilon\eta}(\lambda^2, 0)\, .
\end{equation*}

\noindent
The product of the following Lorentzian and Riemannian three-manifolds carry $\varepsilon\,$-contact supergravity solutions as described in Theorem \ref{thm:pcontactsugrasolution} for the values of the parameters indicated below: 
%Then, the product of the following pairs of three-manifolds carry $\varepsilon\,$-contact supergravity solutions as prescribed in Theorem \ref{thm:pcontactsugrasolution} for the specified parameters: 

\vspace{0.25cm}
 \renewcommand{\arraystretch}{1.5}
 \begin{center}
\begin{tabular}{ |P{2cm}|P{1.5cm}|P{2cm}|P{1.5cm}|P{4cm} | }
				
			 \cline{1-4}
				\multicolumn{2}{|c|}{Lorentzian factor} &  \multicolumn{2}{|c|}{Riemannian factor}    \\
				
				\hline 
			
			$N$ & Sasakian & $X$ &  Sasakian & $(\lambda^2, l^2)$	  \\   \hline \hline
			$\widetilde{\mathrm{Sl}}(2, \mathbb{R})$ & Yes & $\mathrm{SU}(2)$ & Yes & $\lambda^2=1\, ,  l^2 \geq 0\, $ \\ \hline
			$\mathfrak{G}_6$ & Yes & $\mathrm{SU}(2)$ & Yes & $\lambda^2=1\, ,  l^2 =  0\, $ \\ \hline
			$\mathfrak{G}_6$& No & $\mathrm{SU}(2)$ & Yes & $\lambda^2=1\, ,  l^2 =  0\, $ \\ \hline
			$\widetilde{\mathrm{E}}(1,1)$	& No & $\widetilde{\mathrm{E}}(2)$			 & No & $\lambda^2=0\, , l^2 \geq 0 \,$	 \\ \hline

			\end{tabular}
			\end{center}
			
			 \renewcommand{\arraystretch}{1}

\begin{remark}
Note that both Sasakian and non-Sasakian $\varepsilon\eta\,$-Einstein null-contact structures on $\mathfrak{G}_6$ can be combined with Riemannian Sasakian structures on $\mathrm{SU}(2)$  to yield $\varepsilon\,$-contact supergravity solutions. This is possible due to the fact that the $\varepsilon\eta\,$-Einstein condition imposes in this case a quadratic constraint on the structure constants. Solutions to this quadratic constraint produce either a Sasakian or a non-Sasakian null contact structure, depending on the particular solution chosen. 
\end{remark}

%%%%%%%%%%%%%%%%%%%%%%%%%%%%%%%%%%%%%%%%%%%%%%%%%%%%%%%%%%%%%%%

\appendix

%%%%%%%%%%%%%%%%%%%%%%%%%%%%%%%%%%%%%%%%%%%%%%%%%%%%%%%%%%%%%%%

%%%%%%%%%%%%%%%%%%%%%%%%%%%%%%%%%%%%%%%%%%%%%%%%%%%%%%%%%%%%%%%

\section{Simply connected three-dimensional Lorentzian Lie groups}
\label{app:Lorentzgroups}

%%%%%%%%%%%%%%%%%%%%%%%%%%%%%%%%%%%%%%%%%%%%%%%%%%%%%%%%%%%%%%%

For the benefit of the reader, we summarize in the following the classification of all three-dimensional Lorentzian simply connected Lie groups. The table below is extracted from \cite[Theorem 4.1]{Calvaruso3d}. In the table below $\widetilde{\mathrm{Sl}}(2,\mathbb{R})$ denotes the universal cover of $\mathrm{Sl}(2,\mathbb{R})$, $\widetilde{\mathrm{E}}(2)$ denotes the universal cover of the group of rigid motions of the Euclidean plane, whereas $\widetilde{\mathrm{E}}(1,1)$ denotes the universal cover of the group of rigid motions of the Minkowski plane and $\H_3$ denotes the three-dimensional Heisenberg group. 

\begin{thm}\cite{Calvaruso3d,CorderoParker,Rahmani} Let $(\G,g)$ be a three-dimensional connected, simply connected, Lorentzian Lie group $\G$ with left-invariant metric $g$. Then, precisely,one of the following cases occurs:
	\begin{itemize}
		\item $\G$ is unimodular and there exists an orthonormal frame $\left\{ e_0 , e_1 , e_2\right\}$, with $e_0$ time-like,  such that the Lie algebra of $\G$ is one of the following:
		
		\
		
		\begin{enumerate} 
			\item $\mathfrak{g}_1$:
			\begin{equation*}
			[e_1 , e_2] = a\, e_1 - b\, e_0\, , \quad [e_1 , e_0] = - a\, e_1 - b\, e_2\, , \quad [e_2 , e_0] = b\,e_1 + a\, e_2 + a\, e_0\, , \quad a \neq 0\, .
			\end{equation*}
			
			\noindent
			In this case $\G\simeq \widetilde{\mathrm{Sl}}(2,\mathbb{R})$ if $b \neq 0$, while $\G \simeq \widetilde{\mathrm{E}}(1,1)$ if $b = 0$.
			
			\item $\mathfrak{g}_2$:
			\begin{equation*}
			[e_1 , e_2] =- c\, e_2 - b\, e_0\, , \quad [e_1 , e_0] = - b\, e_2 + c\, e_0\, , \quad [e_2 , e_0] = a\, e_1\, , \quad c \neq 0\, .
			\end{equation*}
			
			\noindent
			In this case $\G\simeq \widetilde{\mathrm{Sl}}(2,\mathbb{R})$ if $a \neq 0$, while $\G \simeq \widetilde{\mathrm{E}}(1,1)$ if $a = 0$.
			
			\item $\mathfrak{g}_3$:
			\begin{equation*}
			[e_1 , e_2] =  - c\, e_0\, , \quad [e_1 , e_0] = - b\, e_2\, , \quad [e_2 , e_0] = a\, e_1\, ,
			\end{equation*}
			
			\noindent
			In this case the isomorphism type of $\G$ is listed in the following table.
			\vspace{0.05cm}
			\begin{center}

			\
			\renewcommand{\arraystretch}{1.3}
			\begin{tabular}{ |P{3cm}||P{3cm}|P{3cm}|P{3cm}|  }
				\hline
				\multicolumn{4}{|c|}{Simply-connected unimodular groups with Lie algebra $\mathfrak{g}_3$} \\
				\hline
				Lie group $\G$  & $a$ & $b$ & $c$ \\
				\hline
				$\widetilde{\mathrm{Sl}}(2,\mathbb{R})$  & $+$    & $+$ & $+$\\
				$\widetilde{\mathrm{Sl}}(2,\mathbb{R})$ & $+$  & $-$   & $-$\\
				$\SU(2)$ & $+$ & $+$ &  $-$ \\
				$\widetilde{\mathrm{E}}(2)$    & $+$ & $+$ & $0$ \\
				$\widetilde{\mathrm{E}}(2)$  & $+$  & $0$   & $-$\\
				$\widetilde{\mathrm{E}}(1,1)$ & $+$  & $-$   & $0$\\
				$\widetilde{\mathrm{E}}(1,1)$ & $+$  & $0$   & $+$\\
				$\mathrm{H}_3$ & $+$  & $0$   & $0$\\
				$\mathrm{H}_3$ & $0$  & $0$   & $-$\\
				$\mathbb{R}\oplus \mathbb{R}\oplus \mathbb{R}$ & $0$  & $0$   & $0$\\
				\hline
			\end{tabular}
			\renewcommand{\arraystretch}{1}
			\
				\end{center}
			\item $\mathfrak{g}_4$:
			\begin{equation*}
			[e_1 , e_2] = - e_2 + (2 \mu - b) e_0\, , \quad [e_1 , e_0] = - b\, e_2 + e_0\, , \quad [e_2 , e_0] = a\, e_1\, , \quad \mu \in \mathbb{Z}_2\, .
			\end{equation*}
			
			\noindent
			In this case the isomorphism type of $\G$ is listed in the following tables.
			
			\
			\renewcommand{\arraystretch}{1.3}
			\begin{center}
				\begin{tabular}{ |P{3cm}||P{3cm}|P{3cm}|  }
					\hline
					\multicolumn{3}{|c|}{Simply-connected unimodular groups with Lie algebra $\mathfrak{g}_4$ and $\mu =1$} \\
					\hline
					Lie group $\G$  & $a$ & $b$ \\
					\hline
					$\widetilde{\mathrm{Sl}}(2,\mathbb{R})$  & $\neq 0$    & $\neq 1$ \\
					$\widetilde{\mathrm{E}}(1,1)$ & $0$  & $\neq 1$ \\
					$\widetilde{\mathrm{E}}(1,1)$ & $< 0$  & $1$ \\
					$\widetilde{\mathrm{E}}(2)$ & $> 0$  & $1$ \\
					$\mathrm{H}_3$ & $0$  & $1$ \\
					\hline
				\end{tabular}
				
				\
				\
				
				\begin{tabular}{ |P{3cm}||P{3cm}|P{3cm}|  }
					\hline
					\multicolumn{3}{|c|}{Simply-connected unimodular groups with Lie algebra $\mathfrak{g}_4$ and $\mu = -1$} \\
					\hline
					Lie group $\G$  & $a$ & $b$ \\
					\hline 
					$\widetilde{\mathrm{Sl}}(2,\mathbb{R})$  & $\neq 0$    & $\neq -1$ \\
					$\widetilde{\mathrm{E}}(1,1)$ & $0$  & $\neq -1$ \\
					$\widetilde{\mathrm{E}}(1,1)$ & $> 0$  & $-1$ \\
					$\widetilde{\mathrm{E}}(2)$ & $< 0$  & $-1$ \\
					$\mathrm{H}_3$ & $0$  & $-1$ \\
					\hline 
				\end{tabular}
			\end{center}	
		\end{enumerate}
		\renewcommand{\arraystretch}{1}
		\
		
		\item $\G$ is non-unimodular and there exists an orthonormal frame $\left\{ e_0 , e_1 , e_2\right\}$, with $e_0$ time-like,  such that the Lie algebra of $\G$ is one of the following:
		
		\begin{enumerate}
			\item $\mathfrak{g}_5$:
			\begin{eqnarray*}
			& [e_1 , e_2] = 0\, , \quad [e_1 , e_0] = a\, e_1 + b\, e_2\, , \quad [e_2 , e_0] = c\,e_1 + d\, e_2\, , \\
			& a + d \neq 0\, ,  \quad a\,c + b\,d = 0\, .
			\end{eqnarray*}
			
			\noindent
			%In this case $\G\simeq \widetilde{\mathrm{Sl}}(2,\mathbb{R})$ if $b \neq 0$, while $\G \simeq \widetilde{\mathrm{E}}(1,1)$ if $b = 0$.
We denote the unique connected and simply connected Lie group with Lie algebra $\mathfrak{g}_5$ as $\mathfrak{G}_5$. 					
			\item $\mathfrak{g}_6$:
			\begin{eqnarray*}
			&[e_1 , e_2] = a\, e_2 + b\, e_0\, , \quad [e_1 , e_0] = c\, e_2 + d\, e_0\, , \quad [e_2 , e_0] = 0\, , \\
			&a + d \neq 0\, ,  \quad a\,c - b\,d = 0\, .
			\end{eqnarray*}
			
			%\noindent
			%In this case $\G\simeq \widetilde{\mathrm{Sl}}(2,\mathbb{R})$ if $a \neq 0$, while $\G \simeq \widetilde{\mathrm{E}}(1,1)$ if $a = 0$.
We denote the unique connected and simply connected Lie group with Lie algebra $\mathfrak{g}_6$ as $\mathfrak{G}_6$. 			
			
			\item $\mathfrak{g}_7$:
			\begin{eqnarray*}
			& [e_1 , e_2] = - a\, e_1 - b\, e_2 - b\, e_0\, , \quad [e_1 , e_0] = a\, e_1 + b\, e_2 + b\, e_0\, , \quad [e_2 , e_0] = c\, e_1 + d\, e_2 + d\, e_0\, , \\ 
			& a + d \neq 0\, , \quad a c = 0\, .
			\end{eqnarray*}
			
			We denote the unique connected and simply connected Lie group with Lie algebra $\mathfrak{g}_7$ as $\mathfrak{G}_7$. 			 			
		\end{enumerate}	
	\end{itemize}
\label{thm:LorentzGroups}
\end{thm}

\noindent 
Non-unimodular Lie algebras can be interpreted as deformations of the unimodular ones, with deformation parameter given by:
\begin{equation*}
\beta\eqdef a+d\, .
\end{equation*}

\noindent
This is due to the fact that in the limit $\beta\to 0$ all the previous non-unimodular Lie algebras reduce to unimodular Lie algebras \cite{Milnor}. Consider as an example the Lie algebra $\mathfrak{g}_6$ in the limit $\beta\to 0$. In such limit the Lie brackets of $\mathfrak{g}_6$ reduce to:
\begin{eqnarray*}
&[e_1 , e_2] = a\, e_2 + b\, e_0\, , \quad [e_1 , e_0] = c\, e_2 -a\, e_0\, , \quad [e_2 , e_0] = 0\, , \\
& a\,(c + b) = 0\, .
\end{eqnarray*}
			
\noindent
The specific unimodular algebras occurring in the limit $\beta\to 0$ of $\mathfrak{g}_6$ can be easily identified in some particular cases:
			
\begin{itemize}
\item If $a=0$, we obtain:
\begin{equation*}
[e_1 , e_2] =  b\, e_0\, , \quad [e_1 , e_0] = c\, e_2 \, , \quad [e_2 , e_0] = 0\, . 
\end{equation*}

\noindent
Comparing with the classification of unimodular Lie algebras, we conclude that the case $cb>0$ corresponds to the Lie algebra of $\widetilde{\mathrm{E}}(1,1)$, the case $cb<0$ corresponds to the Lie algebra of $\widetilde{\mathrm{E}}(2)$, the case $c\neq 0\, , b=0$ or $c=0\, , b\neq 0$ corresponds to the Lie algebra of $\mathrm{H}_3$ and the case $c=b=0$ corresponds to $\mathbb{R}^3$.

\item If $a \neq 0$ then we must have $c=-b$. If in addition we set $c=0$, the brackets take the form:
\begin{equation*}
[e_1 , e_2] = a\, e_2 \, , \quad [e_1 , e_0] = -a\, e_0\, , \quad [e_2 , e_0] = 0\, .
\end{equation*}

\noindent
Comparing with the classification of unimodular Lie algebras, we conclude that this algebra is isomorphic to $\widetilde{\mathrm{E}}(1,1)$.

\item Also, if $a = b = 1$ and $c=-b=-1$ the brackets read:
\begin{equation*}
[e_1 , e_2] = e_0 + e_2 \, , \quad [e_1 , e_0] = - e_0 - e_2\, , \quad [e_2 , e_0] = 0\, .
\end{equation*}

\noindent
which defines a Lie algebra isomorphic to $\mathrm{H}_3$.

\end{itemize} 

\noindent 
Therefore, the non-unimodular Lie group $\mathfrak{G}_6$ can be understood as a deformation of various unimodular groups such as $\widetilde{\mathrm{E}}(1,1)$, $\widetilde{\mathrm{E}}(2)$ or $\mathrm{H}_3$. Similar remarks hold for $\mathfrak{G}_5$ and $\mathfrak{G}_7$.

%%%%%%%%%%%%%%%%%%%%%%%%%%%%%%%%%%%%%%%%%%%%%%%%%%%%%%%%%%%%%%%

\section{Curvature of left-invariant metrics on Lorentzian Lie groups}
\label{app:CurvatureLorentzgroups}

Let $(\G,g)$ be a three-dimensional Lorentzian Lie group with Lie algebra $\mathfrak{g}$. Fix a global left-invariant frame $\{e_0,e_1,e_2\}$ on $\G$. We have: 
\begin{equation*}
[e_0,e_1]=ae_0+be_1+ce_2\, , \quad [e_1,e_2]=de_0+fe_1+he_2 \, \quad [e_0,e_2]=ge_0+je_1+ke_2 \, ,
\end{equation*}

\noindent
where $a,b,c,d,f,h,g,j,k \in \mathbb{R}$ such that the Jacobi identity is satisfied. In this case, imposing $[[e_0,e_1],e_2]+[[e_2,e_0],e_1]+[[e_1,e_2],e_0]=0$ yields the constraints
\begin{equation*}
bd+kd-fa-hg=0 \, , \quad ja-gb+kf-hj=0\, , \quad ak+hb-gc-fc=0\, .
\end{equation*}
 Using Koszul formula, we compute the following covariant derivatives:
\begin{equation*}
\begin{matrix}
\nabla_{e_0} e_0=ae_1+ge_2\, ,  & \nabla_{e_0} e_1=ae_0+\frac{(c-j+d)}{2}e_2\, , & \nabla_{e_0}e_2=ge_0+\frac{(j-c-d)}{2}e_1\, , \\ \\
\nabla_{e_1} e_0=-be_1+\frac{(-c+d-j)}{2}e_2\, , & \nabla_{e_1} e_1=-be_0-fe_2\, , & \nabla_{e_1} e_2=\frac{(d - c - j)}{2}e_0+fe_1\, , \\ \\
\nabla_{e_2} e_0=-\frac{(j+d+c)}{2}e_1-ke_2\, , & \nabla_{e_2} e_1=-\frac{(j+d+c)}{2}e_0-h e_2\, , & \nabla_{e_2} e_2=-ke_0+he_1\, . \\ \\
\end{matrix}
\end{equation*}
Using the previous covariant derivatives we compute the Riemann curvature tensor (with the convention $\mathrm{R}^g(u,v)w=\nabla_u \nabla_v w-\nabla_v \nabla_u w-\nabla_{[u,v]} w$):
\begin{equation*}
\begin{split}
&\mathrm{R}^g(e_0,e_1)e_0=\bigg [-a^2+b^2-gf+\frac{1}{4}(-c+d-j)(j-c-d)+\frac{c}{2}(j+d+c)  \bigg] e_1\\&+ [-bd+af-ag+ck+bj]e_2\, , \\
&\mathrm{R}^g(e_0,e_1)e_1=\bigg[ -a^2+b^2-gf+\frac{1}{4}(-c+d-j)(j-c-d)+\frac{c}{2}(j+d+c)  \bigg] e_0\\& +[-bg+bf-ad+hc+aj]e_2\, , \\
&\mathrm{R}^g(e_0,e_1)e_2=[fa-ag+bj-bd+ck]e_0+[ad-aj+gb-bf-ch]e_1\, , \\
&\mathrm{R}^g(e_1,e_2)e_0=[-kf+fb-da+hc+hj]e_1+[-bh+hk-dg+fj+fc]e_2 \, , \\
&\mathrm{R}^g(e_1,e_2)e_1=[hc+hj-fk-da+fb]e_0\\&+\bigg [-bk+f^2+h^2-\frac{d}{2}(c-j+d)-\frac{1}{4}(d+j+c)(-c+d-j) \bigg] e_2 \,, \\
&\mathrm{R}^g(e_1,e_2)e_2=[-bh+fj+fc-dg+hk]e_0\\&+\bigg[[bk-f^2-h^2+\frac{d}{2}(c-j+d)+\frac{1}{4}(d+j+c)(-c+d-j) \bigg] e_1 \,, \\
&\mathrm{R}^g(e_0,e_2)e_0=[kc+kd-gh-ga+jb]e_1\\&+\bigg[ah-g^2+k^2-\frac{j}{2}(-c+d-j)-\frac{1}{4}(j+d+c)(c-j+d) \bigg]e_2 \, , \\
&\mathrm{R}^g(e_0,e_2)e_1=[-hg-ga+jb+kc+kd]e_0+[ak+jf+kh-gc-gd]e_2 \, , \\
&\mathrm{R}^g(e_0,e_2)e_2=\bigg [ \frac{1}{4}(j-c-d)(d+j+c)-g^2+k^2+\frac{j}{2}(-d+c+j)+ah \bigg] e_0 \\
&+[-ka+gd+gc-fj-kh]e_1 \, , \\
\end{split}
\end{equation*}
From here, we obtain the Ricci curvature tensor $\ric^g$ defined as $\ric^g(u,v)=\text{Tr}\left ( w \rightarrow \mathrm{R}^g(w,u)v \right )$. 
\begin{equation*}
\begin{split}
\ric^g(e_0,e_0)&=a^2-b^2+gf+g^2-k^2-ah+\frac{d^2}{2}-cj-\frac{c^2}{2}-\frac{j^2}{2}\, , \\
\ric^g (e_0,e_1)&=bh-fj-fc+dg-hk \, , \\
\ric^g(e_0,e_2)&=hc+hj-fk-da+fb\, , \\
\ric^g(e_1,e_1)&=-a^2+b^2-f^2-h^2-fg+bk-\frac{j^2}{2}+dc+\frac{c^2}{2}+\frac{d^2}{2}\, , \\
\ric^g(e_1,e_2)&=fa-ag+bj-bd+ck \, , \\
\ric^g(e_2,e_2)&=-g^2+k^2+ah+kb-f^2-h^2+\frac{j^2}{2}-jd+\frac{d^2}{2}-\frac{c^2}{2}\, .
\end{split}
\end{equation*}
Finally, the scalar curvature $\text{Scal}^g$ reads 
\begin{equation*}
\text{Scal}^g=-2a^2+2b^2-2gf-2g^2+2k^2+2ah+\frac{d^2}{2}+cj+\frac{c^2}{2}+\frac{j^2}{2}-2f^2-2h^2 +2bk+dc-jd\,.
\end{equation*}
%%%%%%%%%%%%%%%%%%%%%%%%%%%%%%%%%%%%%%%%%%%%%%%%%%%%%%%%%%%%%%%

%%%%%%%%%%%%%%%%%%%%%%%%%%%%%%%%%%%%%%%%%%%%%%%%%%%%%%%%%%%%%%%
%%%%%%%%%%%%%%%%%%%%%%%%%%%%%%%%%%%%%%%%%%%%%%%%%%%%%%%%%%%%%%%

\end{document}